\documentclass[oneside,10pt,reqno]{amsart}
\usepackage{amssymb,amsmath,amsthm,bbm,enumerate,multirow,hyperref,stmaryrd,mathrsfs,mathabx,amsaddr}
\usepackage[shortlabels]{enumitem}
\usepackage{scalerel,stackengine}

\stackMath
\newcommand\reallywidehat[1]{%
\savestack{\tmpbox}{\stretchto{%
  \scaleto{%
    \scalerel*[\widthof{\ensuremath{#1}}]{\kern-.6pt\bigwedge\kern-.6pt}%
    {\rule[-\textheight/2]{1ex}{\textheight}}
  }{\textheight}%
}{0.5ex}}%
\stackon[1pt]{#1}{\tmpbox}%
}

\makeatletter
\@namedef{subjclassname@2020}{%
  \textup{2020} Mathematics Subject Classification}
\makeatother

\allowdisplaybreaks
\theoremstyle{definition}
\newtheorem{definition}{Definition}
\theoremstyle{theorem}
\newtheorem{proposition}[definition]{Proposition}
\newtheorem{lemma}[definition]{Lemma}
\newtheorem{theorem}[definition]{Theorem}
\newtheorem{corollary}[definition]{Corollary}
\numberwithin{equation}{section}
\numberwithin{definition}{section}
\theoremstyle{remark}
\newtheorem{remark}[definition]{Remark}
\newtheorem{question}[definition]{Question}
\newtheorem{example}[definition]{Example}
\def\PP{\mathbb{P}}
\def\HH{\mathcal{H}}
\def\lll{\mathfrak{l}}
\def\L2{\mathrm{L}^2}
\def\pr{\mathrm{pr}}
\def\QQ{\mathbb{Q}}
\def\BB{\mathcal{B}}

\def\AA{\mathcal{A}}
\def\UU{\mathcal{U}}
\def\RR{\mathbb{R}}
\def\GG{\mathcal{G}}
\def\FF{\mathcal{F}}
\def\dd{\mathrm{d}}
\def\PPP{\mathsf{P}}
\def\WW{\mathsf{W}}

\begin{document}
\title{Stationary local  random countable sets over the Wiener noise}

\author{Matija Vidmar}
\address{Department of Mathematics, Faculty of Mathematics and Physics, University of Ljubljana}
\email{matija.vidmar@fmf.uni-lj.si}
\thanks{Most of this research was done while MV was in receipt of a Fernandes Fellowship at the University of Warwick. MV also  acknowledges support from the  Slovenian Research Agency  under programme No. P1-0402.}
\author{Jon Warren}
\address{Department of Statistics, University of Warwick}
\email{j.warren@warwick.ac.uk}

\begin{abstract}
The times of Brownian local minima, maxima and their union are three distinct examples of local, stationary,  dense,  random countable sets associated with   classical Wiener noise.  Being local means, roughly, determined by the local behavior of the sample paths of the Brownian motion, and  stationary means invariant relative to the L\'evy shifts of the sample paths. We answer to the affirmative Tsirelson's question, whether or not there are any others, and  develop some  general theory for such sets. An extra ingredient to their structure, that  of an honest indexation, leads to a splitting result that is akin to the Wiener-Hopf factorization of the Brownian motion at the minimum (or maximum) and has the latter as a special case. Sets admitting an honest indexation are moreover shown to have the property that no stopping time belongs to them with positive probability. They  are  also  minimal:  they do not have any non-empty proper local stationary subsets. Random sets, of the kind studied in this paper, honestly indexed or otherwise, give rise to nonclassical one-dimensional noises, generalizing the  noise of splitting. Some properties of these noises  and the inter-relations between them are investigated. In particular, subsets are connected to subnoises. 
\end{abstract}



\date{\today}

\maketitle


\section{Introduction}

 \subsection{Motivation and focus}\label{subsection:motivation}
 Random, dense, countable  sets, like the set of times at which a two-sided Brownian path attains a local minimum, are usually considered to be pathological objects \cite{kendall_2000,tsirelson-rc}. There is no good measurable structure on the space of such sets, meaning that there are no interesting events to be described concerning such a random set: ``Could
you imagine a function of the set of all Brownian local minimizers that gives a
non-degenerate random variable?'' \cite[p. 189]{tsirelson-nonclassical}. 

One way around the measurability issue  is to use some  enumeration  of the points belonging to the random set in question, by a countable family of random variables.  For the case of (here and always: the times of) the Brownian local minima we may indeed set, for each Brownian path $\omega : \mathbb{R}\to \mathbb{R}$, and for rational $p<q$,
$X_{p,q}(\omega):= \text{argmin}\{ \omega(t): t \in [p,q]\}$. With probability one with respect to Wiener measure,  $X_{p,q}(\omega)$ is well-defined for every pair  $(p,q)$, and a.s. the set $M(\omega):=\{X_{p,q}(\omega): (p,q) \in {\mathbb Q}^2,  p<q\}$  is precisely the  set of times at which $\omega$ attains a local minimum. Note the $X_{p,q}(\omega)$ do not take distinct values, and we do not require this of our enumeration.  The countable dense set $M$  is {\em local} in that for all real $s<t$ we have that $M\cap (s,t)$  may be enumerated by a  sequence of random variables  each of which is measurable with respect to the increments of the Brownian path between times $s$ and $t$. $M$ is also \emph{stationary} in that for any real $h$, for almost every $\omega$ with respect to Wiener measure, the set $M(\omega)-h$ is almost surely equal to  $M(\Delta_h\omega)$,  where  $\Delta_h(\omega):=\omega(h+\cdot)-\omega(h)$ is the L\'evy shifted path. 

Brownian local maxima give  another example of a random, dense, countable set ``of the Wiener noise\footnote{The designation ``Wiener noise'' will be explained in technical terms in Section~\ref{section:noises-out-of-random-sets}. It will not be relevant for us to  do so before then.}'' which is local and stationary. A third example is obtained by  taking the union of these two. Tsirelson asked \cite[Question~2e3]{tsirelson-nonclassical} (and again in \cite[Remark~9.11]{tsirelson-rc}), are there any others. By giving further examples, we  show that the answer to this question is to the affirmative and that, moreover, there are at least a continuum many of them, all pairwise a.s. disjoint. 

In the light of our additional examples, we see that such sets are not as rare as might have been originally thought, and  we wish to explore  their general properties in more depth. 
Random dense countable subsets have been treated in the literature previously,  namely in \cite{kendall_2000,tsirelson-rc}, but in these works the focus is on intrinsic properties of the sets, whereas we wish to study a coupling between a random set and a Brownian motion. Thus, the Brownian maxima, minima and extrema are different objects for us, whereas according to the definition and main result of \cite{tsirelson-rc} they are  all equal in distribution; evidently so for the maxima and minima, more surprisingly for the extrema.

Returning to Tsirelson's question, the initial motivation for the introduction of stationary local random countable sets  was to obtain new one-dimensional noises in the sense of \cite{tsirelson-nonclassical,picard2004lectures}. One such (nonclassical) noise is got by attaching independent equiprobable random signs to each Brownian local minimum, which leads to Warren's noise of splitting \cite{warren-splitting}. ``New examples [of stationary local random countable sets] could lead to new noises'' \cite[just after Question~2e3]{tsirelson-nonclassical}. A similar construction applies generally and we explore these noises in relative detail as well.

\subsection{Article structure, highlights and roadmap to  results}
 Section~\ref{section:notation-preliminaries} contains some  definitions and notation related to two-sided Brownian motion.  Then, in Section~\ref{section:rcs-l-s} we begin by formally  defining  random countable sets with  various properties.  We observe   that 
 a  stationary, local,  random countable set must be either  dense  or empty (Proposition~\ref{proposition:dense-or-empty}) and then
  apply the  main result of Tsirelson, \cite{tsirelson-rc},   to establish that
  \begin{enumerate}
\item\label{initial:1} all dense such sets are equal in law to the realization of the range of an i.i.d. sequence of random variables whose law is equivalent to Lebesgue measure (Proposition~\ref{proposition:equal-in-law});
\item\label{initial:2} any event concerning such a set has probability zero or one (Proposition~\ref{proposition:final-is-trivial}).
\end{enumerate}
This makes plain  the pathological character of stationary local random countable sets alluded to above. In particular it renders useless any approach to   random countable sets based on hitting probabilities and intensity measures \cite{herriger}. 

On the ``positive'' side we establish in Proposition~\ref{proposition:version-perfect-stationary} the important technical fact that any stationary random countable set $M$ admits a version that is perfectly stationary in the sense that $M=h+M(\Delta_h)$ for all $h\in \mathbb{R}$. It allows us to deduce that any stationary local random countable set is a.s. equal to the visiting set $\{t\in \mathbb{R}:\Delta_t\in A\}$  by the process of the L\'evy shifts $(\Delta_t)_{t\in \mathbb{R}}$ of some measurable event $A$ of path space, which belongs to the germ $\sigma$-field around zero (the $A$, naturally, depends on $M$). For instance, in the case of the Brownian local minima $A=\{0\text{ is a (time of) local minimum}\}$. Theorem~\ref{theorem:construction-through-zero-time-section} has the precise statement. However we have not been able to identify any good criteria for determining whether  a given set $A$ belonging to the germ $\sigma$-field generates a non-empty countable  random set; this remains a natural open question.

In Section~\ref{section:new-family-of-examples}, which can be read largely independently of Section~\ref{section:rcs-l-s}, we settle Tsirelson's question, referred to in Subsection~\ref{subsection:motivation}, by providing a family $(M^{(d)})_{d\in (0,2)}$ of stationary local random countable sets over the Wiener noise satisfying $M^{(d_1)}\cap M^{(d_2)}=\emptyset$ a.s. for $d_1\ne d_2$ from $(0,2)$ (Propositions~\ref{proposition:gamma} and~\ref{proposition:a.s.distinct}). Specifically, for each $d\in (0,2)$,  $M^{(d)}$ comes from collecting the last zeros of  squared Bessel processes $Z^{(d)}$ of dimension $d$, started at every rational point $p$ of the real line and driven by the post-$p$ increments of the   underlying Brownian motion. The case $d=1$ corresponds to $Z^{(1)}$ being the square of the Brownian motion reflected in its running infimum  and gives for $M^{(d)}$ the local minima of the Brownian motion.

The examples of Section~\ref{section:new-family-of-examples} suggest the study  of an extra property (which all the $M^{(d)}$, $d\in (0,2)$, share) of a random set $M$  admitting what we call an honest indexation, extending   the notion  of an honest time, see  \cite[Chapter~XX]{dellacherie-sets}.  Namely, an honest indexation for a random set $M$ consists of a family of random variables $(\tau_{s,t})$ indexed by real numbers $s<t$ with the following properties: (i) the $\tau_{s,t}$, $s<t$ rational, precisely exhaust $M$ a.s.; (ii) each $\tau_{s,t}$ takes values in $(s,t)$ and is measurable relative to the increments of the Brownian motion between times $s$ and $t$; (iii) $\tau_{h+s,h+t}=h+\tau_{s,t}(\Delta_h)$ a.s. for all real $h$ and $s<t$; and  (iv) $\tau_{s,t}=\tau_{u,v}$ a.s. on $\{\tau_{s,t}\in (u,v)\}$ for all real $s\leq u<v\leq t$. For random sets admitting an honest indexation we can say much more about their general structure. We explore this in Section~\ref{section:honest-indexations}. Two significant results, described only informally here, are 
\begin{enumerate}[(1)]
\item a splitting into independent pieces at an exponentially sampled honest indexator, generalization of the Wiener-Hopf factorization at the minimum before an independent exponential random time in the case of the local minima (Theorem~\ref{theorem:splitting});
\item the fact that honestly indexed sets a.s. do not meet the graph of any stopping time (Theorem~\ref{theorem:stopping-times-no}). 
\end{enumerate}
 The first of these properties leads to the establishment of the fact that honestly indexed random sets are minimal in that they do not contain any proper dense stationary local random countable subset. Minimality is the subject of Section~\ref{section:minimality} and the preceding result the reader will  find in  Corollary~\ref{corollary:minimal}. A simple consequence of Corollary~\ref{corollary:minimal}  is that the local extrema, unlike the local minima and the local maxima, cannot be honestly indexed.

Finally, in Section~\ref{section:noises-out-of-random-sets} we turn to the theory of noises. Every dense local stationary random countable set $M$ engenders a non-classical one-dimensional noise $N^M$. Loosely speaking, $N^M$ consists of attaching to each point in $M$ an independent equiprobable random sign. The classical (stable) part and the first superchaos of $N^M$ are identified (Proposition~\ref{lemma:identifications-stable-super}). In Theorem~\ref{thm:not-iso-noises}  we characterize when two such noises are isomorphic and are able to conclude that the noises attached to the $M^{(d)}$, $d\in [1,2)$, are pairwise non-isomorphic. Subnoises of an $N^M$ are shown to correspond to stationary local random countable subsets of $M$  (Theorem~\ref{thm:subnoises}). Together with the results on minimality it entails that the noises attached to the $M^{(d)}$, $d\in (0,2)$, have for their only proper non-void subnoise the classical Wiener noise. In particular this is true of Warren's noise of splitting mentioned above.

\subsection{Miscellaneous general notation}
We agree to denote by $\AA/\BB$  the set of $\AA/\BB$-measurable maps. Also, to write $\mathsf{Q}[\ldots]$ for the expectation $\mathbb{E}_\mathsf{Q}[\ldots]$ and $Z_\star\mathsf{Q}$ for  the law of a random element $Z$ under a probability $\mathsf{Q}$ relative to a $\sigma$-field on the codomain of $Z$ that shall be mentioned explicitly or will be understood from context. More generally, $\mu[g]$, resp. $g\cdot\mu=(A\mapsto \mu[g;A])$, will denote the definite (resp. indefinite) integral of a measurable numerical $g$ against a measure $\mu$. $\uparrow$ (resp. $\downarrow$) means nondecreasing (resp. nonincreasing). $(2^X)_{\mathrm{fin}}$ are the finite subsets of a set $X$ and $[n]:=\{1,\ldots,n\}$ ($=\emptyset$ when $n=0$) for $n\in \mathbb{N}_0$. A process $Z$ stopped at a random time $S$ is written $Z^S$, as usual. For a topological space $X$, $\mathcal{B}_X$ is its Borel $\sigma$-field. Lastly, given a map $g$ into a measurable space $(E,\mathcal{E})$ we write $g^{-1}(\mathcal{E}):=\{g^{-1}(F):F\in \mathcal{E}\}$ or just $\sigma(g):=g^{-1}(\mathcal{E})$ ($\mathcal{E}$ being then understood from context) for the pull-back of $\mathcal{E}$ along $g$ (the generated $\sigma$-field). 

\section{Preliminaries}\label{section:notation-preliminaries}

The base sample space shall be $\Omega_0:=\{\omega\in \mathbb{R}^\mathbb{R}:\omega(0)=0\text{ and }\omega\text{ is continuous}\}$ endowed with the completed (two-sided) Wiener measure $\WW$. Thus the domain of $\WW$, which we denote $\GG$, is the completion under Wiener measure of the Borel $\sigma$-field $\mathcal{B}_{\Omega_0}$ on $\Omega_0$ for the locally uniform topology. We may recall that  $\mathcal{B}_{\Omega_0}$ is also generated by the family  $B=(B_t)_{t\in \mathbb{R}}$ of the canonical projections on $\Omega_0$. To dispel any lingering doubts as to the meaning of ``two-sided Wiener measure'': $B\vert_{[0,\infty)}$ and $B_{-\cdot}\vert_{[0,\infty)}$ are independent standard Brownian motions under $\WW$, which specifies $\WW$ uniquely. For extended-real $s<t$, $\FF_{s,t}$ will denote the sigma-field generated by the increments of $B$ between the times $s$ and $t$ \emph{and} by $\WW^{-1}(\{0,1\})$, the $\WW$-trivial sets. The family $\FF=(\FF_{s,t})_{(s,t)\in [-\infty,\infty]^2,s< t}$ is the so-called (noise-)factorization of the Wiener process $B$. Two ``one-sided'' subfamilies thereof, $\FF^{\rightarrow}:=(\FF_{-\infty,t})_{t\in [0,\infty)}$ and $\FF^{0,\rightarrow}:=(\FF_{0,t})_{t\in [0,\infty)}$ will  be used quite often. The symbol $\mathfrak{l}$  shall denote  Lebesgue measure on the Borel sets of $\mathbb{R}$. For extended-real $a\leq 0\leq b$, $a\ne b$, we write $\Omega_0\vert_{(a,b)}:=\{\omega\vert_{(a,b)}:\omega\in  \Omega_0\}$ endowed with the $\sigma$-field of evaluation maps that we shall designate $\mathcal{B}_{\Omega_0\vert_{(a,b)}}$; similarly for $\Omega_0\vert_{[a,b)}$ and $\Omega_0\vert_{(a,b]}$.

 On occasion --- especially when ``taming the continuum'' using stopping times --- it will  serve us well to work with the one-sided usual Wiener space ${\Theta_0}:=\{\omega\in \mathbb{R}^{[0,\infty)}:\omega\text{ continuous and }\omega(0)=0\}$, endowed with the completion $\HH$ of its Borel $\sigma$-field $\mathcal{B}_{\Theta_0}$ under the completed Wiener measure $\PPP$, canonical process $C=(C_t)_{t\in [0,\infty)}$, and $\PPP$-completed natural filtration $\mathcal{U}=(\mathcal{U}_t)_{t\in[0,\infty)}$ of $C$. Then the symbol $\mathscr{L}$ will  be used to denote  Lebesgue measure on the Borel sets of  $[0,\infty)$. However, by and large, the two-sided setting shall be  more natural for us to work in. 
 
 For the reader's convenience of recollection of notation later on, let us display the more significant objects pertaining to one and to the other landscape described above, side by side: 
 $$((\Omega_0,\GG,\WW);B,\FF)\text{ and }\mathfrak{l}\quad \text{vis-\`a-vis}\quad ((\Theta_0,\HH,\PPP);C,\UU)\text{ and }\mathscr{L}.$$

Some further concepts that we shall find useful throughout follow.

\begin{definition}
For $u\in \mathbb{R}$, the map $\Delta_u:\Omega_0\to\Omega_0$ is the L\'evy shift: $\Delta_u(\omega)(t):=\omega(u+t)-\omega(u)$ for $t\in \RR$, $\omega\in \Omega_0$. We write $\Delta:=(\Delta_u)_{u\in \mathbb{R}}$ for short. When $u\in [0,\infty)$ by an abuse of notation we shall use the same symbol $\Delta_u$ for the corresponding map on $\Theta_0$:  $\Delta_u(\omega)(t):=\omega(u+t)-\omega(u)$ for $t\in [0,\infty)$, $\omega\in \Theta_0$.
\end{definition}
Since $B$ is just the identity on $\Omega_0$, $\Delta_u=\Delta_uB$ (similarly, $\Delta_u=\Delta_uC$ when $u\geq 0$) and we will be quite liberal as to which of the two we shall find more convenient to use.
\begin{remark}\label{rmk:shifts-preserving}
The L\'evy shifts are measure-preserving for $\WW$ by stationary  independent increments of $B$ and the fact that $B_0=0$ (the same for $\PPP$, $C$ in lieu of $\WW$, $B$). Note also the identities $\Delta_u\circ \Delta_v=\Delta_{u+v}$ for real $u$ and $v$ and $\Delta_0=\mathrm{id}_{\Omega_0}=B$ (the same with just nonnegative $u$, $v$ and $\Delta_0=\mathrm{id}_{\Theta_0}=C$ for the space $\Theta_0$). Thus we have an indexed group $(\Delta_u)_{u\in\mathbb{R}}$ of $\WW$-preserving bi-measurable bijections (just a semigroup of $\PPP$-preserving measurable maps in the case of $\Theta_0$).
\end{remark}

\begin{definition}\label{def:invariance-closed-shift}
A subset $A\subset \Omega_0$ is shift-invariant if $\Delta_v^{-1}(A)=A$ for all $v\in\mathbb{R}$. A subset $A\subset \Theta_0$ is shift-closed if  $A\subset \Delta_v^{-1}(A)$ for all $v\in [0,\infty)$.
\end{definition}

To get some feeling for the concepts of Definition~\ref{def:invariance-closed-shift} we give a couple of examples and make some elementary observations.
\begin{example}
The event $\{\lim_{t\to \infty} B_t=\infty\}$ is shift-invariant and $\WW$-almost certain.
\end{example}
\begin{example}
The event $E:=\{0\text{ is a local minimum}\}$ is $\WW$-negligible and not shift-invariant, in fact $\cup_{h\in \mathbb{R}}\Delta_h^{-1}(E)$ is equal to the set of paths each of which has at least one local minimum and this set has full $\WW$-measure. For $F:=\{B\vert_{[0,\epsilon)}\leq 0 \text{ or }B\vert_{(-\epsilon,0]}\leq 0\text{ for some $\epsilon>0$}\}$, which is still $\WW$-negligible we even have $\cup_{h\in \mathbb{R}}\Delta_h^{-1}(F)=\Omega_0$.
\end{example}

\begin{enumerate}[(1)]
\item\label{shifts-varia:i} Suppose $A\subset \Omega_0$ is closed for the L\'evy shifts $\Delta_t$, $t\in \mathbb{R}$, i.e. suppose that $A\subset \Delta_t^{-1}(A)$ for all $t\in \mathbb{R}$. Let $t\in \mathbb{R}$. Then $A\subset \Delta_{t}^{-1}(A)$ but also $A\subset \Delta_{-t}^{-1}(A)=\Delta_t(A)$, hence $\Delta_t^{-1}(A)\subset A$ and thus $A=\Delta_t^{-1}(A)$. Therefore $A$ is shift-invariant.
\item If $A\subset \Theta_0$ is shift-closed (resp. and $\PPP$-almost certain), then $(B\vert_{[0,\infty)})^{-1}(A)$ is closed for the L\'evy shifts $\Delta_t$, $t\in [0,\infty)$ (resp. and $\WW$-almost certain). This is because $(\Delta_tB)\vert_{[0,\infty)}=\Delta_t(B\vert_{[0,\infty)})$ for all $t\in [0,\infty)$ (note, the $\Delta_t$ on the left-hand side acts on $\Omega_0$, the one on the right-hand side -- on $\Theta_0$), resp. and because $(B\vert_{[0,\infty)})_\star \WW=\PPP$.
\item Suppose $A\subset \Omega_0$ is closed for the L\'evy shifts $\Delta_t$, $t\in [0,\infty)$ [or just $t$ belonging to $[0,\epsilon)$ for some $\epsilon>0$, leaving the argument in this case to the reader], and $\WW$-almost certain. Then $B:=\cap_{n\in \mathbb{Z}}\Delta_n^{-1}(A)$ is $\WW$-almost certain and is contained in $A$. Furthermore, for $t\in \mathbb{R}$, $\omega\in B$ and then for all $n\in \mathbb{Z}$, $\Delta_n\Delta_t\omega=\Delta_{t-\lfloor t\rfloor}\Delta_{n+\lfloor t\rfloor}\omega\in A$, viz. $\Delta_t\omega\in B$. By \ref{shifts-varia:i} we deduce that $B$ is shift-invariant.
\end{enumerate}

\begin{remark}\label{rmk:zero-one}
By independent increments of $B$, trivially, $\lim_{u\to \pm \infty}\WW(\Delta_u^{-1}(A)\cap A')=\WW(A)\WW(A')$ (the limit may be only over some set of $u\in \mathbb{R}$ that is unbounded) whenever $A$ and $A'$ are two events depending only on the increments of $B$ in a bounded interval. By approximation and the fact that the L\'evy shifts $\Delta_u$, $u\in \mathbb{R}$, are measure-preserving for $\WW$, the limit prevails for all $A$ and $A'$ from $\GG$. If an $A\in \GG$ is shift-invariant and more generally if merely $A=\Delta_v^{-1}(A)$ a.s.-$\WW$ for a non-zero  $v\in \mathbb{R}$ (so that $A=\Delta_{nv}^{-1}(A)$ a.s.-$\WW$ for all $n\in \mathbb{Z}$), then taking $A'=\Omega_0\backslash A$ shows that $\WW(A)\in \{0,1\}$: the zero-one law for shift-invariant events.
\end{remark}

\section{Random countable sets over the Wiener noise: locality and stationarity}\label{section:rcs-l-s}
\subsection{Major concepts}
We will want to measurably enumerate our random sets. Because we do not wish a priori to insist that the sets are non-empty, but still want to list them through a countably infinite sequence of random variables, we introduce for convenience a coffin state $\dagger\notin\mathbb{R}$, whose value is ignored when the range of such a sequence is considered as a subset of $\mathbb{R}$. Such a coffin state also comes in handy when we ``localize'' an enumeration. To wit:

\begin{definition}
For $s<t$ from $[-\infty,\infty]$, $(s,t)^\dagger:=(s,t)\cup \{\dagger\}$ is endowed with the $\sigma$-field $\mathcal{B}_{(s,t)^\dagger}:=\sigma(\mathcal{B}_{(s,t)}\cup \{\{\dagger\}\})$, where $\mathcal{B}_{(s,t)}$ are the Borel sets of the interval $(s,t)$. We understand $u\pm \dagger:=\dagger\pm u:=\dagger$ for $u\in \mathbb{R}$. For a sequence $s=(s_i)_{i\in \mathbb{N}}$ with values in $\mathbb{R}^\dagger$: $[s]:=\{s(i):i\in \mathbb{N}\}\backslash \{\dagger\}$ is the effective range of $s$; for extended-real $a<b$, the sequence $s\vert_{(a,b)}=(s\vert_{(a,b)}(i))_{i\in\mathbb{N}}$ given by $s\vert_{(a,b)}(i):=s(i)$ if $s(i)\in (a,b)$ and $s\vert_{(a,b)}(i):=\dagger$ otherwise, is $s$ localized to the interval $(a,b)$ of $\mathbb{R}$.
\end{definition}
We now define random countable sets over the Wiener noise, their locality and (various concepts of) stationarity, indicating the presence of these notions in (some of the) existent literature en route. 
\begin{definition}\label{definition:varia}
Let $M:\Omega_0\to 2^\mathbb{R}$ --- a random set ---  and let $\AA$ be a sub-$\sigma$-field of $\GG$. A (resp. $\AA$-) measurable enumeration for $M$ is a sequence $S=(S_i)_{i\in \mathbb{N}}$ of (resp. $\AA$-) measurable random variables with values in $\mathbb{R}^\dagger$, which satisfies $M=[S]$ a.s.-$\WW$. Such an enumeration is said to be perfect if the a.s.-$\WW$ qualifier can be dropped.  $M$ is a random countable set if it admits  a measurable enumeration.  If $M$ and $M':\Omega_0\to 2^\mathbb{R}$ are both random countable sets, then we say they have the same law provided there exist an enumeration $S$ for $M$, an enumeration $S'$ for $M'$ and a coupling $R$ of $S_\star \WW$ and $S'_\star \WW$ such that for $R$-a.e. $(s,s')$ one has $[s]=[s']$ (cf. \cite[Definition~2.2]{tsirelson-rc}, given there for $(0,1)$, not $\mathbb{R}$, as the ambient set,  which does not really matter). 
A random countable set $M$ is said to be:
\begin{enumerate}[(a)]
\item\label{definition:varia:a} dense (resp. empty), if $M$ is dense (resp. empty) $\WW$-a.s.;
\item\label{definition:varia:b} local, if for all extended-real $s<t$, $M\cap (s,t)$ admits an $\FF_{s,t}$-measurable enumeration (cf. \cite[Definition~2e2, first display]{tsirelson-nonclassical});
\item\label{definition:varia:f} stationary, if for all $u\in \mathbb{R}$, $\WW$-a.s. $M=u+M(\Delta_u)$ (cf. \cite[Definition~2e2, second line of second display]{tsirelson-nonclassical}, but note the typo: the left-hand side  should be offset by $h$ to the right), perfectly so, if  a.s.-$\WW$ can be dropped;
\item\label{definition:varia:c} stationary-in-law, if for all $u\in \mathbb{R}$ the sets $M$ and $u+M$ have the same law, i.e. if for all $u\in \mathbb{R}$ there exist measurable enumerations $T$ and $T_u$ of $M$ and a coupling $R$ of $T_\star\WW$ and $(T_u)_\star\WW$ such that, for $R$-a.e. $(x,x_u)$ one has $u+[x]=[x_u]$ (we connect it to the ``stationarity'' of \cite[Definition~6.8]{tsirelson-rc} in Proposition~\ref{proposition:reinforce-local+stationary-in-law});
\item \label{definition:varia:d} hit-or-miss stationary, if for all events $A$ from the so-called hit-or-miss $\sigma$-algebra $\mathfrak{h}:=\sigma_{2^\mathbb{R}}(\{2^{\mathbb{R}\backslash E}:E\in \mathcal{B}_\mathbb{R}\})$ the map $(\mathbb{R}\ni x\mapsto \WW(x+M\in A))$ is constant (just ``stationary'' in \cite[Definition~2.5(i)]{kendall_2000});
\item \label{definition:varia:e} hit-or-miss quasi-stationary if $\WW(x+M\in A)$ is either $=0$ or $>0$ simultaneously for all $x\in \mathbb{R}$ for any fixed $A\in \mathfrak{h}$ (just ``quasi-stationary'' in the terminology of \cite[Definition~2.5(ii)]{kendall_2000}).
\end{enumerate}
For $M':\Omega_0\to 2^\mathbb{R}$ we also say that $M'$ is a version of $M$ when $\{M=M'\}\in \GG$ (automatic if $M$ and $M'$ are random countable sets) and $\WW(M=M')=1$.
\end{definition}

Several immediate comments are in order. 

 As indicated in the Introduction, it is locality \emph{together} with stationarity (without further qualification, to be understood always in the sense of  Definition~\ref{definition:varia}\ref{definition:varia:f}) that will be the primary interest of our study. Nevertheless, we shall find it worthwhile, at least at first, to give some results holding the two notions  separate and/or employing, in lieu of stationarity, one of its weaker forms thereof delineated above. This will serve to better emphasize the relevance of the individual properties.
 
 Of considerable importance is 
 
\begin{remark}\label{remark:ranomd-countable-set}
Let $M:\Omega_0 \to 2^\mathbb{R}$. We may see $M$ as a subset of $\Omega_0\times \mathbb{R}$ in the natural way (namely, as the set $\llbracket M\rrbracket:=\{(\omega,t)\in \Omega_0\times \mathbb{R}:t\in M(\omega)\}$); vice versa, a subset of $\Omega_0\times \mathbb{R}$ is viewed canonically as a map $\Omega_0\to 2^\mathbb{R}$ (carrying a point from $\Omega_0$ onto its section of the subset in question). We will usually not make the distinction between the two explicit, leaving it to context to determine which is intended. By \cite[Theorem~3.2]{kendall_2000} \cite[\# 117]{dellacherie1975probabilites}, if $M\in \GG\otimes \mathcal{B}_\mathbb{R}$, then  the property of $M$ being a random countable set (in our sense; ``constructively countable'' in the sense of \cite[Definition~3.2]{kendall_2000}) is equivalent to $M$ being a.s. countable (``weakly countable'' in the sense of \cite[Definition~3.1]{kendall_2000}). Note also that, for any sub-$\sigma$-field $\AA$ of $\GG$, if $M$ admits a perfect $\AA$-measurable enumeration, then $M\in \AA\otimes \mathcal{B}_\mathbb{R}$; dropping ``perfect'' in the antecedent, the consequent may fail, however. A random countable set allows for ``mischievous'' behaviour on a null set. 
\end{remark}
 
 Two sets having the same law clearly are hit-or-miss stationary simultaneously either both, or simultaneously both not so. The same is true for hit-or-miss quasi-stationarity (again it is immediate) and stationarity-in-law, in which latter case it follows from the fact that equality in law is transitive \cite[Remark~2.3]{tsirelson-rc}. We stress that, by contrast,   stationarity itself  (or its absence) is not necessarily shared  by two sets having the same law. See e.g. Example~\ref{example:not-stationary} to follow. 

Some observations on the role of exceptional sets in Definition~\ref{definition:varia}. The property of $S$ being an  enumeration for $M$ is not affected if we change $S$ or $M$ on a $\WW$-negligible set. Thus, in the definition of locality, Item~\ref{definition:varia:b} above, we could just as well have asked for $M\cap (s,t)$ to admit an enumeration measurable relative to the increments of $B$ on $(s,t)$ (but it is usually more convenient to work with the completed $\sigma$-field $\FF_{s,t}$). Except for perfect stationarity the properties listed in \ref{definition:varia:a}-\ref{definition:varia:e} above  remain unaffected if $M$ is changed on a $\WW$-negligible set. In particular, a random countable set may, despite its name, be uncountable on a $\WW$-negligible set. By contrast, the value of a perfectly stationary random countable set is determined already by specifying its value on one member of each equivalence class of $\Omega_0$ w.r.t.  the equivalence relation $\sim$  specified according to $\omega^1\sim \omega^2\Leftrightarrow \exists u\in \mathbb{R}(\omega^1=\Delta_u(\omega^2))$ for $\{\omega^1,\omega^2\}\subset \Omega_0$. It is plain that this is sensitive to changes on $\WW$-negligible sets. Nevertheless, we may change a perfectly-shift invariant set on a shift-invariant $\WW$-negligible set to the empty set, say, and it does not affect the perfect stationarity property. If a random countable set admits an $\AA$-measurable enumeration, is countable with certainty (not just $\WW$-a.s.) and $\WW^{-1}(\{0,1\})\subset \AA$, then it admits an $\AA$-measurable perfect enumeration (because $\WW$ is complete). 

In general the reader will come to find, as he/she progresses through this paper, that there is little to no place to hide in this paper when it comes to considerations of exceptional (negligible) sets. The presence or indeed absence of a.s. qualifiers should not be taken lightly.
\begin{remark}\label{rmk:no-stationary}
Prima facie it may seem that a stationary random countable set $M$ should admit an enumeration $S$ that is stationary, i.e. one for which $S=h+S(\Delta_h)$ a.s.-$\WW$ for all $h\in \mathbb{R}$; however, it is impossible, unless $M$ is empty. In fact there can be no map $Q\in \GG/\mathcal{B}_{\mathbb{R}^\dagger}$, real-valued with positive $\WW$-probability, such that $Q=h+Q(\Delta_h)$ a.s.-$\WW$ for all $h\in \mathbb{R}$: if such a map did exist, then, for all $h\in \mathbb{R}$, since $\Delta_h$ preserves $\WW$, $Q_\star \WW\vert_{\{Q\in \mathbb{R}\}}=(h+\cdot)_\star (Q_\star\WW\vert_{\{Q\in \mathbb{R}\}})$, which is clearly absurd.
\end{remark}
All of the notions of Definition~\ref{definition:varia} make sense beyond the setting of $\WW$ being Wiener measure, $\FF$ the (noise-)factorization of $B$ and $\Delta$ the measure-preserving group of L\'evy shifts. Very naturally, the probability $\WW$ could, for instance, be replaced by the law of any two-sided ($=$ indexed by the real line, vanishing at zero) L\'evy process, mutatis mutandis; and still more generally, one could work with a one-dimensional noise in the sense of \cite[Definition~3d1]{tsirelson-nonclassical} (it gives us all the ingredients: the probability, the (noise-)factorization and the group of measure-preserving maps; satisfying certain conditions between them, of course). However, we have restricted, and shall in what follows continue to restrict our attention to Wiener noise, leaving the eventual extensions for future work. 

\subsection{General properties}
We expose now the salient features of the various types of stationary and of local random countable sets, and  investigate some of the inter-relations between them.

First, three simple  (counter)examples, to get us going.
\begin{example}
The rational numbers are a  local dense random countable set, however they are not stationary. More generally, any deterministic non-empty countable subset of $\mathbb{R}$ (dense or not) is a local random countable set, which is not stationary in any of the senses of Definition~\ref{definition:varia}.
\end{example}

\begin{example}\label{example:stationary-dense-not-local}
A random countable set can be stationary and dense without being local, indeed translating a dense stationary random countable set by any deterministic quantity is again a dense stationary random countable set, while locality will not be preserved by such a translation. 
\end{example}



\begin{example}
Let $M$ be the local minima of $B$ on the set $\Omega':=\{\omega\in \Omega_0:\omega\text{ hits }(0,\infty)\text{ immediately after time $0$}\}$ and let $M$ be the empty set off $\Omega'$. Then $M$ is a local dense stationary random countable set, which is not perfectly stationary.  In fact, the property that $M=h+M(\Delta_h)$ for all $h\in \mathbb{R}$ fails a.s.-$\WW$.
\end{example}


Here is the connection between stationarity and stationarity-in-law.
\begin{proposition}
Let $M$ be a random countable set with a measurable enumeration $T$. The following are equivalent. 
\begin{enumerate}[(A)]
\item  $M$ is stationary.
\item $M$ is stationary-in-law, moreover,  in Definition~\ref{definition:varia}\ref{definition:varia:c} one can take $T_u=T$ and $R=(T,T(\Delta_{-u}))_\star \WW$ for all $u\in \mathbb{R}$.
\end{enumerate}
\end{proposition}
\begin{proof}
Fix $u\in \mathbb{R}$. 
By Remark~\ref{rmk:shifts-preserving}, $(T,T(\Delta_{-u}))_\star \WW$ is a coupling of $T_\star \WW$ with itself. Thus the property of Definition~\ref{definition:varia}\ref{definition:varia:c} with $T_u=T$ and $R=(T,T(\Delta_{-u}))_\star \WW$ is equivalent to $u+[T]=[T(\Delta_{-u})]$ a.s.-$\WW$, i.e. to $M=-u+M(\Delta_{-u})$ a.s.-$\WW$. 
%
%
\end{proof}
Stationarity is not implied by stationarity-in-law, even when coupled with locality; it is a (much) stronger condition.
\begin{example}\label{example:not-stationary}
Let $M$ be the local minima of $B$ on the positive half-line and the local maxima of $B$ on the negative half-line. Then $M$ is equal in law to the local minima of $B$ (and hence also equal in law to the local maxima of $B$), is local and stationary-in-law, but  is not stationary. The locality and non-stationarity of $M$ are evident. Since the local minima are stationary and therefore stationary-in-law the same must be true of $M$, provided the stipulated equality in law does indeed hold true. It does, because we can enumerate the local minima of $B$ on the positive half-line with  the $S_k$, $k\in \mathbb{N}$ odd, and the local maxima of $B$ on the negative half-line with the $S_k$, $k\in \mathbb{N}$ even; then take the coupling $(S,\tilde S)_\star \WW$, where $\tilde S$ is equal to $S$ on the odd natural numbers, but equal to $S\circ (-B)$ on the even natural numbers (so that $\tilde S$ is an enumeration of the local minima).
\end{example}

We bring our definition of stationarity-in-law in line with  \cite[Definition~6.8]{tsirelson-rc}, modulo of course the natural changes that need to be made to go from $(0,1)$ of \cite{tsirelson-rc} to the real line $\mathbb{R}$ as the ambient set.  Note that our locality trivially implies the ``independence'' of \cite[Definition~4.2]{tsirelson-rc}. Thus it is not surprising that assuming the set to be also local makes it possible to strengthen the property of \cite[Definition~6.8]{tsirelson-rc} to one which works with ``local enumerations'' (the ``moreover'' part below).

\begin{proposition}\label{proposition:reinforce-local+stationary-in-law}
Suppose the random countable set $M$ is  stationary-in-law. For all extended-real $s<t$ [not just $s=-\infty$ and $t=\infty$] and all real $u$, $u+M\cap (s,t)$ and $M\cap (s+u,t+u)$ have the same law; if $M$ is local, then, moreover, there exist an $\FF_{s,t}$-measurable enumeration $S'$ of $M\cap (s,t)$, an $\FF_{s+u,t+u}$-measurable enumeration $S_u'$ of $M\cap (s+u,t+u)$, and a coupling $P'$ 
of $S'_\star\WW$ and $(S_u')_\star \WW$ such that for $P'$-a.e. $(x',x_u')$ one has $u+[x']=[x_u']$.
\end{proposition}
\begin{proof}
Let $T$, $T_u$ and $R$ be as in Definition~\ref{definition:varia}\ref{definition:varia:c}. Then $S:=T\vert_{(s,t)}$ is a measurable enumeration of $M\cap (s,t)$, while $S_u:=T_u\vert_{(s+u,t+u)}$ is  a measurable enumeration of $M\cap (s+u,t+u)$. Furthermore $P:=(\mathrm{pr}_1\vert_{(s,t)},\mathrm{pr}_2\vert_{(s+u,t+u)})_\star R$ is a coupling of $S_\star\WW$ and $(S_u)_\star\WW$ such that for $P$-a.e. $(x,x_u)$ one has $u+[x]=[x_u]$. This gives the first statement. To get its strengthening, we assume $M$ is local and remedy the coupling as follows.

Using locality, let $S'$ be an $\FF_{s,t}$-measurable enumeration of $M\cap (s,t)$ and let $S_u'$ be an $\FF_{s+u,t+u}$-measurable enumeration of $M\cap (s+u,t+u)$. Put $Q:=(S',S)_\star \WW$ and $Q_u:=(S_u,S'_u)_\star \WW$. Let 
\begin{equation*}
Q(\dd(x',x))=:q(\dd x',x)\WW(S\in \dd x),\quad (x',x)\in ((s,t)^\dagger)^\mathbb{N}\times ((s,t)^\dagger)^\mathbb{N},
\end{equation*}
 be a disintegration of $Q$ against the second marginal; likewise let  
 \begin{equation*}Q_u(\dd (x_u,x_u'))=:\WW(S_u\in \dd x_u)q_u(x_u,\dd x_u'),\quad (x_u,x'_u)\in ((s+u,t+u)^\dagger)^\mathbb{N}\times ((s+u,t+u)^\dagger)^\mathbb{N},
 \end{equation*}
 be a disintegration of $Q_u$ against the first marginal. Then define $$P'(\dd(x',x_u')):=\int q(\dd x',x)q(x_u,\dd x_u')P(\dd(x,x_u)),\quad (x',x'_u)\in ((s,t)^\dagger)^\mathbb{N}\times ((s+u,t+u)^\dagger)^\mathbb{N},$$ which is a coupling of $(S')_\star\WW$ and of $(S'_u)_\star\WW$; for instance, the computation
\begin{align*}
\WW(S'_u\in \dd x_u')&= Q_u(( (s+u,t+u)^\dagger)^\mathbb{N}\times  \dd x_u')=\int q_u(x_u,\dd x_u')\WW(S_u\in \dd x_u)=\int q_u(x_u,\dd x_u')P(\dd(x,x_u))\\
&=\int \int q(\dd x',x)q_u(x_u,\dd x_u')P(\dd(x,x_u))= P'(((s,t)^\dagger)^\mathbb{N}\times \dd x_u'),\quad x'_u\in ((s+u,t+u)^\dagger)^\mathbb{N},
 \end{align*} shows that $P'$ has the correct second marginal. 
 Furthermore, note that because $[S']=M\cap (s,t)=[S]$ a.s.-$\WW$, we have $[x']=[x]$ for $Q$-a.e. $(x',x)$, so  $[x']=[x]$ a.e.-$q(\dd x',x)$ for $S_\star \WW$-a.e. $x$. Likewise $[x_u]=[x_u']$ a.e.-$q_u(x_u,\dd x_u')$ for $(S_u)_\star \WW$-a.e. $x_u$.
 
 Letting now $p_1$, resp. $p_2$, be a disintegration of $P$ against the first, resp. second marginal, we see, putting $A:=\{(x'',x_u'')\in ((s,t)^\dagger)^\mathbb{N}\times ((s+u,t+u)^\dagger)^\mathbb{N}:u+[x'']=[x''_u]\}$, that 
\begin{align*}
P'(A) &=\int P(\dd(x,x_u))\int q(\dd x',x)\int q_u(x_u,\dd x_u')\mathbbm{1}_A(x',x_u')\\ 
&=\int \WW(S\in \dd x)\int p_1(x,\dd x_u)\int q(\dd x',x)\int q_u(x_u,\dd x_u')\mathbbm{1}_A(x',x_u')\\ 
&=\int \WW(S\in \dd x)\int q(\dd x',x)\int p_1(x,\dd x_u)\int q_u(x_u,\dd x_u')\mathbbm{1}_A(x',x_u')\\ 
&=\int \WW(S\in \dd x)\int q(\dd x',x)\int p_1(x,\dd x_u)\int q_u(x_u,\dd x_u')\mathbbm{1}_A(x,x_u')\\ 
&=\int \WW(S\in \dd x)\int p_1(x,\dd x_u)\int q(\dd x',x)\int q_u(x_u,\dd x_u')\mathbbm{1}_A(x,x_u')\\ 
&=\int P(\dd(x,x_u))\int q(\dd x',x)\int q_u(x_u,\dd x_u')\mathbbm{1}_A(x,x_u')\\ 
&=\int \WW(S_u\in \dd x_u)\int p_2(\dd x,x_u)\int q(\dd x',x)\int q_u(x_u,\dd x_u')\mathbbm{1}_A(x,x_u')\\ 
&=\int \WW(S_u\in \dd x_u)\int q_u(x_u,\dd x_u')\int p_2(\dd x,x_u)\int q(\dd x',x)\mathbbm{1}_A(x,x_u')\\ 
&=\int \WW(S_u\in \dd x_u)\int q_u(x_u,\dd x_u')\int p_2(\dd x,x_u)\int q(\dd x',x)\mathbbm{1}_A(x,x_u)\\ 
&=\int \WW(S_u\in \dd x_u)\int p_2(\dd x,x_u)\int q(\dd x',x)\int q_u(x_u,\dd x_u')\mathbbm{1}_A(x,x_u)\\ 
&=\int P(\dd(x,x_u))\int q(\dd x',x)\int q_u(x_u,\dd x_u')\mathbbm{1}_A(x,x_u)\\ 
&=\int P(\dd(x,x_u))\mathbbm{1}_A(x,x_u)=P(A)=1,
\end{align*}
the last equality being true due to the very nature of $P$. With these $S'$, $S_u'$ and $P'$ the proof is complete.
\end{proof}

A hit-or-miss stationary set is trivially hit-or-miss quasi-stationary.  Less obvious is the implication of
\begin{proposition}\label{proposition:stationary-in-law-gives-kendall}
If the  random countable set $M$ is stationary-in-law, then $M$ is hit-or-miss stationary. 
\end{proposition}
\begin{proof}
Take $u\in \mathbb{R}$ and  $E\in \mathcal{B}_\mathbb{R}$ and let $T$, $T_u$  and $R$ be as in Definition~\ref{definition:varia}\ref{definition:varia:c}. Then $\WW(M\cap E=\emptyset)=\WW([T]\cap E=\emptyset)=\WW(([T]+u)\cap (E+u)=\emptyset)=R(([\pr_1]+u)\cap (E+u)=\emptyset)=R([\pr_2]\cap (E+u)=\emptyset)=\WW([T_u]\cap (E+u)=\emptyset)=\WW(M\cap (E+u)=\emptyset)$. It remains to note that the constancy property of Definition~\ref{definition:varia}\ref{definition:varia:d} is a  Dynkin system in $A\in 2^{2^\mathbb{R}}$ containing the  $\mathfrak{h}$-generating $\pi$-system consisting of the ``miss-events'' $\{\cdot \cap E=\emptyset\}=2^{\mathbb{R}\backslash E}$, $E\in \mathcal{B}_\mathbb{R}$.
\end{proof}
Let us explore how locality and stationarity affect the probability of inclusion of a deterministic real number into the random countable set under inspection.
\begin{proposition}\label{SIL:probability-zero}
A local random countable set $M$ satisfies $\WW(x\in M)\in \{0,1\}$ for all except at most denumerably many $x\in \mathbb{R}$. If $M$ is hit-or-miss quasi-stationary, then $\WW(x\in M)=0$ for all $x\in \mathbb{R}$. 
\end{proposition}
\begin{proof}
Consider the first statement. By locality the family of events $(\{x\in M\})_{x\in \mathbb{R}}$ is an independency. Since (the domain of) $\WW$ is essentially separable, the first claim follows (otherwise one could construct, by centering the indicators of those events $\{x\in M\}$, $x\in \mathbb{R}$, which are non-trivial, an uncountable orthogonal family of non-zero vectors in the separable Hilbert space $\mathrm{L}^2(\WW)$, which cannot be). Suppose now $M$ is hit-or-miss quasi-stationary. Then by definition $\WW(x\in M)=\WW(\{x\}\cap M\ne \emptyset)=\WW(\{0\}\cap (M-x)\ne \emptyset)$ is $=0$ or $>0$, one or the other simultaneously for all $x\in \mathbb{R}$. It cannot be the latter, because if $S=(S_k)_{k\in \mathbb{N}}$ is a measurable enumeration of $M$, then in order for $\WW(x\in M)=\WW(x\in [S])$ to be $>0$, $x$ must be an atom of one of the $(S_k)_\star(\WW\vert_{\{S_k\in\mathbb{R}\}})$, $k\in \mathbb{N}$, and there can be only countably many such atoms (of course, in consequence, there are in fact none).
\end{proof}
%

\begin{example}
Recall Example~\ref{example:stationary-dense-not-local}. For an instance of a dense stationary random countable set that is not a deterministic translate of a  local set, we may proceed as follows. Take two stationary dense local random countable sets $M^1$ and $M^2$, two distinct real numbers, $h_1$ and $h_2$, notice that by enumerability, locality, and the conclusion of Proposition~\ref{SIL:probability-zero} in the stationary part, $(h_1+M_1)\cap (h_2+ M_2)$ is empty, and form $\tilde M:=(h_1+M_1)\cup (h_2+M_2)$. $\tilde M$ is evidently a stationary dense random countable set, but it is not a deterministic translate of a local set. Suppose indeed there were an $h\in \mathbb{R}$ and a local random countable set $M$ such that $h+M=\tilde M$, i.e. $M=(h_1-h+M_1)\cup (h_2-h+M_2)$. Inspecting the latter on $(-\epsilon,\epsilon)$ for $\epsilon\in (0,(\vert h_1-h\vert\lor\vert h_2-h\vert)/2)$ yields an immediate contradiction with enumerability, locality,  density and the conclusion of Proposition~\ref{SIL:probability-zero} in the stationary part (again).
\end{example}

Combining locality with hit-or-miss stationarity we get the property of the conclusion of
\begin{proposition}\label{proposition:dense-or-empty}
Let $M$ be local and hit-or-miss stationary. Then $M$ is dense or empty.
\end{proposition}
A related but trivial fact is that a stationary random countable set is either empty or $\WW$-a.s. not empty. This is just because for $v\in \mathbb{R}$, $\Delta_v^{-1}(\{M=\emptyset\})=\{M(\Delta_v)=\emptyset\}=\{v+M(\Delta_v)=\emptyset\}=\{M=\emptyset\}$ a.s.-$\WW$, so one can apply Remark~\ref{rmk:zero-one}.
\begin{proof}
Since $M$ is hit-or-miss stationary the quantity $f(a):=\WW(M\cap (h,h+a)=\emptyset)$ depends on the real parameters $h$ and $a\geq 0$ through $a$ only; by locality and because of Proposition~\ref{SIL:probability-zero}, $f(a+b)=f(a)f(b)$ for $\{a,b\}\subset [0,\infty)$. By monotonicity of $f:[0,\infty)\to [0,1]$ it follows that $f(a)=e^{-\lambda a}$ for all $a\in (0,\infty)$ for some $\lambda\in [0,\infty]$. If $\lambda=0$ it means that a.s. $M$ is empty; suppose $\lambda>0$. Then $M$ being dense a.s. will  follow if we can show that $\lambda=\infty$. Suppose per absurdum that $\lambda<\infty$. Let $S:=\inf\{h\in (0,\infty): M\cap (0,h)\ne\emptyset\}$, which is exponentially distributed with parameter $\lambda$ by the preceding and an $\FF^{0,\rightarrow}$-stopping time by locality of $M$. Moreover, the process $(\mathbbm{1}_{\{S\leq t\}}-\lambda(S\land t))_{t\in [0,\infty)}$ is a discontinuous real martingale  in the Brownian filtration $\FF^{0,\rightarrow}$: for all real $s\leq t$ a.s.-$\WW$,
\begin{align*}
\WW[\mathbbm{1}_{\{S\leq t\}}-\lambda(S\land t)\vert \FF_s]&=\mathbbm{1}_{\{S\leq s\}}(1-\lambda S)+\mathbbm{1}_{\{s<S\}}\WW[\mathbbm{1}_{\{S\leq t\}}-\lambda(S\land t)\vert \FF_s]\\
&=\mathbbm{1}_{\{S\leq s\}}(1-\lambda S)+\mathbbm{1}_{\{s<S\}}\WW[\mathbbm{1}_{\{M\cap (s,t)\ne \emptyset\}}\\
&\qquad \qquad \qquad \qquad \qquad \qquad -\lambda((s+\inf\{h\in (0,\infty):M\cap (s,s+h)\ne \emptyset\})\land t)\vert \FF_s]\\
&=\mathbbm{1}_{\{S\leq s\}}(1-\lambda S)+\mathbbm{1}_{\{s<S\}}\WW[\mathbbm{1}_{\{M\cap (s,t)\ne \emptyset\}}\\
&\qquad \qquad \qquad \qquad \qquad \qquad -\lambda((s+\inf\{h\in (0,\infty):M\cap (s,s+h)\ne \emptyset\})\land t)]\\
&=\mathbbm{1}_{\{S\leq s\}}(1-\lambda S)+\mathbbm{1}_{\{s<S\}}\left(1-e^{-\lambda (t-s)}-\lambda\left(s+\int_0^\infty \lambda e^{-\lambda u}(u\land (t-s))\dd u\right)\right)\\
&=\mathbbm{1}_{\{S\leq s\}}(1-\lambda S)+\mathbbm{1}_{\{s<S\}}\left(1-e^{-\lambda (t-s)}-\lambda\left(s+\frac{1-\mathrm{e}^{-\lambda (t-s)}}{\lambda}\right)\right)\\
&=\mathbbm{1}_{\{S\leq s\}}(1-\lambda S)-\mathbbm{1}_{\{s<S\}}\lambda s\\
&=\mathbbm{1}_{\{S\leq s\}}-\lambda(S\land s).
\end{align*}
But this cannot be.  (Incidentally, here is one of a legion of instances in this paper when one can ask oneself, ``are you careful enough?'' Indeed, one is tempted to assert that $(\vert M\cap (0,t]\vert -\lambda t)_{t\in [0,\infty)}$ in the preceding should be a compensated homogeneous Poisson process. Is it true? For sure, we had arrived at a contradiction, so anything in principle could have went. Could it have been argued ``directly''?  We feel that no, not even if we had assumed  stationarity in lieu of  hit-or-miss stationarity, at least not until one has armed oneself with nice versions of such $M$, which allow for arguments involving stopping times and the strong Markov property, cf. Proposition~\ref{proposition:two-sided-viz-one-sided} to follow.)
\end{proof}

So, a hit-or-miss stationary, in particular a stationary, local random countable set over the Wiener noise that is not trivial (empty) is automatically dense. 
As it happens, in the non-trivial (dense) case we cannot tell two local stationary random countable sets apart based solely on their laws: they are all the same as far as couplings of enumerations can see.

\begin{proposition}\label{proposition:equal-in-law}
Suppose $M_1$ and $M_2$ are two dense random countable sets, both of them local, both of them stationary-in-law. Then $M_1$ and $M_2$ have the same law. Moreover, for any probability law $\mathcal{L}$ on  the Borel sets of $\mathbb{R}$ equivalent to $\mathfrak{l}$ and for any measurable enumeration $S$ of $M_1$  there exists a coupling $R$ of $ \mathcal{L}^{\times \mathbb{N}}$ and $S_\star \WW$ such that $[x]=[y]$ for $R$-a.e. $(x,y)$.
\end{proposition}
\begin{proof}
From Proposition~\ref{proposition:reinforce-local+stationary-in-law} and the comments preceding it, we know that for each $k\in \mathbb{Z}$, the sets $M_1\cap (k,k+1)-k$ and $M_2\cap (k,k+1)-k$  satisfy the ``independence condition'' and are ``stationary'' in the sense of \cite{tsirelson-rc}. Therefore, being also dense by assumption, according to \cite[Theorem~6.9]{tsirelson-rc}  they have the same law. By a routine diagonal argument, taking into account also the stationary part of Proposition~\ref{SIL:probability-zero}, we deduce that $M_1$ has the same law as $M_2$. Furthermore, from the part of \cite[Theorem~6.9]{tsirelson-rc} having to do with the ``uniform distribution'' and from \cite[Lemma~1.2, Definition~2.4]{tsirelson-rc} the second claim follows.
\end{proof}

We establish next that any stationary random countable set admits a version that is perfectly stationary. 

\begin{lemma}\label{lemma:steinhaus}
Let $A_1\subset \mathbb{R}$ and $A_2\subset \mathbb{R}$ be two co-negligible sets w.r.t. $\mathfrak{l}$. Then $A_2-A_1=\mathbb{R}$.
\end{lemma}
\begin{proof}
Let $m\in \mathbb{R}$. By the translation invariance of $\mathfrak{l}$ and because the union of two null sets is null, for $\mathfrak{l}$-a.e., and a fortiori for some $a\in \mathbb{R}$, we have $a\in A_1$ and also $a+m\in A_2$, therefore $m\in A_2-A_1$.
\end{proof}

\begin{proposition}\label{proposition:version-perfect-stationary}
Every stationary random countable set admits a version that is perfectly stationary, countable with certainty and that belongs to $\mathcal{B}_{\Omega_0}\otimes \mathcal{B}_\mathbb{R}$.
\end{proposition}
Not even in the dense case can we ask the version to be countably infinite with certainty, since on the zero path any perfectly stationary version must be empty or the whole of the real line.

\begin{proof}
The proof is inspired by the technique of perfecting crude cocycles in the theory of random dynamical systems, see e.g. \cite[Section~1.3]{arnold}.

Let $M$ be a stationary random countable set. Changing $M$ on a $\WW$-negligible set we may and do ask that there is a $\mathcal{B}_{\Omega_0}$-measurable enumeration $S$ thereof such that $M=[S]$ with certainty (not just a.s.-$\WW$), a perfect  $\mathcal{B}_{\Omega_0}$-measurable enumeration. Define:
$$R:=\{(s,\omega)\in \mathbb{R}\times \Omega_0:M(\omega)=s+M(\Delta_s\omega)\},$$
$$\Omega_1:=\{\omega\in \Omega_0:(s,\omega)\in R\text{ for $\lll$-a.e. }s\in \mathbb{R}\}\text{ and }$$
$$\Omega_2:=\{\omega\in \Omega_0:\Delta_u\omega\in \Omega_1\text{ for $\lll$-a.e. }u\in \mathbb{R}\}.$$
By measurably enumerating  $M$ through $S$  we see that $R\in  \mathcal{B}_\mathbb{R}\otimes\mathcal{B}_{\Omega_0}$. By the stationarity of $M$ and Tonelli we have that $\Omega_1\in \mathcal{B}_{\Omega_0}$ and $\WW(\Omega_1)=1$. Noting that for $\omega\in \Omega_0$, 
\begin{equation*}\omega\in \Omega_2\Leftrightarrow \int\mathbbm{1}_{\Omega_0\backslash\Omega_1}(\Delta_u\omega)\lll(\dd u)=0,\end{equation*} by Tonelli again, and because the L\'evy shifts preserve the probability $\WW$, we obtain that also $\Omega_2\in \mathcal{B}_{\Omega_0}$ and $\WW(\Omega_2)=1$. Besides, $\Omega_2$ is shift-invariant because $\lll$ is translation invariant.

Define now, for $\omega\in \Omega_0$, $$\tilde{M}(\omega):=
\begin{cases}
s+M(\Delta_s\omega)& \text{for any $s\in \mathbb{R}$ for which $\Delta_s\omega\in \Omega_1$, if $\omega\in \Omega_2$;}\\
\emptyset & \text{if $\omega\notin\Omega_2$.}
\end{cases}
$$ We must check at once that $\tilde{M}$ is well-defined. Let then $\omega\in\Omega_2$ and  $\{s,u\}\subset \mathbb{R}$ be such that $\{\Delta_s\omega,\Delta_u\omega\}\subset \Omega_1$. Then for $\lll$-a.e. $v\in \mathbb{R}$ we have $s+M(\Delta_s\omega)=s+v+M(\Delta_{s+v}\omega)$ and at the same time for $\lll$-a.e. $w\in \mathbb{R}$ we have $u+M(\Delta_u\omega)=u+w+M(\Delta_{u+w}\omega)$. By Lemma~\ref{lemma:steinhaus} we obtain  $s+M(\Delta_s\omega)=u+M(\Delta_u\omega)$. On the other hand, if $\omega\in \Omega_2$ then  $\Delta_s\omega\in \Omega_1$ even for $\lll$-a.e. $s\in \mathbb{R}$, in particular for some $s\in \mathbb{R}$. Thus $\tilde{M}$ is well-defined.

Since $M$ is countable with certainty, so too is $\tilde M$.

We now check that $\tilde M\in \mathcal{B}_{\Omega_0}\otimes \mathcal{B}_\mathbb{R}$. Indeed, the set $$Q:=\{(s,t,\omega)\in \mathbb{R}\times \mathbb{R}\times \Omega_0: t\in s+M(\Delta_s(\omega))\}$$ belongs to $\mathcal{B}_\mathbb{R}\otimes \mathcal{B}_\mathbb{R}\otimes \mathcal{B}_{\Omega_0}$. By the preceding, for $\omega\in \Omega_2$,  $\int_0^1 \mathbbm{1}_Q(s,t,\omega)\lll(\dd s)$ is equal to $1$  if $t\in \tilde M(\omega)$ and is equal to $0$ for $t\in \mathbb{R}\backslash \tilde M(\omega)$. Therefore, by Tonelli, the asserted measurability of $\tilde M$ in the tensor product follows.

Next, turn to proving perfect stationarity. If $\omega\in\Omega_0\backslash \Omega_2$, then because $\Omega_2$ is shift-invariant we have $\tilde{M}(\omega)=\emptyset=u+\tilde{M}(\Delta_u\omega)$ for all $u\in \mathbb{R}$. On the other hand, let $\omega\in \Omega_2$ and $u\in \mathbb{R}$. By shift-invariance of $\Omega_2$ again,  $\Delta_u\omega\in \Omega_2$ also. Therefore for some $s\in \mathbb{R}$ such that $\Delta_s\omega\in \Omega_1$ and for some $v\in \mathbb{R}$ such that $\Delta_{u+v}\omega\in \Omega_1$ we have
$$\tilde{M}(\omega)=s+M(\Delta_s\omega)\text{ and }u+\tilde{M}(\Delta_u\omega)=u+v+M(\Delta_{u+v}\omega).$$ In turn, for $\lll$-a.e. $z\in \mathbb{R}$ and for $\lll$-a.e. $z'\in \mathbb{R}$, $$s+M(\Delta_s\omega)=s+z+M(\Delta_{s+z}\omega)\text{ and }u+v+M(\Delta_{u+v}\omega)=u+v+z'+M(\Delta_{u+v+z'}\omega),$$ which by Lemma~\ref{lemma:steinhaus} entails that $\tilde{M}(\omega)=u+\tilde{M}(\Delta_u\omega)$. We see thus that $\tilde{M}$ has perfect stationarity.

Finally, let us check that $\tilde{M}=M$ on $\Omega_1\cap\Omega_2$ rendering $\tilde{M}$ a version of $M$. Take $\omega\in \Omega_1\cap\Omega_2$. Then because $\omega\in \Omega_2$ certainly, by the very definition of $\tilde{M}$,  $\tilde{M}(\omega)=s+M(\Delta_s\omega)$ for $\lll$-a.e. $s$, while $s+M(\Delta_s\omega)=M(\omega)$ for $\lll$-a.e. $s\in \mathbb{R}$ because $\omega\in \Omega_1$; for some (indeed $\lll$-a.e.) $s\in \mathbb{R}$ both of the preceding equalities prevail and so $\tilde{M}(\omega)=s+M(\Delta_s\omega)=M(\omega)$.
\end{proof}
Thus, at least in some sense, perfect stationarity is merely an added technicality over and above stationarity, which is fundamental. Nevertheless, the fact that a set can be perfected in this way is non-trivial and we will find occasion to apply it before long.
%

We continue our collage of properties of stationary and local sets by discussing in technical terms the somewhat lax assertion from the Introduction that ``there are no interesting events to be described concerning such a [dense] random [countable] set.''

The following simple technical truth will be used.
\begin{lemma}\label{lemma:commute-decreasing-intersection-trace}
Let $\Gamma$ be any set and $(\AA_n)_{n\in \mathbb{N}}$ a nonincreasing sequence of $\sigma$-fields on $\Gamma$. Let also $X:\Theta\to \Gamma$ be any map. Then $$X^{-1}(\cap_{n\in \mathbb{N}}\AA_n)=\cap_{n\in \mathbb{N}}X^{-1}(\AA_n).$$ In particular, for every $\Gamma'\in 2^\Gamma$, $$(\cap_{n\in \mathbb{N}}\AA_n)\vert_{\Gamma'}=\cap_{n\in \mathbb{N}}(\AA_n\vert_{\Gamma'}).$$ 
\end{lemma}
In plain(er) terms: nonincreasing intersections of $\sigma$-fields and pull-backs, the trace operation in particular, commute.
\begin{proof}
The inclusion $\subset$ is trivial. For the reverse inclusion, take an $A$ from the r.h.s. of the stipulated equality. For each $n\in \mathbb{N}$ there is $A_n\in \AA_n$ such that $A=X^{-1}(A_n)$. Then $\limsup_{n\to\infty}A_n\in \cap_{n\in \mathbb{N}}\AA_n$ ($\because$ $\AA_n$ is nonincreasing in $n\in \mathbb{N}$) and $X^{-1}(\limsup_{n\to\infty}A_n)=\limsup_{n\to\infty}X^{-1}(A_n)=A$. For the special case take $X=\mathrm{id}_{\Gamma'}:\Gamma'\to\Gamma$.
\end{proof}

\begin{proposition}\label{proposition:final-is-trivial}
Assume that $M$ is a local random countable set that is stationary-in-law.  
Let $A\subset 2^\mathbb{R}$ be such that $\{M\in A\}\in \GG$. Then $\WW(M\in A)\in \{0,1\}$. 
\end{proposition}
In other words, the $\WW$-law of $M$ on its final $\sigma$-field $\{A\in 2^{2^\mathbb{R}}:\{M\in A\}\in \GG\}$ is trivial. 
\begin{proof}
From Propositions~\ref{proposition:stationary-in-law-gives-kendall} and~\ref{proposition:dense-or-empty} we get that $M$ is empty or dense. The case when $M$ is empty is trivial, so we focus on the case when $M$ is dense. Then we may and do further assume that $M$ admits a $\mathcal{B}_{\Omega_0}$-measurable perfect injective real-valued enumeration $S$  (just by changing $M$ on a $\WW$-negligible set if necessary, using completeness of $\WW$). Fix $\Omega_1\in\mathcal{B}_{\Omega_0}$ of  $\WW$-probability one  such that $\{M\in A\}\cap\Omega_1=\{[S\vert_{\Omega_1}]\in A\}\in \mathcal{B}_{\Omega_0}\vert_{\Omega_1}$ (it exists because $\{M\in A\}$ is from the $\WW$-completion of $\mathcal{B}_{\Omega_0}$). 
 
Apply Proposition~\ref{proposition:equal-in-law}. Fixing a probability law $\mathcal{L}$ on $\mathcal{B}_\mathbb{R}$ equivalent to $\mathfrak{l}$ there exists a coupling $R$ of $ \mathcal{L}^{\times \mathbb{N}}$ and of $S_\star \WW$ such that $[x]=[y]$ for $R$-a.e. $(x,y)$.  Let $$ \WW(\dd\omega)=:\int (S_\star \WW)(\dd y)\mathfrak{w}(y,\dd \omega),\quad \omega\in \Omega_0,$$ be a disintegration of $\WW$ against $S_\star\WW$; then put $$P(\dd(x,\omega)):=\int R(\dd (x,y))\mathfrak{w}(y,\dd \omega),\quad (x,\omega)\in\mathbb{R}^\mathbb{N}\times \Omega_0,$$ which is a coupling of $ \mathcal{L}^{\times \mathbb{N}}$  and of $\WW$ satisfying $[x]=[S(\omega)]$ for $P$-a.e. $(x,\omega)$. The space $\mathbb{R}^\mathbb{N}\times \Omega_0$ endowed with the $\sigma$-field $(\mathcal{B}_\mathbb{R})^{\otimes \mathbb{N}} \otimes \mathcal{B}_{\Omega_0}$, on which lives $P$, is Blackwell in the sense of Meyer \cite[D~III.15]{meyer}, indeed it is the measurable space of a Polish space (a Polish space endowed with its Borel $\sigma$-field is Blackwell \cite[T~III.16]{meyer}). Restricting $P$ to to the $P$-almost certain set $$\Omega':=\{(x,\omega)\in \mathbb{R}^\mathbb{N}\times \Omega_1:x\text{ injective and }[x]=[S(\omega)]\}\in (\mathcal{B}_\mathbb{R})^{\otimes \mathbb{N}} \otimes \mathcal{B}_{\Omega_0}$$ we get $P'$ whose underlying $\sigma$-field $\mathcal{P}$ is still Blackwell (the trace [on a measurable set] preserves the Blackwell property which is immediate from the definition \cite[D~III.15]{meyer}).

Denoting by $(X',W')$ the coordinate projections on $\Omega'$ we see that $$ \{[X']\in A\}=\{[S(W')]\in A\}=(\mathbb{R}^\mathbb{N}\times \{[S\vert_{\Omega_1}]\in A\})\cap\Omega'\in \mathcal{P}$$ and  
\begin{align*}
\WW(M\in A)&=\WW(\{M\in A\}\cap\Omega_1)= \WW([S\vert_{\Omega_1}]\in A)=P(\mathbb{R}^\mathbb{N}\times \{[S\vert_{\Omega_1}]\in A\})=P'((\mathbb{R}^\mathbb{N}\times \{[S\vert_{\Omega_1}]\in A\})\cap\Omega')\\
&=P'([X']\in A).
\end{align*} 
It remains to establish that $P'([X']\in A)\in \{0,1\}$. 
 
Then set, for $n\in \mathbb{N}_0\cup\{\infty\}$, $$\mathcal{E}_n:=\{E\in (\mathcal{B}_{\mathbb{R}})^{\otimes \mathbb{N}}\vert_{(\mathbb{R}^\mathbb{N})_\ne}:\omega\in E\Rightarrow \omega\circ p\in E\text{ for all permutations $p$ of $\mathbb{N}$ that restrict to the identity on $\mathbb{N}_{>n}$}\}.$$ Thus $\mathcal{E}_\infty=\cap_{n\in \mathbb{N}_0}\mathcal{E}_n$ is the exchangeable $\sigma$-field (traced on $(\mathbb{R}^\mathbb{N})_\ne:=$ the space of injective real sequences). It is elementary to verify that, for all $n\in \mathbb{N}_0$, $\mathcal{E}_n$ is a countably generated $\sigma$-field (its elements are indeed in a bijective correspondence with the Borel subsets of $\{s\in (\mathbb{R}^\mathbb{N})_\ne:s_1<\cdots<s_n\}$) whose atoms are given by 
$$\{s'\in (\mathbb{R}^\mathbb{N})_\ne:s'\vert_{\mathbb{N}_{>n}}=s\vert_{\mathbb{N}_{>n}}\text{ and }[s']=[s]\},\quad s\in  (\mathbb{R}^\mathbb{N})_\ne.$$ (On the other hand, it would not be straightforward to see and is perhaps [probably] not true that $\mathcal{E}_\infty$ is separable.)
 Still holding the $n\in \mathbb{N}_0$ fixed we see that $\{[X']\in A\}$ is a union of the atoms of the separable sub-$\sigma$-field ${X'}^{-1}(\mathcal{E}_n)$ of $\mathcal{P}$. It follows by Blackwell's theorem \cite[T~III.17]{meyer} that $\{[X']\in A\}\in {X'}^{-1}(\mathcal{E}_n)$. This being true for all $n\in \mathbb{N}_0$ we deduce from Lemma~\ref{lemma:commute-decreasing-intersection-trace} that $\{[X']\in A\}\in{ X'}^{-1}(\mathcal{E}_\infty)$. 
 
We conclude by applying  the Hewitt-Savage zero-one law (which delivers the $P'$-triviality of ${ X'}^{-1}(\mathcal{E}_\infty)$).
\end{proof}
It may be mentioned that Kendall  \cite[Theorem~4.3]{kendall_2000} proves (when applied to our setting) that the $\WW$-law of $M$ on $\mathfrak{h}$ is trivial assuming only that $M$ is hit-or-miss quasi-stationary and dense. In  \cite[Eq.~(10), Theorem~5.1]{kendall_2000}, the $\WW$-triviality of $M^{-1}(\mathfrak{h})$ is extended to the $\WW$-triviality of $\cap\sigma(S)$, where $S$ runs over all measurable enumerations of $M$, albeit we have not managed to fully understand the proof sketched there. In any event, even the $\sigma$-field $\cap\sigma(S)$ is not completely definitive (vis-\'a-vis establishing that all events \emph{of} $M$ are trivial), whereas the final $\sigma$-field of $M$ (as rendered trivial by Proposition~\ref{proposition:final-is-trivial})  absolutely is so.

Local stationary random countable sets (we no longer bother to separate the two properties or intervene with weaker forms of  stationarity in lieu of stationarity)  are actually determined by their restrictions to $(0,\infty)$ as we proceed to demonstrate.

\begin{definition}\label{definition:one-sided}
 For a random set $N:{\Theta_0}\to 2^{(0,\infty)}$  and $\BB$ a sub-$\sigma$-field of $\mathcal{H}$ say that $N$ admits the $\BB$-measurable  [just measurable when $\BB$ is $\mathcal{H}$] enumeration $S$ (resp. is a random countable set, is local, is stationary, perfectly stationary) when $S$ is a sequence of $\BB$-measurable random variables with values in $(0,\infty)^\dagger$, which satisfies $N=[S]$ a.s.-$\PPP$ (resp. $N$ admits a measurable enumeration, for all $s<t$ from $[0,\infty]$ the random set $N\cap (s,t)$ admits a $\overline{\sigma}(C_u-C_s:u\in (s,t))$-measurable enumeration (the bar over $\sigma$ indicating completion w.r.t. $\PPP$),  $N\cap (h,\infty)=h+N(\Delta_h)$ a.s.-$\PPP$ for all $h\in [0,\infty)$, $N\cap (h,\infty)=h+N(\Delta_h)$ for all $h\in [0,\infty)$). 
\end{definition}
These are just the natural analogues of our notions in the one-sided setting. For reasons which will become clear in due course it is somewhat more natural to work with subsets of $(0,\infty)$ rather than $[0,\infty)$. The perfect stationarity property is reminiscent of the ``homogeneity'' property for $N$ in the Markovian setting, whereby $N\cap (h,\infty)=h+N(\theta_h)$ for all $h\in [0,\infty)$ with $\theta=(\theta_h)_{h\in [0,\infty)}$ being the usual Markov shifts on $\Theta_0$, see e.g. \cite[(2.2)]{Maisonneuve}. As in the two-sided landscape we shall not care much about distinguishing between $N:\Theta_0\to 2^{(0,\infty)}$ and its associated subset $\llbracket N\rrbracket:=\{(\omega,t)\in \Theta_0\times (0,\infty):t\in N(\omega)\}$ of $\Theta_0\times (0,\infty)$. Since there is some potential conflict as to what we intend when we say ``random (countable) set'' --- it may be a map from $\mathbb{R}^{\Omega_0}$ or indeed from $(0,\infty)^{\Theta_0}$ --- we agree that in absence of further qualification the meaning of it being  a random set living on $\Omega_0$ (not $\Theta_0$) shall prevail.

\begin{proposition}\label{proposition:two-sided-viz-one-sided}

The following are equivalent for a random set $M:\Omega_0\to 2^\mathbb{R}$.
\begin{enumerate}[(A)]
\item\label{proposition:two-sided-viz-one-sided:A} $M$ is a stationary local random countable set.
\item\label{proposition:two-sided-viz-one-sided:B} $M\cap (u,\infty)=u+N(\Delta_u\vert_{[0,\infty)})$ a.s.-$\WW$ for all $u\in \mathbb{R}$ for some  local stationary random countable set $N:\Theta_0\to 2^{(0,\infty)}$.
\end{enumerate}
For a given $M$ the set $N$ is $\PPP$-a.s. unique and may be chosen to be perfectly stationary, from $ \mathcal{B}_{\Theta_0}\otimes \mathcal{B}_{(0,\infty)}$ and countable with certainty. In particular, always, $N\cap (U,\infty)=U+N(\Delta_U)$ a.s.-$\PPP$ on $\{U<\infty\}$ for any $\mathcal{U}$-stopping time $U$. 
\end{proposition}
\begin{definition}\label{definition:hat-M}
We denote (any version of) $N$ from the preceding proposition by $\widehat M$.
\end{definition}

\begin{proof}
Remark that $B\vert_{[0,\infty)}=\Delta_0\vert_{[0,\infty)}$ and that the push-forward of $\WW$ under the latter map, and moreover under any of the $\GG/\HH$-measurable maps $\Delta_u\vert_{[0,\infty)}$, $u\in \mathbb{R}$, is $\PPP$.

Condition \ref{proposition:two-sided-viz-one-sided:B} is necessary for \ref{proposition:two-sided-viz-one-sided:A}. For, by locality of $M$, $M\cap (0,\infty)$ admits an $\FF_{0,\infty}$-measurable enumeration.  It means that there is a sequence  $S=(S_i)_{i\in \mathbb{N}}$ of $\mathcal{H}$-measurable random variables with values in $(0,\infty)^\dagger$, which satisfies $M\cap (0,\infty)=[S(B\vert_{[0,\infty)})]$ a.s.-$\WW$. Setting $N:=[S]$  we get a random countable set $N:\Theta_0\to 2^{(0,\infty)}$ satisfying $M\cap (0,\infty)=N(\Delta_0\vert_{[0,\infty)})$ a.s.-$\WW$. Then, by stationarity of $M$,  for all $u\in \mathbb{R}$,  a.s.-$\WW$, $M\cap (u,\infty)=u+((M-u)\cap (0,\infty))=u+(M(\Delta_u)\cap (0,\infty))=u+N(\Delta_u\vert_{[0,\infty)})$; in particular $N(\Delta_0\vert_{[0,\infty)})\cap (h,\infty)=M\cap (h,\infty)=h+N(\Delta_h\vert_{[0,\infty)})=h+N(\Delta_h(\Delta_0\vert_{[0,\infty)}))$  a.s.-$\WW$ for all $h\in [0,\infty)$, i.e. $N\cap (h,\infty)=h+N(\Delta_h)$  a.s.-$\PPP$ for all $h\in [0,\infty)$, so  that $N$ is  stationary. Also, for $s<t$ from $[0,\infty]$, $N(\Delta_0\vert_{[0,\infty)})\cap (s,t)=M\cap (s,t)$ a.s.-$\WW$, which means, by locality of $M$, that there is a a sequence  $T=(T_i)_{i\in \mathbb{N}}$ of $\FF_{s,t}$-measurable random variables such that $N(\Delta_0\vert_{[0,\infty)})\cap (s,t)=[T]$ a.s.-$\WW$, in other words $N\cap (s,t)$ admits a $\overline{\sigma}(C_u-C_s:u\in (s,t))$-measurable enumeration. Therefore $N$ is local.

Condition \ref{proposition:two-sided-viz-one-sided:B} is sufficient for \ref{proposition:two-sided-viz-one-sided:A}. Indeed, for extended-real $ t>s>-\infty$, $M\cap (s,t)=(M\cap (s,\infty))\cap (s,t)=s+[N(\Delta_s\vert_{[0,\infty)})\cap (0,t-s)]$ a.s.-$\WW$. But $N$ is local. Therefore there is a $\overline{\sigma}(C_v:v\in (0,t-s))$-measurable sequence $S=(S_k)_{k\in \mathbb{N}}$ with values in $(0,\infty)^\dagger$ such that $N\cap (0,t-s)=[S]$ a.s.-$\PPP$. In other words, $N(\Delta_s\vert_{[0,\infty)})\cap (0,t-s)$ and hence $M\cap (s,t)$ admits an $\FF_{s,t}$-measurable enumeration. It means (letting  $s\downarrow -\infty$ over the integers, say) that $M$ is a local random countable set. As for stationarity, let $u\in \mathbb{R}$. For all $v\in \mathbb{R}$ we have, a.s.-$\WW$, $(u+M(\Delta_u))\cap (v,\infty)=u+(M(\Delta_u)\cap (v-u,\infty))=u+(v-u+N(\Delta_{v}\vert_{[0,\infty)}))=v+N(\Delta_{v}\vert_{[0,\infty)})=M\cap (v,\infty)$. Letting $v\downarrow -\infty$ over the integers, say, shows that $u+M(\Delta_u)=M$ a.s.-$\WW$, as required.

For the uniqueness of $N$ just note that $M\cap (0,\infty)=N(\Delta_0\vert_{[0,\infty)})$ a.s.-$\WW$. That $N$ may be chosen perfectly stationary, countable with certainty and belonging to $\mathcal{B}_{\Theta_0}\otimes \mathcal{B}_{(0,\infty)}$ follows by the following  adaptation of the proof of Proposition~\ref{proposition:version-perfect-stationary}, which we will keep more brief in the parts which are essentially verbatim the same to the two-sided case. The argument is valid for any  stationary random countable set $N:\Theta_0\to 2^{(0,\infty)}$.

Changing $N$ on a $\PPP$-negligible set we may and do ask that there is a $\mathcal{B}_{\Theta_0}$-measurable enumeration $S$ thereof such that $N=[S]$ with certainty (not just $\PPP$-a.s.), a perfect enumeration. Define:
$$R:=\{(s,\omega)\in [0,\infty)\times \Theta_0:N(\omega)\cap (s,\infty)=s+N(\Delta_s\omega)\},$$
$$\Theta_1:=\{\omega\in \Theta_0:(s,\omega)\in R\text{ for $\mathscr{L}$-a.e. }s\in [0,\infty)\}\text{ and }$$
$$\Theta_2:=\{\omega\in \Theta_0:\Delta_u\omega\in \Theta_1\text{ for $\mathscr{L}$-a.e. }u\in  [0,\infty)\}.$$
Then $R\in  \mathcal{B}_{[0,\infty)}\otimes\mathcal{B}_{\Theta_0}$,  $\Theta_1\in \mathcal{B}_{\Theta_0}$, $\PPP(\Theta_1)=1$, $\Theta_2\in \mathcal{B}_{\Theta_0}$ is shift-closed and $\PPP(\Theta_2)=1$. 

Define, for $\omega\in \Theta_0$, $$(0,\infty)\supset \tilde{N}(\omega):=
\begin{cases}
s+N(\Delta_s\omega)& \text{on $(s,\infty)$ for any $s\in [0,\infty)$ for which $\Delta_s\omega\in \Theta_1$, if $\omega\in \Theta_2$;}\\
\cup_{v\in (0,\infty),\Delta_v\omega\in \Theta_2}(v+\tilde N(\Delta_v\omega)) & \text{if $\omega\notin\Theta_2$}
\end{cases}
$$ (the empty union is understood as being equal to $\emptyset$, of course). Since for $\omega\in \Theta_2$, $\Delta_s\omega\in \Theta_1$ for $\mathscr{L}$-a.e. $s\in [0,\infty)$, hence for arbitrarily small $s\in (0,\infty)$, thus there is at most one $\tilde N(\omega)\subset (0,\infty)$ satisfying the above. To see that there is at least one, let  $\{s,u\}\subset [0,\infty)$ be such that $\{\Delta_s\omega,\Delta_u\omega\}\subset \Theta_1$. Then (*) for $\mathscr{L}$-a.e. $v\in[0,\infty)$ we have $s+N(\Delta_s\omega)=s+v+N(\Delta_{s+v}\omega)$ on $(s+v,\infty)$; at the same time for $\mathscr{L}$-a.e. $w\in [0,\infty)$ we have $u+N(\Delta_u\omega)=u+w+N(\Delta_{u+w}\omega)$ on $(u+w,\infty)$. Without loss of generality $u\leq s$; then (**) for $\mathscr{L}$-a.e. $v\in [0,\infty)$ we have $u+N(\Delta_u\omega)=u+(s-u)+v+N(\Delta_{u+(s-u)+v}\omega)$ on $(u+(s-u)+v,\infty)$. Combining (*) and (**) we get that  for $\mathscr{L}$-a.e. $v\in[0,\infty)$ and thus for arbitrarily small $v\in (0,\infty)$ we have $s+N(\Delta_s\omega)=s+v+N(\Delta_{s+v}\omega)=u+N(\Delta_u\omega)$ on $(s+v,\infty)$, whence  $s+N(\Delta_s\omega)=u+N(\Delta_u\omega)$ on $(s,\infty)$.
 Thus $\tilde{N}$ is well-defined.

Let $\omega\in \Theta_2$ and $u\in [0,\infty)$; of course then $\Delta_u\omega\in \Theta_2$ also. For arbitrarily small $s\in (0,\infty)$  and  $v\in (0,\infty)$ we have $\Delta_s\omega\in \Theta_1$, $\Delta_{u+v}\omega\in \Theta_1$ and
$$\tilde{N}(\omega)=s+N(\Delta_s\omega)\text{  on $(s,\infty)$ and }u+\tilde{N}(\Delta_u\omega)=u+v+N(\Delta_{u+v}\omega)\text{ on $(u+v,\infty)$}.$$ In turn, fixing such $s$ and $v$, for $\mathscr{L}$-a.e. $z\in [0,\infty)$ and for $\mathscr{L}$-a.e. $z'\in [0,\infty)$, 
\begin{align*}
s+N(\Delta_s\omega)&=s+z+N(\Delta_{s+z}\omega)\text{ on $(s+z,\infty)$ and }\\
u+v+N(\Delta_{u+v}\omega)&=u+v+z'+N(\Delta_{u+v+z'}\omega)\text{ on $(u+v+z',\infty)$}.
\end{align*}
From the latter it follows  that $s+N(\Delta_s\omega)=u+v+N(\Delta_{u+v}\omega)$ on $(s\lor (u+v),\infty)$. 
Therefore $\tilde{N}(\omega)=u+\tilde{N}(\Delta_u\omega)$ on $(s\lor (u+v),\infty)$. Letting $s$ and $v$ descend to zero we conclude that $\tilde{N}(\omega)=u+\tilde{N}(\Delta_u\omega)$ on $(u,\infty)$. This is the stationarity of $\tilde N$ on $\Theta_2$.  On the other hand, let $\omega\in\Theta_0\backslash  \Theta_2$ and still $u\in [0,\infty)$. If $\Delta_u\omega\in \Theta_2$ then $u>0$ and trivially by definition and by the stationarity of $\tilde N$ on $\Theta_2$, $\tilde N(\omega)\cap (u,\infty)=u+\tilde N(\Delta_u\omega)$; if $\Delta_u\omega\notin \Theta_2$, then $\Delta_v\omega\notin \Theta_2$ for all $v\in [0,u]$ and hence $u+\tilde N(\Delta_u\omega)=u+\cup_{v\in (0,\infty),\Delta_{v+u}\omega\in \Theta_2}(v+\tilde N(\Delta_{v+u}\omega))=\cup_{v\in (u,\infty),\Delta_{v} \omega\in \Theta_2}(v+\tilde N(\Delta_v\omega))=\cup_{v\in (0,\infty),\Delta_v\omega\in\Theta_2}(v+\tilde N(\Delta_v\omega))=\tilde N(\omega)=\tilde N(\omega)\cap (u,\infty)$. We see thus that $\tilde{N}$ has perfect stationarity.

$N$ being countable with certainty, so is $\tilde N$. 

Next, set \begin{equation*}Q:=\{(s,t,\omega)\in [0,\infty)\times (0,\infty)\times \Theta_0: t\in s+N(\Delta_s(\omega))\},
\end{equation*}which belongs to $\mathcal{B}_{[0,\infty)}\otimes \mathcal{B}_{(0,\infty)}\otimes \mathcal{B}_{\Theta_0}$. Let $h\in (0,\infty)$. Then for  $\omega\in \Theta_2$,  $h^{-1}\int_0^h \mathbbm{1}_Q(s,t,\omega)\mathscr{L}(\dd s)$ is equal to $1$  if $t\in \tilde N(\omega)\cap (h,\infty)$ and is equal to $0$ for $t\in (h,\infty)\backslash \tilde N(\omega)$. Therefore, by Tonelli, we deduce that $\tilde N\cap (\Theta_2\cap (h,\infty))\in \mathcal{B}_{\Theta_2}\otimes \mathcal{B}_{(h,\infty)}$. Letting $h\downarrow 0$ over a sequence it follows that $\tilde N\cap (\Theta_2\cap (0,\infty))\in \mathcal{B}_{\Theta_2}\otimes \mathcal{B}_{(0,\infty)}$. Since $\tilde N\subset (0,\infty)$ on $\Theta_2$ and since $\Theta_2$ is closed for the L\'evy shifts we have $$\tilde N(\omega)=\cup_{v\in \mathbb{Q}\cap (0,\infty),\Delta_v\omega\in \Theta_2}(v+\tilde N(\Delta_v\omega)),\quad \omega\in \Theta_0\backslash \Theta_2,$$ i.e., setting $\overline{N}$ equal to $\tilde N$ on $\Theta_2$ and to $\emptyset$ on $\Theta_0\backslash \Theta_2$, $$\tilde N=\cup_{v\in \mathbb{Q}\cap (0,\infty)}(v+\overline{N}(\Delta_v))=\cup_{v\in \mathbb{Q}\cap (0,\infty)}\{(\Delta_v,\mathrm{id}_{(0,\infty)}-v)\in \overline{N}\}\text{ on }\Theta_0\backslash \Theta_2\equiv (\Theta_0\backslash \Theta_2)\times (0,\infty),$$
where we have been very relaxed indeed about not distinguishing between maps $\Theta_0\to 2^{(0,\infty)}$ and their associated subsets of $\Theta_0\times (0,\infty)$. It now follows easily that $\tilde N\in \mathcal{B}_{\Theta_0}\otimes \mathcal{B}_{(0,\infty)}$.

It remains to check that $\tilde{N}=N$ on $\Theta_1\cap\Theta_2$ rendering $\tilde{N}$ a version of $N$. Take $\omega\in \Theta_1\cap\Theta_2$. Then because $\omega\in \Theta_2$ certainly, by the very definition of $\tilde{N}$,  $\tilde{N}(\omega)=s+N(\Delta_s\omega)$ on $(s,\infty)$ for $\mathscr{L}$-a.e. $s\in [0,\infty)$, while $s+N(\Delta_s\omega)=N(\omega)$ on $(s,\infty)$ for $\mathscr{L}$-a.e. $s\in[0,\infty)$ because $\omega\in \Theta_1$; for arbitrarily small (indeed $\mathscr{L}$-a.e.) $s\in (0,\infty)$ both of the preceding equalities prevail, concluding the argument.

Lastly, to see that $N\cap (U,\infty)=U+N(\Delta_U)$ a.s.-$\PPP$ on $\{U<\infty\}$ pass to a perfectly stationary version of $N$ (for which it holds with certainty) and then back, recalling that $(\Delta_U)_\star \PPP(\cdot\vert U<\infty)=\PPP$ by the strong Markov property of $C$ under $\PPP $ at the time $U$ (unless $\PPP(U<\infty)=0$, but the latter case is trivial). 
\end{proof}
It appears the statement of Proposition~\ref{proposition:two-sided-viz-one-sided} cannot be improved to a version in which, ceteris paribus, we would allow $0$ to belong to $N$, replace $(u,\infty)$ with $[u,\infty)$ and  in the perfect stationarity of Definition~\ref{definition:one-sided} ask for $N\cap [h,\infty)=h+N(\Delta_h)$, at least not with only trivial modifications to the proof. The reason being that such trivial modifications would presumably still not use properties of $\WW$ beyond stationary independent increments, and such conclusion is clearly false in the L\'evy setting, since the jump times of a two-sided (resp. and infinite activity) L\'evy process give an example of a stationary local (resp. and dense) random countable set $M$, whose corresponding one-sided set $\hat{M}$ is exhausted by the graphs of stopping times, a property that would be precluded under the above stipulated modifications (for, given those, we would also have $\{U\in \hat{M}\}=\{0\in \hat{M}(\Delta_U)\}\cap \{U<\infty\}$ a.s.-$\PPP$, which has probability zero by the strong Markov property for any $\UU$-stopping time $U$). This begs, however,
\begin{question}\label{question:stopping-times-belonging}
For a local stationary random countable set $N:\Theta_0\to 2^{(0,\infty)}$, can we have $\PPP(U\in N)>0$ for some $\UU$-stopping time $U$ or even that $N$ is enumerated by a sequence of stopping times? 
\end{question}
A partial answer to this (to the negative!) is given in Theorem~\ref{theorem:stopping-times-no}. Another important observation we make into
\begin{remark}\label{rmk:one-sided-perfection}
The proof   of Proposition~\ref{proposition:two-sided-viz-one-sided} has shown that  any   stationary random countable set $N:\Theta_0\to 2^{(0,\infty)}$ admits a version that is perfectly stationary, countable everywhere and belongs to $\mathcal{B}_{\Theta_0}\otimes \mathcal{B}_{(0,\infty)}$. When  $N$ has the first two of the preceding properties we can construct $O:=\cup_{h\in \mathbb{R}}(h+N(\Delta_h\vert_{[0,\infty)})):\Omega_0\to2^\mathbb{R}$, which manifestly is a perfectly stationary  random countable set having the property that  for each $a\in \mathbb{R}$, $O\cap (a,\infty)=a+N(\Delta_a\vert_{[0,\infty)})$  admits a $(\Delta_a \vert_{[0,\infty)})^{-1}(\HH)$-measurable perfect enumeration. If $N$ has even all three of the preceding properties then we get in addition that $O\cap (a,\infty)\in (\Delta_a\vert_{[0,\infty)})^{-1}(\mathcal{B}_{\Theta_0})\otimes \mathcal{B}_{\mathbb{R}}$ for all $a\in \mathbb{R}$ and in particular $O\in \mathcal{B}_{\Omega_0}\otimes \mathcal{B}_{\mathbb{R}}$.
\end{remark}
\begin{definition}
We denote $O$ from the preceding remark by $\widecheck{N}$.
\end{definition}
\begin{corollary}\label{corollary:nice-version}
Any local stationary random countable set $M:\Omega_0\to 2^\mathbb{R}$ has a version which is perfectly stationary, countable with certainty  and for which each $M\cap (a,\infty)$ admits a $(\Delta_a\vert_{ [0,\infty)})^{-1}(\HH)$-measurable perfect enumeration  (a priori it has only an $\FF_{a,\infty}$-measurable enumeration), this for all $a\in \mathbb{R}$.
\end{corollary}
\begin{proof}
Choose a perfectly stationary everywhere countable version of $\widehat{M}$ and take $\widecheck{\widehat{M}}$.
\end{proof}
It would not have escaped the reader that had we not proved the ``two-sided'' Proposition~\ref{proposition:version-perfect-stationary}, earlier, then it would follow at once by the preceding technique from the ``one-sided'' version, which was proved independently (recall Remark~\ref{rmk:one-sided-perfection}). But, the two-sided variant is somewhat easier to digest, and so we felt it justifiable to  give that one first. We will have occasion to apply Corollary~\ref{corollary:nice-version} in due course. The  added value thereof, versus Proposition~\ref{proposition:version-perfect-stationary}, is that one has the perfect stationarity property \emph{combined with} the fact that $M\cap (0,\infty)$ admits a $(B\vert_{ [0,\infty)})^{-1}(\HH)$-measurable perfect enumeration.

We close this section by establishing that local stationary random countable sets are precisely the visiting times of a measurable set $A\subset  \Omega_0$ belonging to the germ $\sigma$-field around zero by the path-valued process $\Delta$, all of this in a sense that we make precise at once.

\begin{definition}\label{definition:MA-set}
For $A\subset \Omega_0$ put $$M^A:=\{t\in \mathbb{R}:\Delta_t\in A\}=\{\Delta\in A\}$$ for the visiting set to $A$ of the process $\Delta$.
\end{definition}

\begin{example}
Taking $A:=\{0\text{ a local minimum}\}$ (resp. $A:=\{0\text{ a local maximum}\}$, $A:=\{0\text{ a local extremum}\}$) yields for $M^A$ the local minima (resp. maxima, extrema).
\end{example}

\begin{remark}\label{remark:MA-equal}
$M^A$ is perfectly stationary. If $\{A_1,A_2\}\subset 2^{\Omega_0}$ and $A_1=A_2$ off a shift-invariant $\WW$-negligible set $R$, then $M^{A_1}=M^{A_2}$ off $R$, hence a.s.-$\WW$. More generally,  for $\{A_1,A_2\}\subset 2^{\Omega_0}$,  $M^{A_1}=M^{A_2}$ a.s.-$\WW$ iff $A_1\triangle A_2$ is ``polar'' for (is a.s.-$\WW$ not visited by) $\Delta$.
\end{remark}

\begin{definition}
We let $\mathfrak{G}:=\cap_{\epsilon\in (0,\infty)}\sigma(B\vert_{(-\epsilon,\epsilon)})\subset \mathcal{B}_{\Omega_0}$ (no completions!) denote the germ $\sigma$-field around zero. 
\end{definition}

\begin{theorem}\label{theorem:construction-through-zero-time-section}
Let $M:\Omega_0\to 2^\mathbb{R}$. The following are equivalent. 
\begin{enumerate}[(A)]
\item\label{cross-section:A} $M$ is a local stationary random countable set.
\item\label{cross-section:B} There exists an  $A\in \mathfrak{G}$ such that $M^A$ is countable a.s.-$\WW$ and such that $M=M^A$ a.s.-$\WW$.
\end{enumerate}
\end{theorem}
\begin{remark}\label{remark:onlyAzero}
In Item~\ref{cross-section:B}, since $A\in \mathfrak{G}$, by Kolmogorov's zero-one law, $\WW(A)\in \{0,1\}$. Only the case when $\WW(A)=0$ is actually possible. For, if $\WW(A)=1$, then by Tonelli and the fact that the L\'evy shifts are measure-preserving for $\WW$, a.s.-$\WW$ for $\mathfrak{l}$-almost every (therefore, for uncountably many) $t\in \mathbb{R}$, $\Delta_t\in A$ (which contradicts the stipulation of \ref{cross-section:B} that $M^A$ be countable $\WW$-a.s.).
\end{remark}
\begin{proof}
First, the easy direction: \ref{cross-section:B} $\Rightarrow$ \ref{cross-section:A}. We may just as well assume that $M=M^A$ with certainty (since the properties of \ref{cross-section:A} are not affected by changing $M$ on a $\WW$-negligible set). Then, since $M^A$ may be viewed as the preimage of $A$ under the $( \mathcal{B}_{\Omega_0}\otimes \mathcal{B}_\mathbb{R})/\mathcal{B}_{\Omega_0}$-measurable map $(\omega,t)\mapsto \Delta_t(\omega)$ we get that $M^A$ is $ \mathcal{B}_{\Omega_0}\otimes \mathcal{B}_\mathbb{R}$-measurable. By Remark~\ref{remark:ranomd-countable-set} $M^A$ is a random countable set. As already noted in Remark~\ref{remark:MA-equal}, directly by construction $M^A$ is (even perfectly) stationary. Further, for extended real $s<t$, and then any $h\in (0,\infty)$ for which $s+h<t-h$, we may pick an $A'\in \mathcal{B}_{\Omega_0\vert_{(-h,h)}}$ such that $A=(B\vert_{(-h,h)})^{-1}(A')$; viewing $M^A\cap (s+h,t-h)$ as the pre-image of $A'$ under the $(\sigma(\Delta_s\vert_{(0,t-s)})\otimes \mathcal{B}_{(s+h,t-h)})/\mathcal{B}_{\Omega_0\vert_{(-h,h)}}$-measurable map $(\omega,u)\mapsto \Delta_u(\omega)\vert_{(-h,h)}$ we get by the very same token of  \cite[Theorem~3.2]{kendall_2000} employed in Remark~\ref{remark:ranomd-countable-set} that $M^A\cap (s+h,t-h)$ admits an $\FF_{s,t}$-measurable enumeration. Taking union over a sequence of $h$ descending to zero allows to conclude that $M^A$  is also local. So, all in all, $M^A$ is a stationary local random countable set. 

Now for the difficult part: \ref{cross-section:A} $\Rightarrow$ \ref{cross-section:B}. Return to the construction of the perfectly stationary version of $M$ of the proof of Proposition~\ref{proposition:version-perfect-stationary}. Fix a a sequence $(h_n)_{n\in \mathbb{N}}$ in $(0,\infty)$ descending to $0$.  Just before introducing $R$ we ask in addition (as we may, by further changes to $M$ on  a $\WW$-negligible set) that $M\cap (\{h_n,-h_n:n\in \mathbb{N}\}\cup \{0\})=\emptyset$ and that for all $n\in \mathbb{N}$, $M\cap [(-h_n,-h_{n+1})\cup (h_{n+1},h_n)]$ admits a perfect enumeration measurable relative to the $\sigma$-field generated by the increments of $B$ on $(-h_n,-h_{n+1})$ and by its increments on $(h_{n+1},h_n)$ (no completions!). Note that as a consequence, for each $n\in \mathbb{N}$, $M\cap (-h_n,h_n)$ admits a perfect enumeration measurable relative the increments of $B$ on $(-h_n,h_n)$. With this extra requirement having been made, continue to introduce $R$, $\Omega_1$, $\Omega_2$ and $\tilde M$ exactly as in the proof of Proposition~\ref{proposition:version-perfect-stationary}. Then recall that $\tilde M$ is perfectly stationary and belongs to $\mathcal{B}_{\Omega_0}\otimes \mathcal{B}_{\mathbb{R}}$, besides, $\Omega_2$ is shift-invariant and $\WW$-almost certain. Furthermore, for each $n\in \mathbb{N}$ and then for every $\epsilon\in (0,h_n/2)$, 
 $$Q^\epsilon:=\{(s,t,\omega)\in (-\epsilon,\epsilon)\times (-h_n+2\epsilon,h_n-2\epsilon)\times \Omega_0: t\in s+M(\Delta_s(\omega))\}$$
 belongs to $\mathcal{B}_{(-\epsilon,\epsilon)}\otimes \mathcal{B}_{(-h_n+2\epsilon,h_n-2\epsilon)}\otimes \sigma(B\vert_{(-h_n,h_n)})$; and, for $\omega\in \Omega_2$,  $\frac{1}{2\epsilon}\int_{-\epsilon}^\epsilon \mathbbm{1}_{Q^\epsilon}(s,t,\omega)\lll(\dd s)$ is equal to $1$  if $t\in \tilde M(\omega)\cap (-h_n+2\epsilon,h_n-2\epsilon)$ and is equal to $0$ for $t\in (-h_n+2\epsilon,h_n-2\epsilon)\backslash \tilde M(\omega)$. Accordingly, by Tonelli, $M$ restricted to $\Omega_2\times (-h_n+2\epsilon,h_n-2\epsilon)$ belongs to $\sigma(B\vert_{(-h_n,h_n)})\vert_{\Omega_2}\otimes \mathcal{B}_{(-h_n+2\epsilon,h_n-2\epsilon)}$. Letting $\epsilon\downarrow 0$ over a sequence we deduce that $M$ restricted to $\Omega_2\times (-h_n,h_n)$ belongs to $\sigma(B\vert_{(-h_n,h_n)})\vert_{\Omega_2}\otimes \mathcal{B}_{(-h_n,h_n)}$. 
 
 With this nice version of $\tilde{M}$ in hand, assume without loss of generality that $M=\tilde{M}$ to begin with, and put 
 $${A^0}:=\{\Delta_t\omega:(\omega,t)\in \Omega_0\times \mathbb{R}\backepsilon t\in M(\omega)\}.$$ 
 
 Then $M=M^{A^0}$ even with certainty. For, let $\omega\in \Omega_0$. Clearly $M^{A^0}(\omega)\supset M(\omega)$. For the reverse inclusion take $t\in M^{A^0}(\omega)$. Since $\Delta_t\omega\in {A^0}$, there are $\omega'\in\Omega_0$ and a $t'\in M(\omega')$ such that $\Delta_t\omega=\Delta_{t'}\omega'$, i.e. $\omega=\Delta_{t'-t}\omega'$. But by perfect stationarity $M( \omega')=t'-t+M(\Delta_{t'-t}\omega')=t'-t+M(\omega)$. Thus $t+t'\in t+M(\omega')=t'+M(\omega)$. Therefore  $t\in M(\omega)$.

 Next, since  $M\in \mathcal{B}_{\Omega_0}\otimes \mathcal{B}_\mathbb{R}$ we get that ${A^0}=\{0\in M\}$, the zero-time section of $M$, belongs to $\mathcal{B}_{\Omega_0}$. Hence $E:={A^0}\cap \Omega_2\in \mathcal{B}_{\Omega_2}$.  Fix $n\in \mathbb{N}$. For $\{\omega_1,\omega_2\}\subset\Omega_2$ belonging to the same atom of $\sigma(B\vert_{(-h_n,h_n)})\vert_{\Omega_2}$, for all $t\in \mathbb{R}$, $\omega_1$ is from the $t$-time section of $M\cap (-h_n,h_n)$ iff $\omega_2$ is from the $t$-time section of $M\cap (-h_n,h_n)$ ($\because$  $M$ restricted to $\Omega_2\times (-h_n,h_n)$ belongs to $\sigma(B\vert_{(-h_n,h_n)})\vert_{\Omega_2}\otimes \mathcal{B}_{(-h_n,h_n)}$). Thus $[M\cap (-h_n,h_n)](\omega_1)=[M\cap (-h_n,h_n)](\omega_2)$. Therefore $\omega_1\in {A^0}$ iff $\omega_2\in {A^0}$. So, $E$ is a union of the atoms of the separable sub-$\sigma$-field $\sigma(B\vert_{(-h_n,h_n)})\vert_{\Omega_2}$ of $\mathcal{B}_{\Omega_2}$. By Blackwell's theorem  \cite[T~III.17]{meyer} we infer that $E\in \sigma(B\vert_{(-h_n,h_n)})\vert_{\Omega_2}$ (the definition of a Blackwell space  \cite[D~III.15]{meyer} entails that not only is $(\Omega_0,\mathcal{B}_{\Omega_0})$ Blackwell [as a Polish space \cite[T~III.16]{meyer}], but so is its trace $(\Omega_2,\mathcal{B}_{\Omega_2})$ on $\Omega_2\in \mathcal{B}_{\Omega_0}$). This being true for all $n\in \mathbb{N}$, we get $E\in \cap_{n\in \mathbb{N}} [\sigma(B\vert_{(-h_n,h_n)})\vert_{\Omega_2}]$. But $\cap_{n\in \mathbb{N}} [\sigma(B\vert_{(-h_n,h_n)})\vert_{\Omega_2}] =\mathfrak{G}\vert_{\Omega_2}$ by Lemma~\ref{lemma:commute-decreasing-intersection-trace}. Hence there is $A\in \mathfrak{G}$ such that $A\cap \Omega_2=E$. It remains to note that $M^{A^0}=M^E=M^{A}$ a.s.-$\WW$ (recall Remark~\ref{remark:MA-equal}).
\end{proof}

By way of a  check (of the sort: does it imply something, which is obviously false), we may notice that the proof of Theorem~\ref{theorem:construction-through-zero-time-section}, namely the part concerning the perfection of $M$,  has --- essentially, modulo elementary extra considerations, which we leave to the reader --- shown, that a local stationary random countable set $M:\Omega_0\to 2^\mathbb{R}$ admits a version $\tilde M\in\mathcal{B}_{\Omega_0}\otimes \mathcal{B}_\mathbb{R}$ that is perfectly stationary, countable with certainty and for which there is a shift-invariant $\Omega_2\in \mathcal{B}_{\Omega_0}$ of full $\WW$-measure such that for all extended-real $s<t$, $$\tilde M\cap (\Omega_2\times (s,t))\in\sigma (\Delta_s\vert_{(0,t-s)})\vert_{\Omega_2}\otimes \mathcal{B}_{(s,t)}.$$ It is perfect stationarity, combined with a kind of almost perfect locality (``almost'', since  $\Omega_2$ is not necessarily $\Omega_0$; ``kind of'', since it is to do with measurability in the tensor product, not enumerations). Can we ``believe it''? If we trust the proof, we had better. Still, being prudent is the mother of all virtues. So, is it true for $M=\{\text{local minima of }B\}$? Indeed it is, for one can take for $\Omega_2$ the set of  paths from $\Omega_0$, which satisfy the property that at no two distinct local minima does $B$ take the same value (which belongs to $\mathcal{B}_{\Omega_0}$ by ``rational exhaustion'' and is known to be $\WW$-almost certain). On this shift-invariant  $\Omega_2$ all of the local minima are strict, hence one may further take $\tilde M$ equal to the local minima on $\Omega_2$ and to $\emptyset$ off $\Omega_2$. Such $\tilde M$ is evidently countable with certainty and perfectly stationary, but also satisfies the locality property of the above display.

\section{New examples of stationary local random countable sets over the Wiener noise}\label{section:new-family-of-examples}
We procure in this section a positive answer to Tsirelson's question as announced in the Introduction. Before doing so let us first mention some would-be ``obvious candidates'' of exceptional times of $B$, which  would --- in view  of Theorem~\ref{theorem:construction-through-zero-time-section},  \ref{cross-section:B} $\Rightarrow$ \ref{cross-section:A} --- give new examples of dense stationary local random countable sets, were it not for the fact that they fail to be $\WW$-a.s. countably infinite.

\begin{itemize}
\item (Points of increase.) For $A=\{\exists \epsilon>0\backepsilon \inf_{s\in [0,\epsilon]}B_s\land (-B_{-s})\geq 0\}$ it was shown in \cite{nonincrease} that $M^A$ is empty.
\item (Fast points.) Let $\alpha\in [0,\infty)$. For $A=\{\limsup_{h\downarrow  0}\frac{B(h)}{\sqrt{2 h \log  h^{-1}}}\geq \alpha\}$ it was shown in \cite{fast-points} that $M^A$ is a.s.-$\WW$ of full Lebesgue measure ($\therefore$ of cardinality continuum, $\therefore$ uncountable), empty or of positive Hausdorff measure ($\therefore$ uncountable) according as $\alpha=1$, $\alpha>1$ or $\alpha<1$. (The case $\alpha\leq 1$ is anyway automatically precluded from being interesting for us by the fact that $\WW(A)=1$ in such case, by the law of the iterated logarithm (recall Remark~\ref{remark:onlyAzero}).)
\item (Slow points.) Let $\alpha\in [0,\infty)$. For $A=\{\exists \epsilon>0\backepsilon \vert B(h)\vert\leq \alpha\sqrt{h}\text{ for }h\in [0,\epsilon]\}$ the set $M^A$ is empty when $\alpha\leq 1$ \cite{slow-points} and of positive Hausdorff measure ($\therefore$ uncountable) a.s.-$\WW$ when $\alpha>1$ \cite[Corollary~3, Proposition~1(a-b)]{slow-perkins}.
\end{itemize}

We feel that the odds are stacked against us. Nevertheless, we shall be able to construct new examples of stationary local random countable sets from the zero sets of solutions to some stochastic differential equations (SDEs). The argument leading from the second to the first can be made general and we record this first; $\Gamma$ below corresponds to such a zero set. 

In the formulation of the next couple of results we do something unusual in that we indicate in part what is used in the proof already in the statements themselves. We do so to stress the relevance of the individual properties.
\begin{lemma}\label{lemma:gamma}
Let $\Gamma\subset \Omega_0\times [0,\infty)$ satisfy the following: 
\begin{enumerate}[(a)]
\item\label{lemma:gamma:a} it $\WW$-a.s. contains $0$; 
\item \label{lemma:gamma:c}it is $\WW$-a.s. closed in the upper limit topology; 
\item\label{lemma:gamma:b} it is progressive in $\FF^{0,\rightarrow}$; 
\item \label{lemma:gamma:d} it is  ``coalescent'' in the sense that for all real $u\geq 0$   a.s.-$\WW$ for all $s$, if $s\in \Gamma\cap [u,\infty)$, then  $\Gamma=u+\Gamma(\Delta_u)$  on $[s,\infty)$.
\end{enumerate} 

For real $s\leq t$ put $$g_{s,t}:=s+\sup (\Gamma(\Delta_s)\cap [0,t-s])\in [s,t]\quad (\sup\emptyset:=0).$$ 

We have the following assertions.
\begin{enumerate}[(i)]
\item\label{lemma:gamma:i} For all real $s\leq t$ a.s.-$\WW$ the supremum in the definition of $g_{s,t}$  is attained (because of \ref{lemma:gamma:a}-\ref{lemma:gamma:c}). 
\item\label{lemma:gamma:ii} For all real $s\leq t$ the random variable $g_{s,t}$ is $\FF_{s,t}$-measurable (because of \ref{lemma:gamma:b}).
\item\label{lemma:gamma:iii} For  all real $ s_1\leq s_2\leq t_2\leq t_1$ a.s.-$\WW$, if $g_{s_1,t_1}\in [s_2,t_2]$, then $g_{s_2,t_2}=g_{s_1,t_1}$ (because of \ref{lemma:gamma:d}).
\end{enumerate}
\end{lemma}

\begin{remark}
If $\Gamma$ is $\WW$-a.s. closed also in the lower limit topology, then ``for all $s$, if $s\in \Gamma\cap [u,\infty)$, then  $\Gamma=u+\Gamma(\Delta_u)$  on $[s,\infty)$'' of Item~\ref{lemma:gamma:d} can  be replaced by the more succinct ``$\Gamma=u+\Gamma(\Delta_u)$ on $[\inf(\Gamma\cap [u,\infty)),\infty)$'' ($\inf\emptyset:=\infty$).
\end{remark}


\begin{proof}
\ref{lemma:gamma:i}. With $\WW$-probability one, $0\in \Gamma$ and $\Gamma$ is closed in the upper limit topology. Since $\Delta_s$ is measure-preserving for $\WW$, the same is true with $\Gamma(\Delta_s)$  replacing $\Gamma$. \ref{lemma:gamma:ii}. $\Delta_s^{-1}( \FF_{0,t-s})=\FF_{s,t}$ so progressive measurability of $\Gamma$ yields $\Gamma(\Delta_s)\cap  (\Omega_0\cap [0,t-s])\in \FF_{s,t}\otimes \mathcal{B}_{[0,t-s]}$. We may apply the D\'ebut theorem (recall that $\WW$ is complete). \ref{lemma:gamma:iii}.
 We use  \ref{lemma:gamma:d} with $u=s_2-s_1$ together with the fact that $\Delta_{s_1}$ is measure-preserving for $\WW$ to obtain that $\WW$-a.s. the following holds on $\{g_{s_1,t_1}\in [s_2,t_2]\}$. 
Because, by \ref{lemma:gamma:i}, $g_{s_1,t_1}-s_1\in \Gamma(\Delta_{s_1})\cap [s_2-s_1,\infty)$, then also $g_{s_1,t_1}-s_1-(s_2-s_1)\in \Gamma(\Delta_{s_2-s_1}\Delta_{s_1})$, i.e. $g_{s_1,t_1}-s_2
\in \Gamma(\Delta_{s_2})$. Because further $(g_{s_1,t_1}-s_1,t_1-s_1]\cap \Gamma(\Delta_{s_1})=\emptyset$, then also $(g_{s_1,t_1}-s_1,t_1-s_1]\cap ((s_2-s_1)+\Gamma(\Delta_{s_2-s_1}\Delta_{s_1}))=\emptyset$, i.e. $(g_{s_1,t_1}-s_2,t_1-s_2]\cap \Gamma(\Delta_{s_2})=\emptyset$ and a fortiori $(g_{s_1,t_1}-s_2,t_2-s_2]\cap \Gamma(\Delta_{s_2})=\emptyset$. Therefore  $g_{s_2,t_2}=g_{s_1,t_1}$.
\end{proof}

\begin{proposition}\label{proposition:gamma}
Retain the setting of Lemma~\ref{lemma:gamma} and put
$$M:=\{g_{s,t}:(s,t)\in \mathbb{Q}^2, s< t,g_{s,t}\in (s,t)\}.$$ 
We have the following assertions.
\begin{enumerate}[(i)]
\item For all extended-real $s<t$ one has that $\WW$-a.s. $M\cap (s,t)=\{g_{p,q}:(p,q)\in \QQ^2,s\leq p<q\leq t, g_{p,q}\in (p,q)\}$ (because of Lemma~\ref{lemma:gamma}\ref{lemma:gamma:iii}); therefore the random countable set $M$ is local (because of Lemma~\ref{lemma:gamma}\ref{lemma:gamma:ii}). 
\item Because of Lemma~\ref{lemma:gamma}\ref{lemma:gamma:iii}, $M$ is  stationary. 
\end{enumerate}
\end{proposition}
\begin{proof}
Only the second assertion requires further explanation. Let $u\in \mathbb{R}$. It suffices to show that $M\supset u+M(\Delta_u)$ a.s.-$\WW$ (because then, by this very token with $u\rightsquigarrow -u$, also $M\supset -u+M(\Delta_{-u})$ a.s.-$\WW$, i.e. we get the reverse inclusion /upon recalling that $\Delta_u$ is measure-preserving for $\WW$/). Fix $(p,q)\in \mathbb{Q}^2$, $p<q$. We are to show that $\WW$-a.s., if $g_{p,q}(\Delta_u)\in (p,q)$, then $u+g_{p,q}(\Delta_u)\in M$. Now, if $g_{p,q}(\Delta_u)\in (p,q)$, then for some rational pair $(s,t)$ satisfying $p+u<s<t<q+u$, we have $s<u+g_{p,q}(\Delta_u)<t$. On the other hand,  for any rational pair $(s,t)$ satisfying $p+u<s<t<q+u$,  by Lemma~\ref{lemma:gamma}\ref{lemma:gamma:iii}, $(s,t)\ni u+g_{p,q}(\Delta_u)=u+g_{s-u,t-u}(\Delta_u)$ a.s.-$\WW$ on $\{s<u+g_{p,q}(\Delta_u)<t\}$, while from the very definitions $u+g_{s-u,t-u}(\Delta_u)=g_{s,t}$. Hence $u+g_{p,q}(\Delta_u)\in M$ a.s.-$\WW$ on $\{g_{p,q}(\Delta_u)\in (p,q)\}$.
\end{proof}

\begin{example}
The random set $\Gamma=\{B\vert_{[0,\infty)}-\underline{B\vert_{[0,\infty)}}=0\}$ (the underline signifies the running infimum) satisfies the conditions of Lemma~\ref{lemma:gamma} and $\WW$-a.s. the associated $M$ of Proposition~\ref{proposition:gamma} is the set of  the local minima of $B$. Similarly we  get the local maxima (taking $\Gamma=\{\overline{B\vert_{[0,\infty)}}-B\vert_{[0,\infty)}=0\}$, the overline signifying the running supremum). 
Finally, for  $\Gamma=\{(B\vert_{[0,\infty)}-\underline{B\vert_{[0,\infty)}})(\overline{B\vert_{[0,\infty)}}-B\vert_{[0,\infty)})=0\}$ the corresponding set $M$ are  the local extrema. Though, the latter $\Gamma$ is not coalescent.
\end{example}
\begin{example}
The random sets $\Gamma=\Omega_0\times \{0\}$ and $\Gamma=\Omega_0\times [0,\infty)$ trivially satisfy  the conditions of Lemma~\ref{lemma:gamma},  however the associated $M$ of Proposition~\ref{proposition:gamma} is also trivial, empty.
\end{example}

We proceed now to the promised construction of new  examples of stationary local random countable sets. To this end fix a $d \in [0,\infty)$, and recall \cite[Section~XI.1]{revuz} that the squared Bessel SDE (for the unknown $Z=(Z(t))_{t\in [0,\infty)}$)
\begin{equation}
\label{sde}
Z(t)=  2\int_0^t \sqrt{Z(s)}\dd B(s)+ dt,\quad t\in [0,\infty),
\end{equation}
admits an a.s.-$\WW$ unique continuous, $[0,\infty)$-valued, $\FF^{0,\rightarrow}$-adapted (i.e. pathwise unique, strong) solution. The solution $Z$ is said to have the law of the squared Bessel process of dimension $d$.

\begin{example}\label{example:reflected-loc-min}
For $d=1$ we get $Z=(B\vert_{[0,\infty)}-\underline{B\vert_{[0,\infty)}})^2$ a.s.-$\WW$ (by It\^o for the continuous semimartingale $B\vert_{[0,\infty)}-\underline{B\vert_{[0,\infty)}}$). The case $d=0$ is trivial: $Z=0$.
\end{example}

Consider next  real times  $s_1\leq s_2$, and the a.s.-$\WW$ unique continuous, $[0,\infty)$-valued, resp.  $(\FF_{s_1,t})_{t\in [s_1,\infty)}$- and $(\FF_{s_2,t})_{t\in [s_2,\infty)}$-adapted processes $Z_1=(Z_1(t))_{t\in [s_1,\infty)}$ and $Z_2=(Z_2(t))_{t\in [s_2,\infty)}$ solving $$Z_i(t)=2\int_{s_i}^t \sqrt{Z_i(s)}\dd B(s)+ d(t-s_i),\quad t \in [s_i,\infty),\, i\in \{1,2\}.$$ We identify $Z_1(s_1+\cdot)=Z(\Delta_{s_1})$ and $Z_2(s_2+\cdot)=Z(\Delta_{s_2})$ a.s.-$\WW$. Notice that $S:=\inf \{t\in [s_2,\infty):Z_1(t)=Z_2(t)\}$ is a stopping time of $(\mathcal{F}_{s_1,t})_{t\in [s_2,\infty)}$ so by the strong Markov property $Z_1=Z_2$ on $[S,\infty)$ a.s.-$\WW$. Because the paths of $Z_2$ and $Z_1$ are nonnegative and continuous and since $\WW(Z_2(s_2)=0)=1$ we see that also $\WW(Z_2\leq Z_1\text{ on }[s_2,\infty))=1$. Hence, $S\leq \inf \{t\in [s_2,\infty):Z_1(t)=0\}$ a.s.-$\WW$. It follows, on taking $s_1=0$ and $s_2=u$, that the random set $$\Gamma:=\{Z=0\}$$ --- viz. the zero set of the solution to \eqref{sde} ---  satisfies \ref{lemma:gamma:d} of Lemma~\ref{lemma:gamma}, while the remaining conditions of this lemma are also clearly satisfied. (Incidentally, the preceding should make it clear why we felt it appropriate to call the property of Lemma~\ref{lemma:gamma}\ref{lemma:gamma:d} ``coalescence''.) As a consequence, Proposition~\ref{proposition:gamma} implies that the set $M$ thereof is a stationary local random countable set, empty when $d\geq 2$ since in that case $0$ is polar for $Z$ \cite[p.~442, Item~(ii)]{revuz}, empty for $d=0$ also. We retain in what follows the notation of Lemma~\ref{lemma:gamma} and Proposition~\ref{proposition:gamma} for $\Gamma=\{Z=0\}$; in particular, for real $s\leq t$, $$g_{s,t}=s+\sup (\Gamma(\Delta_s)\cap [0,t-s])$$ and
$$M=\{g_{s,t}:(s,t)\in \mathbb{Q}^2, s< t,g_{s,t}\in (s,t)\}.$$ 

It appears that Proposition~\ref{proposition:gamma} is quite widely applicable despite its facile nature:
\begin{remark}\label{remark:generalize-sde}
In arguing the properties of the set $\Gamma$ (hence that $M$ is a stationary local random countable set) we only used the fact that the SDE \eqref{sde} of the type
\begin{equation}\label{general-sde}
\dd Z_t=\mu(Z_t)\dd t+\sigma(Z_t)\dd B_t,\quad Z_0=0
\end{equation}
($\mu$ and $\sigma$ Borel, $Z$ continuous adapted) has a nonnegative pathwise unique strong solution. More precisely, what was used (for instance) is that the process $Z$ is:  a.s.-$\WW$ nonnegative, continuous and vanishing-at-zero; adapted to $\FF^{0,\rightarrow}$; and coalescent in the sense that, for real $u\geq 0$, setting  $S:=\inf \{t\in [u,\infty):Z(t)=Z(\Delta_{u})(t-u)\}$, then $\WW$-a.s. $Z=Z(\Delta_{u})(\cdot-u)$ on $[S,\infty)$. More generally, it would have been enough to have that $Z$ is: $\WW$-a.s. vanishing at zero and left-continuous; adapted to $\FF^{0,\rightarrow}$; coalescent  at zero in the sense that for all real $u\geq 0$ $\WW$-a.s. for all $s$, if $s\in \{Z=0\}\cap [u,\infty)$, then $Z=Z(\Delta_u)(\cdot-u)$ on $[s,\infty)$.
\end{remark}

Suppose henceforth that  $d \in (0,2)$. It is known 
that in this case  the distribution of $g_{s,t}$ is diffuse and carried by the open interval $(s,t)$ for all real $s<t$. One way to see the latter is by noting that (i) $0$ is regular for $Z$, which follows for instance from the construction of a weak solution to \eqref{sde} via a spatio-temporal transformation of reflecting Brownian motion \cite[Paragraph~V.48.6]{rogers2000diffusions} and that (ii) the laws of $Z_t$, $t\in (0,\infty)$, have no atoms at zero \cite[Corollary~XI.1.4]{revuz}. So $M$ is not empty, and is in fact dense. Besides, we may now write more succinctly
$$M=\{g_{s,t}:(s,t)\in \mathbb{Q}^2, s< t\}\text{ a.s.-$\WW$}.$$  
\begin{remark}\label{rmK.dense-general}
More generally, if for the pathwise unique  strong solution of the SDE \eqref{general-sde}  $0$ is regular for $Z$ and $Z$ has no atoms at zero at positive times, then the $M$ associated to $\Gamma=\{Z=0\}$ is dense.
\end{remark}
In the following we write $M^{(d)}$ instead of $M$ to express the dependence on $d$ and similarly $g^{(d)}_{s,t}$ etc. Note that $M ^{(1)}$ are the local minima (recall Example~\ref{example:reflected-loc-min}).
 
At this point, for all we know, the $M^{(d)}$, $d\in (0,2)$, could still all be but versions of $M^{(1)}$, say.  The following proposition finally settles the question of Tsirelson mentioned in the Introduction. Local minima and maxima (and their union) are not the only examples of stationary local dense random countable sets over the Wiener noise. In fact there are at least continuum many.

\begin{proposition}\label{proposition:a.s.distinct}
For any distinct $d_1,d_2 \in (0,2)$ we have that $M^{(d_1)}\cap M^{(d_2)}$ is empty.
\end{proposition}

In order to establish this result we combine the comparison principle for SDEs with the following property of the path of the solution of \eqref{sde} after its last zero in an interval.
\begin{lemma}\label{lemma:slln}
Let $(Z(t))_{t \in [0,\infty)}$ be a solution to \eqref{sde} with $d\in (0,2)$, let $T\in (0,\infty)$ and set $g:= \sup\{u \in [0,T]: Z(u)=0\}$. Then
\[
\lim_{\epsilon \downarrow 0} \frac{1}{\log(1/\epsilon)}\int_{g+\epsilon}^T \frac{\dd s}{Z(s)}=\frac{1}{2-d}  \quad \text{ a.s.-$\WW$.}
\]
\end{lemma}
\begin{proof}
Set $\tilde{Z}$ to be the scaled post-$g$ path:
\begin{equation}\label{eq:meander}
\tilde{Z}(t):= \frac{1}{{T-g}}Z(g+(T-g)t)   \text{ for  }t \in [0,1].
\end{equation}
It is known, see e.g. \cite[Lemma~2.2(2)]{yor1} coupled with \cite[Proposition~XI.1.6]{revuz}, that $\tilde{Z}$ has a distribution that is absolutely continuous with respect to the  law of a squared Bessel process of dimension $4-d$ starting from zero. But if $\hat{Z}$ denotes the latter process, then it is also known --- see e.g. \cite[Theorem 1.1]{yor2}, coupled with time inversion  (namely, with the fact that $(t^2\hat{Z}(1/t))_{t\in (0,\infty)}$ is again a squared Bessel process of dimension $4-d$ \cite[Theorem~3.1]{shiga}), and  by applying the Markov property of $\tilde{Z}$ at time $1$ --- that 
\[
\lim_{\epsilon \downarrow 0} \frac{1}{\log(1/\epsilon)}\int_{\epsilon}^1 \frac{\dd s}{\hat{Z}(s)}=\frac{1}{2-d},
\] a.s., which concludes the proof after a trivial transposition back to  $Z$.
\end{proof}
\begin{proof}[Proof of Proposition~\ref{proposition:a.s.distinct}]
Assume without loss of generality that $d_1<d_2$.
It suffices to show that for  arbitrary pairs of rational times $s_1<t_1$ and $s_2<t_2$ we have $\WW(g_{s_1,t_1}^{(d_1)}\neq g_{s_2,t_2}^{(d_2)})=1$. Using Lemma~\ref{lemma:gamma}\ref{lemma:gamma:iii} we can assume that $s_1=s_2$ and $t_1=t_2$. By temporal homogeneity; that is to say, by the fact that $\Delta_{s_1}$ is measure-preserving for $\WW$, we may  further assume $s_1=0$, and we write $t_1=:T$.
Consider the a.s.-$\WW$ unique continuous, $[0,\infty)$-valued, $\FF^{0,\rightarrow}$-adapted solutions $Z_i$ to $$Z_i(t)=2\int_{0}^t \sqrt{Z_i(s)}\dd B(s)+ d_it,\quad t\in [0,\infty),\, i\in \{1,2\}.$$
 By the comparison theorem \cite[Theorem~IX.3.7]{revuz} $Z_1(t)  \leq Z_2(t)$ for all $t \geq 0$ a.s.-$\WW$, and consequently  $g_1:=g_{0,T}^{(d_1)}\geq g_{0,T}^{(d_2)}=:g_2$, but also
\[
 \frac{1}{\log(1/\epsilon)}\int_{{g_1}+\epsilon}^T \frac{\dd s}{Z_1(s)} \geq \frac{1}{\log(1/\epsilon)}\int_{{g_1}+\epsilon}^T \frac{\dd s}{Z_2(s)} \text{ for } \epsilon\in (0,T-g_1)\text{ a.s.-$\WW$}.
\]
By Lemma~\ref{lemma:slln} the left-hand side tends to $1/(2-d_1)$ as $\epsilon \downarrow 0$ a.s.-$\WW$. But  the same lemma also implies that on the event  $\{g_1=g_2\}$ the right-hand side tends to $1/(2-d_2)$ as $\epsilon\downarrow 0$ a.s.-$\WW$. Since $1/(2-d_1)<1/(2-d_2)$ it must be that $\WW(g_1=g_2)=0$.
\end{proof}

We make some final remarks concerning possible generalizations of Proposition~\ref{proposition:a.s.distinct}. Let $Z^i$, $i\in \{1,2\}$, be two solutions to the SDE \eqref{general-sde} each with its own drift and volatility coefficient. Assume they are both nonnegative strong pathwise unique solutions with zero regular and not an atom at positive times. For $i\in \{1,2\}$ let $M^i$ be associated to $Z^i$ as $M$ is to $Z$. As discussed above (Remarks~\ref{remark:generalize-sde} and~\ref{rmK.dense-general}) $M^1$ and $M^2$ are then both stationary dense local random countable sets.
\begin{enumerate}[(1)]
\item It is clear from the Bessel examples that, in order for $\WW(M^1\cap M^2=\emptyset)=1$, it is \emph{not} necessary that the two-dimensional process $(Z^1,Z^2)$ never hit the origin (after time zero) with probability one. On the other hand, it is also clear, from the proof of  Proposition~\ref{proposition:a.s.distinct}, that this condition is sufficient. A neccesary and sufficient condition is  that $(Z^1,Z^2)$ on last exit from the coordinate axes (before any given deterministic time $T>0$) does so a.s. from a point other than the origin.

\item In order that $\WW(M^1=M^2)=1$ it is sufficient that the drift and volatility functions of $Z^1$ and $Z^2$ agree on a neighborhood of the origin. This is just because, in the obvious notation, $$M^i=\{g_{s,t}^i:(s,t)\in \mathbb{Q}^2, 0< t-s<\epsilon\}$$ a.s.-$\WW$ for all $\epsilon>0$, $i\in \{1,2\}$.

\item If $Z^1$ is a  solution to  \eqref{general-sde} for the drift-volatility pair $(\mu^1,\sigma)$ and $Z^2$ is the same but for the pair  $(\mu^2,\sigma)$, with $\mu^1\leq d_1<d_2\leq \mu^2$ and with $\sigma$ the volatility function for the Bessel SDE \eqref{sde}, then by the comparison theorem and the findings of the proof of  Proposition~\ref{proposition:a.s.distinct} the  sets $M^1$ and $M^2$ (associated to $\Gamma^1:=\{Z^1=0\}$ and $\Gamma^2:=\{Z^2=0\}$, respectively) are a.s. disjoint. It is a (slightest) generalization outside the confines of the Bessel class of Proposition~\ref{proposition:a.s.distinct}.
\end{enumerate}

\section{Honest indexations}\label{section:honest-indexations}

The following definition is modeled on the construction of the random sets $M^{(d)}$, $d\in (0,2)$, of Section~\ref{section:new-family-of-examples}.

\begin{definition}
Let $M:\Omega_0\to 2^\mathbb{R}$ be a random  set. An honest indexation for $M$ is a family $\tau=(\tau_{s,t})_{(s,t)\in \mathbb{R}^2,s< t}$ from  $\FF_{-\infty,\infty}/\mathcal{B}_{\mathbb{R}^\dagger}$, such that: (inclusion) $\tau_{s,t}\in M$ a.s.-$\WW$ for all real $s<t$; (exhaustion) $M=\{\tau_{s,t}:(s,t)\in \mathbb{Q}^2,s< t\}$ a.s.-$\WW$; (locality) $\tau_{s,t}$ is a.s.-$\WW$ valued in $(s,t)$ and $\FF_{s,t}$-measurable for all real $s<t$; (stationarity) $\tau_{s+h,t+h}=h+\tau_{s,t}(\Delta_h)$ a.s.-$\WW$ for all real $h$ and $s<t$; (nestedness) $\tau_{s,t}=\tau_{u,v}$ a.s.-$\WW$ on $\{\tau_{s,t}\in (u,v)\}$ for all real $s\leq u<v\leq t$.
\end{definition}
The combination of nestedness and locality reminds us of the so-called honest times in a filtration $(\GG_t)_{t\in [0,\infty)}$, being those times $L$ for which for each $t\in (0,\infty)$ there exists a $\GG_t$-measurable $L_t$ satisfying $L=L_t$ on $\{L<t\}$ \cite[Chapter~XX, Section~1, \#~18]{dellacherie-sets}; whence the naming  ``honest indexation''. 

A  random  set admitting an honest indexation is automatically a dense stationary local random countable set. Changing each member of an honest indexation on a $\WW$-negligible set retains the honest indexation property. For real $s<t$, $u$, and for an honest indexation $\tau$ of $M$, $\WW(\tau_{s,t}=u)\leq \WW(u\in M)=0$ by inclusion and Proposition~\ref{SIL:probability-zero}. If $A\subset \mathbb{R}$ is dense, then each member of an honest indexation $\tau$ is determined $\WW$-a.s. by the restriction of the family $\tau$ to those members both of whose indexing endpoints lie in $A$; besides,  $M=\{\tau_{s,t}:(s,t)\in A^2,s< t\}$ a.s.-$\WW$. Due to the first of these two properties one could, in principle, work instead with the concept of an honest indexation  indexed only by the times belonging to $A$, e.g. $A=\mathbb{Q}$. However, the definition would then become less clear-cut (e.g. the stationarity does not restrict naturally to $h\in A$, but we need it for all $h\in \mathbb{R}$, not just $h\in A$). Besides, it will be advantageous to work with a version of $ \tau_{0,t}$ indexed by, and possesing some nice properties as a function of the continuous parameter $t\in (0,\infty)$. For these two reasons we take  already in the definition an indexation over the (pairs of) real times.

\begin{example}\label{example:Md-honest}
For each $d\in (0,2)$, $(g^{(d)}_{s,t})_{(s,t)\in \mathbb{R}^2,s<t}$ is an honest indexation of $M^{(d)}$.
\end{example}
Honest indexations are not unique when they exist.
\begin{example}\label{example:index-min}
For each $\kappa\in \mathbb{R}$ the $\tau=(\tau_{s,t})_{(s,t)\in \mathbb{R}^2,s<t}$, which has,  for real $s<t$, $\tau_{s,t}$ equal a.s.-$\WW$ to a  minimum of  $(B_t+\kappa t)_{t\in \mathbb{R}}$ on $(s,t)$, is an honest indexation of the local minima of $B$ (apply e.g. the Paley, Wiener and Zygmund result on the failure of differentiability of Brownian paths \cite[Theorem~1.30]{morters}). The indexation with $\kappa=0$ we will call standard for the local minima (analogously for the local maxima). 
\end{example}
\subsection{Regularization and splitting at an honest indexator}
The main goal of the present subsection is, for an honest indexation $\tau$, to deliver a kind of Wiener-Hopf splitting statement at $\tau_{0,\mathsf{e}}$, where $\mathsf{e}$ is an independent exponential random time (Theorem~\ref{theorem:splitting}). In order to make proper sense of this, and also for the eventual proof of the splitting, we require a sufficiently regular version of the process $(\tau_{0,t})_{t\in(0,\infty)}$, which is the subject of

\begin{lemma}\label{lemma:perfect-good}
Let $M$ be a  random countable set admitting an honest indexation $\tau$. Then we may change each $\tau_{0,t}$, $t\in (0,\infty)$, on a set of $\WW$-measure zero (therefore, without affecting the ``honest indexation'' property of $\tau$) in such a way that  $(\tau_{0,t})_{t\in (0,\infty)}$ becomes $[0,\infty)$-valued, $(B\vert_{[0,\infty)})^{-1}(\HH)$-measurable, right-continuous, nondecreasing, majorized by $\mathrm{id}_{(0,\infty)}$ 
and has the ``perfect honest indexation property'': on a $\WW$-a.s. set closed for $(\Delta_u)_{u\in [0,\infty)}$ one has for all $t\in [0,\infty)$, $T\in (0,\infty)$ that $\tau_{0,T+t}= t+\tau_{0,T}(\Delta_t)$ on $\{\tau_{0,T+t}\in (t,\infty)\}$. Furthermore, any such version of the process $(\tau_{0,t})_{t\in (0,\infty)}$ is $\WW$-a.s.  constant on its excursions away from the diagonal (in particular, when it jumps, it jumps to the diagonal), i.e. $\tau_{0,\tau_{0,t}}=\tau_{0,t}$ (implicitly, $\tau_{0,t}>0$) for all $t\in (0,\infty)$ a.s.-$\WW$.
\end{lemma}
\begin{proof}
  For each $t\in (0,\infty)$ choose an $\HH$-measurable, $(0,t)$-valued $\hat{\tau}_{0,t}$ such that $\hat{\tau}_{0,t}(B\vert_{[0,\infty)})=\tau_{0,t}$ a.s.-$\WW$. We improve $(\hat{\tau}_{0,t})_{t\in (0,\infty)}$ in three steps. 

First, pass to a version $(\hat{\tau}'_{0,t})_{t\in (0,\infty)}$ of $(\hat{\tau}_{0,t})_{t\in (0,\infty)}$ that is right-continuous, nondecreasing, $[0,\infty)$-valued and majorized by $\mathrm{id}_{(0,\infty)}$, by putting  
$$\hat\tau'_{0,t}:=\inf\{\hat{\tau}_{0,p}:p\in \mathbb{Q}\cap (t,\infty)\},\quad t\in (0,\infty).$$
The right-continuity and nondecreasingness are evident as is the fact that $\hat \tau'_{0,\cdot}$ is majorized by $\mathrm{id}_{(0,\infty)}$ and $[0,\infty)$-valued. We argue that we have a version, namely, that for all $t\in (0,\infty)$, $\hat\tau'_{0,t}=\hat\tau_{0,t}$ a.s.-$\PPP$. By nestedness this is certainly true on the event $A$ that $\hat\tau_{0,p}\in (0,t)$ for  some [sufficiently small] rational $p\in (t,\infty)$, since thereon $\hat\tau_{0,q}=\hat\tau_{0,p}$ a.s.-$\PPP$ for $q\in [t,p)$, while $\hat\tau_{0,q}\geq \hat\tau_{0,p}$ a.s.-$\PPP$ for all $q\in [p,\infty)$. Off $A$ we have a.s.-$\PPP$ for rational $p\in (t,2t)$, that $t\leq \hat\tau_{0,p}=\hat\tau_{p-t,p}=p-t+\hat\tau_{0,t}(\Delta_{t-p})$, hence also a.s.-$\PPP$, for rational $p\in (t,2t)$, that $2t-p\leq \hat\tau_{0,t}$,  which forces $\hat\tau_{0,t}\geq t$ a.s.-$\PPP$. But that can only mean that $\Omega_0\backslash A$ is actually $\PPP$-negligible.

 Second, 
 change $(\hat{\tau}'_{0,t})_{t\in (0,\infty)}$  on a $\PPP$-negligible set to the identity map on $(0,\infty)$ to obtain $(\tilde{\hat{\tau}}_{0,t})_{t\in (0,\infty)}$,   which is $(0,\infty)$-valued everywhere.

  Third, define:
$$R:=\{(s,\omega)\in [0,\infty)\times \Theta_0:\text{  for all }t\in (0,\infty)\text{ if }\tilde{\hat{\tau}}_{0,s+t}(\omega)\in (s,\infty)\text{ then }\tilde{\hat{\tau}}_{0,s+t}(\omega)=s+\tilde{\hat{\tau}}_{0,t}(\Delta_s\omega)\},$$
$$\Theta_1:=\{\omega\in \Theta_0:(s,\omega)\in R\text{ for $\mathscr{L}$-a.e. }s\in [0,\infty)\}\text{ and }$$
$$\Theta_2:=\{\omega\in \Theta_0:\Delta_u\omega\in \Theta_1\text{ for $\mathscr{L}$-a.e. }u\in  [0,\infty)\}.$$
By right-continuity and nondecreasingness of $(\tilde{\hat{\tau}}_{0,t})_{t\in (0,\infty)}$ we get $$R=\cap_{t\in \mathbb{Q}\cap (0,\infty)}\{(s,\omega)\in [0,\infty)\times \Theta_0:\tilde{\hat{\tau}}_{0,s+t}(\omega)\in (s,\infty)\Rightarrow\tilde{\hat\tau}_{0,s+t}(\omega)=s+\tilde{\hat{\tau}}_{0,t}(\Delta_s\omega)\}\in  \mathcal{B}_{[0,\infty)}\otimes\HH;$$ and then in consecutive order that  $\Theta_1\in \HH$, $\PPP(\Theta_1)=1$, $\Theta_2\in \HH$ is shift-closed and $\PPP(\Theta_2)=1$.

Before proceeding further we make the following observation. Let $\omega\in \Theta_0$, $t\in (0,\infty)$, $\{s_1,s_2\}\subset [0,t)$,  $\{\Delta_{s_1}\omega,\Delta_{s_2}\omega\}\subset \Theta_1$, $s_1\leq s_2$. Since $\Delta_{s_1}\omega\in \Theta_1$ then for $\mathscr{L}$-a.e. $u\in [0,\infty)$,  $$\tilde{\hat{\tau}}_{0,(u+s_2-s_1)+(t-s_2)}(\Delta_{s_1}\omega)>u+s_2-s_1\Rightarrow \tilde{\hat{\tau}}_{0,(u+s_2-s_1)+(t-s_2)}(\Delta_{s_1}\omega)=u+s_2-s_1+\tilde{\hat{\tau}}_{0,t-s_2}(\Delta_{(u+s_2-s_1)+s_1}\omega),$$ i.e. 
$$\tilde{\hat{\tau}}_{0,u+t-s_1}(\Delta_{s_1}\omega)>u+s_2-s_1\Rightarrow \tilde{\hat{\tau}}_{0,u+t-s_1}(\Delta_{s_1}\omega)=u+s_2-s_1+\tilde{\hat{\tau}}_{0,t-s_2}(\Delta_{u+s_2}\omega).$$ 
If further for arbitrarily small $u\in [0,\infty)$, $\tilde{\hat{\tau}}_{0,u+t-s_1}(\Delta_{s_1}\omega)\leq u+s_2-s_1$, then by right-continuity of $\tilde{\hat{\tau}}_{0,\cdot}(\Delta_{s_1}\omega)$ we get $\tilde{\hat{\tau}}_{0,t-s_1}(\Delta_{s_1}\omega)\leq s_2-s_1\leq s_2+\tilde{\hat{\tau}}_{0,t-s_2}(\Delta_{s_1}\omega)-s_1$, which renders 
$$s_1+\tilde{\hat{\tau}}_{0,t-s_1}(\Delta_{s_1}\omega)\leq s_2+\tilde{\hat{\tau}}_{0,t-s_2}(\Delta_{s_2}\omega).$$ In the opposite case we have for $\mathscr{L}$-a.e. small enough $u\in [0,\infty)$, 
$$s_1+\tilde{\hat{\tau}}_{0,u+t-s_1}(\Delta_{s_1}\omega)=u+s_2+\tilde{\hat{\tau}}_{0,t-s_2}(\Delta_{u+s_2}\omega);$$ and assume this now. Since $\Delta_{s_2}\omega\in \Theta_1$ and $\tilde{\hat{\tau}}_{0,t-s_2}(\Delta_{s_2}\omega)>0$,  then for all small enough $u\in [0,\infty)$, 
$\tilde{\hat{\tau}}_{0,u+t-s_2}(\Delta_{s_2}\omega)>u$, hence for $\mathscr{L}$-a.e. small enough $u\in [0,\infty)$, $$s_2+\tilde{\hat{\tau}}_{0,u+t-s_2}(\Delta_{s_2}\omega)=s_2+u+\tilde{\hat{\tau}}_{0,t-s_2}(\Delta_{s_2+u}\omega).$$ 
Combining the preceding two displayed conclusions we infer that
$$s_2+\tilde{\hat{\tau}}_{0,u+t-s_2}(\Delta_{s_2}\omega)=s_1+\tilde{\hat{\tau}}_{0,u+t-s_1}(\Delta_{s_1}\omega)$$ for $\mathscr{L}$-a.e. small enough $u\in [0,\infty)$; and further by right continuity of $\tilde{\hat{\tau}}_{0,\cdot}$ that $$s_2+\tilde{\hat{\tau}}_{0,t-s_2}(\Delta_{s_2}\omega)=s_1+\tilde{\hat{\tau}}_{0,t-s_1}(\Delta_{s_1}\omega).$$
Altogether we have deduced that in any case 
$$s_1+\tilde{\hat{\tau}}_{0,t-s_1}(\Delta_{s_1}\omega)\leq s_2+\tilde{\hat{\tau}}_{0,t-s_2}(\Delta_{s_2}\omega)$$
[even with equality if $s_2<\inf_{s\in [0,t),\Delta_s\omega\in \Theta_1}(s+\tilde{\hat{\tau}}_{0,t-s}(\Delta_s\omega))$, since the latter gives $\tilde{\hat{\tau}}_{0,t-s_1}(\Delta_{s_1}\omega)>s_2-s_1$ and hence $\tilde{\hat{\tau}}_{0,u+t-s_1}(\Delta_{s_1}\omega)> u+s_2-s_1$ for all small enough $u\in [0,\infty)$ in the above (the ``opposite'' case), but we shall not need this].

Armed with the preceding observation we put, for $t\in (0,\infty)$, $$\overline{\hat{\tau}}_{0,t}(\omega):=
\begin{cases}
\inf_{s\in [0,t),\Delta_s\omega\in \Theta_1}(s+\tilde{\hat{\tau}}_{0,t-s}(\Delta_s\omega))=\text{$\downarrow$-$\lim$}_{0\leq s\downarrow 0,\Delta_s\omega\in \Theta_1}(s+\tilde{\hat{\tau}}_{0,t-s}(\Delta_s\omega))& \text{if $\omega\in \Theta_2$,}\\
t & \text{if $\omega\in\Theta_0\backslash \Theta_2$,}
\end{cases}$$
stressing that, contrary to the default standard interpretation, the limit in the first line may include $0$: it includes it iff $\omega\in \Theta_1$.
Directly from the definition and the apposite properties of $\tilde{\hat{\tau}}$ it is plain that $\overline{\hat{\tau}}$ is $[0,\infty)$-valued,  majorized by $\mathrm{id}_{(0,\infty)}$, nondecreasing and right-continuous. Also direct from the definition is the fact that $\overline{\hat{\tau}}=\tilde{\hat\tau}$ on $\Theta_1\cap\Theta_2$, which is $\PPP$-almost sure, so we have a version.

To check the perfect honest indexation property of $\overline{\hat\tau}_{0,\cdot}$ on $\Theta_2$ let $\omega\in \Theta_2$, $t\in  [0,\infty)$ (so, $\Delta_t\omega\in \Theta_2$ also), $T\in (0,\infty)$ and suppose that $\overline{\hat\tau}_{0,T+t}(\omega)\in  (t,\infty)$. We compute
\begin{align*}
\overline{\hat{\tau}}_{0,T+t}(\omega)&=\inf_{s\in[ 0,T+t),\Delta_s\omega\in \Theta_1}(s+\tilde{\hat\tau}_{0,T+t-s}(\Delta_s\omega))\\
&=\left(\inf_{s\in[ 0,t),\Delta_s\omega\in \Theta_1}(s+\tilde{\hat\tau}_{0,T+t-s}(\Delta_s\omega))\right)\land \left(\inf_{s\in[ t,T+t),\Delta_s\omega\in \Theta_1}(s+\tilde{\hat\tau}_{0,T+t-s}(\Delta_s\omega)) \right)\\
&=\left(\inf_{s\in[ 0,t),\Delta_s\omega\in \Theta_1}(s+\tilde{\hat\tau}_{0,T+t-s}(\Delta_s\omega))\right)\land \left(t+\inf_{s\in[ 0,T),\Delta_s\Delta_t\omega\in \Theta_1}(s+\tilde{\hat\tau}_{0,T-s}(\Delta_s\Delta_t\omega)) \right)\\
&=\left(\inf_{s\in[ 0,t),\Delta_s\omega\in \Theta_1}(s+\tilde{\hat\tau}_{0,T+t-s}(\Delta_s\omega))\right)\land \left(t+\overline{\hat{\tau}}_{0,T}(\Delta_t\omega)\right).
\end{align*}
Let $s\in[ 0,t)$, $\Delta_s\omega\in \Theta_1$; it remains to check that $s+\tilde{\hat\tau}_{0,T+t-s}(\Delta_s\omega)\geq t+\overline{\hat{\tau}}_{0,T}(\Delta_t\omega)$. Since $\overline{\hat\tau}_{0,T+t}(\omega)>t$ we know that $s+\tilde{\hat\tau}_{0,T+t-s}(\Delta_s\omega)>t$. Therefore, as $\Delta_s\omega\in\Theta_1$, for $\mathscr{L}$-a.e. small enough $u\in[0,\infty)$, 
$$s+\tilde{\hat\tau}_{0,T+t-s+u}(\Delta_s\omega)=s+(t-s+u)+\tilde{\hat\tau}_{0,T}(\Delta_{t+u}\omega)=t+u+\tilde{\hat\tau}_{0,T}(\Delta_{t+u}\omega)\geq t+\left(u+\tilde{\hat\tau}_{0,T-u}(\Delta_{u}\Delta_t\omega)\right).$$
Since $\Delta_t\omega\in \Theta_2$, by definition of $\overline{\hat{\tau}}_{0,T}(\Delta_t\omega)$, taking limit $u\downarrow 0$ over $\mathscr{L}$-a.e. small enough $u$ we get the sought for inequality.

The version of $(\tau_{0,t})_{t\in (0,\infty)}$ stipulated by the lemma is got by taking finally $\overline{\hat{\tau}}_{0,\cdot}(B\vert_{[0,\infty)})$. 

The final assertion of the lemma concerning the relation of $\tau_{0,\cdot}$ to the diagonal $\mathrm{id}_{(0,\infty)}$ follows easily from nestedness by ``rational exhaustion''.
\end{proof}


The reader will have noticed that, like with Proposition~\ref{proposition:two-sided-viz-one-sided}, there is the subtle point of the perfect honest indexation property of Lemma~\ref{lemma:perfect-good} holding with $\{\tau_{0,T+t}\in (t,\infty)\}$ rather than $\{\tau_{0,T+t}\in [t,\infty)\}$. Again it is our impression that one cannot improve to the latter from the former in a trivial way, because the proof is good in the L\'evy setting, and such an improvement cannot prevail there. Consider indeed a two-sided infinite activity subordinator whose measure has for its support the set $\{\frac{1}{n}:n\in \mathbb{N}\}$. For $s<t$ from $\mathbb{R}$ let $\tau_{s,t}$ be: the first jump of size $1$ on the interval $(s,t)$, if any; otherwise, the first jump of size $\frac{1}{2}$ on the interval $(s,t)$, if any, etc.; $\dagger$ if none occur (which happens only with probability zero). The family $(\tau_{s,t})_{(s,t)\in \mathbb{R}^2}$ is an honest indexation of the jump times. But clearly the indicated would-be ``improvement'' fails by considering e.g. the first jump time of size $1$, call it $J$: we would have, for any $T\in (0,\infty)$, a.s. $\tau_{0,T+J}=J+\tau_{0,T}(\Delta_J)>J$ but also  $\tau_{0,T+J}=J$, which is absurd.

A second ingredient in the proof of Theorem~\ref{theorem:splitting} that we prepare beforehand is the technical

\begin{lemma}\label{lemma:stopping-time-conditioning}
Let $T$ be an $\FF^\rightarrow$-stopping time, $\WW(T<\infty)>0$. Then, on $\{T<\infty\}$, $\FF^\rightarrow_T$ is $\WW$-independent of $(\Delta_T\vert_{[0,\infty)})^{-1}(\HH)$, and ${\Delta_T}_\star \WW(\cdot\vert T<\infty)= \WW$ on $(B\vert_{[0,\infty)})^{-1}(\HH)$. Consequently, if $A\in \FF^\rightarrow_T$ and $g\in (\mathcal{B}_{[0,\infty)}\otimes (B\vert_{[0,\infty)})^{-1}(\HH))/\mathcal{B}_{[0,\infty]}$, then 
$$\WW[g(T,\Delta_T);A,T<\infty]=\int_{[0,\infty)}\WW[g(t,B)]\WW(A,T\in \dd t).$$
\end{lemma}
It is basically the strong Markov property  with some careful book-keeping of the null sets. We can couch it in a slightly different, but equivalent form: on $\{T<\infty\}$, under $\WW$, $\FF^\rightarrow_T$  is  independent of $(\Delta_T B)\vert_{[0,\infty)}$, which has the distribution $\PPP$ thereon. Here we could, and did view $(\Delta_T B)\vert_{[0,\infty)}$ as a random element in $(\Theta_0,\HH)$; we did not work, and  could not have worked with $\Delta_T B$, viewing it as an element of $(\Omega_0,\FF_{0,\infty})$ (indeed, the preimage of a $\WW$-negligible set under $\Delta_T$ need not be $\WW$-negligible at all: for instance, consider the first hitting time of $1$ by $B\vert_{[0,\infty)}$ as $T$ and take for the set the collection of paths which are nonpositive on a left neighborhood of $0$)! 
\begin{proof}
That  $\FF^\rightarrow_T$ is independent of $(\Delta_T\vert_{[0,\infty)})^{-1}(\mathcal{B}_{\Theta_0})$ on $\{T<\infty\}$ is just the usual strong Markov property for the Brownian motion $B\vert_{[0,\infty)}$ in the filtration $\FF^\rightarrow$. Since (also by the usual strong Markov property) the law of $\Delta_T\vert_{[0,\infty)}$ (relative to $\mathcal{B}_{\Theta_0}$) under $\WW(\cdot\vert T<\infty)$ is $\PPP\vert_{\BB_{\Theta_0}}$ we see that, on $\{T<\infty\}$, $(\Delta_T\vert_{[0,\infty)})^{-1}(\HH)=(\Delta_T\vert_{[0,\infty)})^{-1}(\mathcal{B}_{\Theta_0})\lor (\Delta_T\vert_{[0,\infty)})^{-1}(\PPP^{-1}(\{0,1\}))\subset (\Delta_T\vert_{[0,\infty)})^{-1}(\mathcal{B}_{\Theta_0})\lor \WW^{-1}(\{0,1\})$, hence  $\FF^\rightarrow_T$ is independent even of $(\Delta_T\vert_{[0,\infty)})^{-1}(\HH)$ (by an application of Dynkin's lemma, say). Similarly we convince ourselves that ${\Delta_T}_\star \WW(\cdot\vert T<\infty)= \WW$ on $(B\vert_{[0,\infty)})^{-1}(\HH)$.

As for the second statement, whenever we have a probability $\QQ$, $\QQ$-independent random elements $Y$ and $Z$ and a nonnegative measurable map $h$ of this pair, then $\WW[h(Y,Z)]=\int \WW[h(y,Z)] (Y_\star\WW)(\dd y)$ (brief argument: combine Tonneli, the image measure theorem and the fact that the law of an independent pair of random elements is the product of their laws). Applying the latter with $\QQ=\WW(\cdot \vert T<\infty)$, $Y=(\mathbbm{1}_A,T)$, $Z=\Delta_T$ (viewed as taking values in $(\Omega_0,(B\vert_{[0,\infty)})^{-1}(\HH))$) and $$h((i,t),z):=\mathbbm{1}_{\{1\}}(i)
g(t,z),\quad ((i,t),z)\in (\{0,1\}\times [0,\infty))\times \Omega_0,$$
we obtain 
\begin{align*}
\WW[g(T,\Delta_T);A,T<\infty]&=\WW[h((\mathbbm{1}_A,T),\Delta_T)\vert T<\infty]\WW(T<\infty)\\
&=\int_{[0,\infty)}\WW[g(t,\Delta_T)\vert T<\infty]\WW(A,T\in \dd t)=\int_{[0,\infty)}\WW[g(t,B)]\WW(A,T\in \dd t),
\end{align*}
 which concludes the argument.
\end{proof}
As final preparation for  the forthcoming ``Wiener-Hopf'' splitting result  we specify some extra notation.

\begin{definition}
Let $\mathsf{r}:=(\Omega_0\ni \omega\mapsto (\mathbb{R}\ni t\mapsto \omega(-t)))$ be the  reflection of time. For an honest indexation $\tau$ for $M$ define its dual  $\hat{\tau}=(\hat{\tau}_{s,t})_{(s,t)\in \mathbb{R}^2,s<t}$ by setting $\hat \tau_{s,t}:=-\tau_{-t,-s}\circ \mathsf{r}$ for real $s<t$ ($-\dagger:=\dagger$). 
\end{definition}

The reflection of time $\mathsf{r}$ is measure-preserving for $\WW$. For $s\in \mathbb{R}$, $\Delta_s\circ \mathsf{r}=\mathsf{r}\circ \Delta_{-s}$, in particular $ \mathsf{r}\circ \Delta_s$ is its own inverse -- the reflection about, subsequent to centering at $s$. $\hat{\tau}$ is an honest indexation of $-M\circ \mathsf{r}$ and $\hat{\tau}_{0,t}=t-\tau_{0,t}\circ \mathsf{r}\circ \Delta_t$ for real $t$. Also, $\widehat{\hat\tau}=\tau$.

\begin{theorem}\label{theorem:splitting}
Fix $\lambda\in (0,\infty)$. Suppose the  stationary local random countable set $M$ admits an honest indexation $\tau$. Let  $M'$ be another stationary local random countable set. Set $\tilde \WW:=\WW\times \mathrm{Exp}(\lambda)$ ($\mathrm{Exp}(\lambda)$ being the exponential law of mean $\lambda^{-1}$ on $\mathcal{B}_{(0,\infty)}$), letting $\mathsf{e}$ be the second coordinate of $\tilde\WW$, while for the copies of the random elements/objects supported by $\WW$ we do not introduce new notation, retaining the same by an abuse. Write  $\tau_t:=\tau_{0,t}$, $t\in (0,\infty)$; as well as     $\hat{\tau}_t:=\hat{\tau}_{0,t}$, $t\in (0,\infty)$, for the dual indexation. We insist that $((0,\infty)\times\Omega_0\ni (t,\omega)\mapsto \tau_t(\omega)\in [0,t])\in (\BB_{(0,\infty)}\otimes \GG)/\mathcal{B}_{[0,\infty)}$ and $((0,\infty)\times\Omega_0\ni (t,\omega)\mapsto \hat\tau_t(\omega)\in [0,t])\in (\BB_{(0,\infty)}\otimes \GG)/\mathcal{B}_{[0,\infty)}$ (note: according to Lemma~\ref{lemma:perfect-good} a version of $\tau$ exists for which this is true). Put $\FF_{-\infty,\tau_e}:=\sigma(Z_{\tau_{\mathsf{e}}}:Z\text{ a  bounded real left-continuous $\FF^\rightarrow$-adapted process})$. 

\begin{enumerate}[(i)]
\item\label{splitting:i} $\FF_{-\infty,\tau_{\mathsf{e}}}$, in particular the pair $(\tau_\mathsf{e},B^{\tau_{\mathsf{e}}})$, is independent of $(\mathsf{e}-\tau_{\mathsf{e}},\Delta_{\tau_\mathsf{e}}\vert_{[0,\infty)})$ under $\tilde\WW$.
\item\label{splitting:ii} The event $\{\tau_\mathsf{e}\in M'\}$ is $\tilde \WW$-trivial.
\item\label{splitting:three} $(\mathsf{e}-\tau_\mathsf{e},\Delta_{\tau_\mathsf{e}}\vert_{[0,\infty)})_\star \tilde\WW=(\hat\tau_\mathsf{e},(\mathsf{r}\circ \Delta_{\hat\tau_\mathsf{e}})\vert_{[0,\infty)})_\star \tilde\WW$.
\end{enumerate}
\end{theorem}
In the case of $M$ being  the local minima (resp. maxima) of $B$ with $\tau_{s,t}$ a minimizer (resp. maximizer) of $B$ on $(s,t)$ for real $s<t$, Items~\ref{splitting:i} and~\ref{splitting:three} give the usual splitting at the minimum (resp. maximum) before an independent exponential random time, which underlies the Wiener-Hopf factorization \cite{greenwood-pitman}. Via Example~\ref{example:index-min} we extend this observation to drifting Brownian motion.

Clearly the statement of the theorem does not go through in its entirety with, ceteris paribus, $\mathsf{e}$ a deterministic random time from $(0,\infty)$. For instance, if $M$ are the local minima of $B$ and, for real $s<t$, $\tau_{s,t}$ is a minimizer of $B$ on $(s,t)$, then clearly $\tau_{0,t}$ is  independent of $t-\tau_{0,t}$ for no $t\in (0,\infty)$ (vis-\'a-vis Item~\ref{splitting:i}). This is not to say that perhaps one could not prove Item~\ref{splitting:ii} in such case, namely that $\{\tau_t\in M'\}$ is $\WW$-trivial for $t\in (0,\infty)$, however at least the structure of the proof which follows does not seem to allow for this. On the other hand, Item~\ref{splitting:three} --- the equality in law of the pre-$\hat \tau_\mathsf{e}$ increments looked backwards together with $\hat\tau_\mathsf{e}$ and of the post-$\tau_\mathsf{e}$ increments together with $\mathsf{e}-\tau_\mathsf{e}$ --- certainly implies (taken for all $\lambda\in (0,\infty)$) its ``$\mathsf{e}$ is deterministic'' version, as is readily verified.

\begin{proof}
Changing the process $\tau_\cdot:=(\tau_t)_{t\in (0,\infty)}$ to any (jointly) measurable version thereof, satisfying $\tau_t\in [0,t]$ for all $t\in (0,\infty)$, changes $\tau_\mathsf{e}$ only on a $\tilde\WW$-negligible set. The same for $\hat\tau_\cdot:=(\hat\tau_t)_{t\in (0,\infty)}$. Therefore we may and do just as well assume that $\tau_\cdot$ has all of properties stipulated by Lemma~\ref{lemma:perfect-good}, the perfect honest indexation property in particular. Likewise for $\hat\tau_\cdot$.

 Pass also to the nice version of $M'$ guaranteed to exist by Corollary~\ref{corollary:nice-version}. The main gain from this is that we have  perfect stationarity combined with  $M'\cap (0,\infty)$ being measurable w.r.t. $B\vert_{[0,\infty)}$ and possibly its null sets, i.e. $M'\cap (0,\infty)$ having a $(B\vert_{[0,\infty)})^{-1}(\HH)$-measurable perfect enumeration; the intervention of the null sets of the whole of $B$ is not needed. This, together with the $(B\vert_{[0,\infty)})^{-1}(\HH)$-measurability of $\tau_\cdot$, will be instrumental in applying Lemma~\ref{lemma:stopping-time-conditioning} below.

Next, set, for $l\in [0,\infty)$, $$T_l:=\inf\{t\in (0,\infty):\tau_{t}> l\}.$$ We have that $T_l$ is an  $\FF^{0,\rightarrow}$-stopping time ($\because$  by locality the process $\tau_\cdot$ is $(\FF_{0,t})_{t\in (0,\infty)}$-adapted), $\WW$-a.s. $(l,\infty)$-valued ($\because$ $\tau_l\leq l$ (even with certainty), $\WW(\tau_l=l)=0$, $\tau_\cdot$ is $\uparrow$ and right-continuous), having $\tau_{T_l}=T_l$ a.s.-$\WW$ ($\because$ with $\WW$-probability one $\tau_\cdot$ only increases on the diagonal) and satisfying, for all $t\in (0,\infty)$, $\{T_l\leq t\}=\{\tau_{t}>l\}$   a.s.-$\WW$ ($\because$ $\WW(\tau_t=l)=0$) and $\tau_t\geq T_l$ a.s.-$\WW$ on $\{T_l\leq t\}$ ($\because$ $\tau_\cdot$ is $\uparrow$ so that $\tau_t\geq \tau_{T_l}=T_l$ a.s.-$\WW$). 
 
Since the filtration $\FF^{0,\rightarrow}$ is Brownian, hence predictable, we may also, and do prepare a sequence $(T^n)_{n\in \mathbb{N}}$ of $\FF^{0,\rightarrow}$-stopping times satisfying $T^n\uparrow T_l$ as $n\to \infty$ and $T^n<T_l$ for all $n\in \mathbb{N}$ a.s.-$\WW$. Since $\WW(T_l>l )=1$ we may and do ask that $T^n\geq l$ for all $n\in \mathbb{N}$ (by passing to $(T^n\lor l)_{n\in \mathbb{N}}$ in lieu of $(T^n)_{n\in \mathbb{N}}$ if necessary).
 

With these preparations in hand, take $\{l,r\}\subset [0,\infty)$, $L\in \FF_{-\infty,l}$, $H\in \mathcal{B}_{\Theta_0}$ and compute as follows:
\begin{align}
\nonumber&\tilde \WW(L,l<\tau_\mathsf{e},\tau_\mathsf{e}\in M',\tau_\mathsf{e}<\mathsf{e}-r,\Delta_{\mathsf{e}-r}\vert_{[0,\infty)}\in H)\\\nonumber
&=\int_0^\infty \dd t \lambda e^{-\lambda t}\WW(L,l<\tau_t,\tau_t\in M',\tau_t<t-r,\Delta_{t-r}\vert_{[0,\infty)}\in H)\quad \text{(just the independence of $\mathsf{e}$ \& $B$ and Tonelli)}\\\nonumber
&=\int_0^\infty \dd t \lambda e^{-\lambda t}\WW(L,T_l\leq t,\tau_t\in M',\tau_t<t-r,\Delta_{t-r}\vert_{[0,\infty)}\in H)\quad \text{(since $\{T_l\leq t\}=\{l<\tau_t\}$ a.s.-$\WW$)}\\\nonumber
&=\lim_{n\to\infty}\int_0^\infty\dd t \lambda e^{-\lambda t}\WW\big(L,T^n\leq   t,T^n+\tau_{t-T^n}(\Delta_{T^n})\in T^n+(M'\cap (0,\infty))(\Delta_{T^n}),T^n+\tau_{t-T^n}(\Delta_{T^n})<t-r,\\\nonumber
&\qquad\qquad\qquad\qquad\qquad\qquad \qquad\qquad\qquad\qquad\qquad\qquad \qquad\qquad\qquad\qquad\qquad (\Delta_{(t-T^n-r)\lor 0}\vert_{[0,\infty)})(\Delta_{T^n})\in H\big)\\\nonumber
&\qquad \qquad \text{(by dominated convergence, perfect stationarity of $M'$ and by perfect honest indexation of $\tau_\cdot$,}\\\nonumber
&\qquad\qquad\qquad\qquad\qquad\qquad\qquad\qquad\qquad \quad \quad \text{using $\tau_t\geq T_l>T^n$ for all $n\in \mathbb{N}$ a.s.-$\WW$ on $\{T_l\leq t\}$)}\\\nonumber
&=\lim_{n\to\infty}\int_0^\infty\dd t \lambda e^{-\lambda t}\int_{[0,t]} \WW(L,T^n\in \dd h)\WW(\tau_{t-h}\in M',\tau_{t-h}<t-h-r,\Delta_{t-h-r}\vert_{[0,\infty)}\in H)\\\nonumber
&\qquad \qquad \text{(by Lemma~\ref{lemma:stopping-time-conditioning} for the $\FF^\rightarrow$-stopping time $T^n$, noting that $L\in \FF_{-\infty,l}$, $l\leq T^n$ and exploiting $\tau_\cdot$}\\\nonumber
& \qquad \qquad\text{being (resp. $M'\cap(0,\infty)$ having a) $(B\vert_{[0,\infty)})^{-1}(\HH)$-measurable (resp. perfect enumeration))}\\\nonumber
&=\lim_{n\to\infty}\int_{[0,\infty)} \WW(L,T^n\in \dd h) e^{-\lambda h}\int_h^\infty\dd t \lambda e^{-\lambda (t-h)}\WW(\tau_{t-h}\in M',\tau_{t-h}<t-h-r,\Delta_{t-h-r}\vert_{[0,\infty)}\in H)\quad \text{(Tonelli)}\\\nonumber
&=\lim_{n\to\infty} \tilde\WW(L,T^n< \mathsf{e})\int_0^\infty\dd t \lambda e^{-\lambda t}\WW(\tau_{t}\in M',\tau_{t}<t-r,\Delta_{t-r}\vert_{[0,\infty)}\in H)\quad \text{(elementary substitution of}\\\nonumber
&\qquad\qquad\text{$t-h\rightsquigarrow h$: it is here where the memoryless property of the exponential distribution intervenes}\\\nonumber
&\qquad\qquad\text{crucially, yielding the factorization)}\\\nonumber
&= \tilde\WW(L,T_l\leq \mathsf{e})\int_0^\infty\dd t \lambda e^{-\lambda t}\WW(\tau_{t}\in M',\tau_{t}<t-r,\Delta_{t-r}\vert_{[0,\infty)}\in H)\quad \text{(continuity of $\tilde\WW$)}\\\nonumber
&=\tilde\WW(L,l<\tau_\mathsf{e})\int_0^\infty\dd t \lambda e^{-\lambda t}\WW(\tau_{t}\in t+M'(\Delta_t),\tau_{t}<t-r,\Delta_{t-r}\vert_{[0,\infty)}\in H)\quad \text{(using stationarity of $M'$}\\\nonumber
&\qquad\qquad \text{and the fact that $\{T_l\leq \mathsf{e}\}=\{l<\tau_\mathsf{e}\}$ a.s.-$\tilde\WW$ /by independence of $\mathsf{e}$ and $B$/)}\\\nonumber
&=\tilde\WW(L,l<\tau_\mathsf{e})\int_0^\infty\dd t \lambda e^{-\lambda t}\WW(t-\tau_{t}\in -M'(\Delta_t),r<t-\tau_{t},\Delta_{t-r}\vert_{[0,\infty)}\in H)\quad \text{(just a slight rearrangement)}\\\nonumber
&=\tilde\WW(L,l<\tau_\mathsf{e})\int_0^\infty\dd t \lambda e^{-\lambda t}\WW(\hat{\tau}_{t}(\mathsf{r}\circ \Delta_t)\in (-M'\circ \mathsf{r})(\mathsf{r}\circ \Delta_t),r<\hat{\tau}_{t}(\mathsf{r}\circ \Delta_t),((\mathsf{r}\circ \Delta_{r})\vert_{[0,\infty)})(\mathsf{r}\circ \Delta_t)\in H))\\\nonumber
&\qquad \qquad \text{(since $t-\tau_{t}=\hat{\tau}_t(\mathsf{r}\circ \Delta_t)$, $\mathsf{r}\circ \mathsf{r}=\mathrm{id}_{\Omega_0}$ and $\mathsf{r}\circ \Delta_r=\Delta_{-r}\circ\mathsf{r}$)}\\\nonumber
&=\tilde\WW(L,l<\tau_\mathsf{e})\int_0^\infty\dd t \lambda e^{-\lambda t}\WW(\hat{\tau}_{t}\in -M'\circ \mathsf{r},r<\hat{\tau}_{t},(\mathsf{r}\circ \Delta_{r})\vert_{[0,\infty)}\in H)\quad\text{($\mathsf{r}\circ \Delta_t$ is measure-preserving for $\WW$)}\\\label{eq:big-display-intermediate}
&=\tilde\WW(L,l<\tau_\mathsf{e})\tilde\WW((\mathsf{r}\circ \Delta_{r})\vert_{[0,\infty)}\in H,r<\hat\tau_\mathsf{e},\hat{\tau}_\mathsf{e}\in -M'\circ \mathsf{r})\quad \text{(independence of $\mathsf{e}$ \& $B$ and Tonelli /bis/)}
\\\label{eq:big-display}
&=\tilde\WW(L,l<\tau_\mathsf{e})\tilde\WW(r<\hat\tau_\mathsf{e},(\mathsf{r}\circ \Delta_{r})\vert_{[0,\infty)}\in H)\tilde \WW({\tau}_\mathsf{e}\in M'),
\end{align}
where in the last line we have used the fact that $-M'\circ \mathsf{r}$ (resp. $-M\circ \mathsf{r}$) is also a shift-invariant local random countable set (resp. of which $\hat{\tau}$ is an honest indexation, nice in the sense of Lemma~\ref{lemma:perfect-good}), while $\{(\mathsf{r}\circ \Delta_{r})\vert_{[0,\infty)}\in H\}\in \FF_{-\infty,r}$, so that \eqref{eq:big-display-intermediate}  can be recycled according to the following substitutions: $$\text{$r\rightsquigarrow0$, $H\rightsquigarrow \Theta_0$, $M\rightsquigarrow -M\circ \mathsf{r}$, $M'\rightsquigarrow -M'\circ \mathsf{r}$, $l\rightsquigarrow r$, $L\rightsquigarrow \{(\mathsf{r}\circ \Delta_{r})\vert_{[0,\infty)}\in H\}$, $\tau_\cdot\rightsquigarrow \hat\tau_\cdot$,}$$  noting finally that $\widehat{\hat\tau}=\tau$, $\tilde \WW(0<\tau_\mathsf{e})=1$ and $-(-M'\circ\mathsf{r})\circ\mathsf{r}=M'$ (this ``trick'' just saves us from having to redo a computation that really we have already done, albeit subject to the preceding substitutions). Taking 
$l=0$, $L=\Omega_0$, $M'=M$ in \eqref{eq:big-display} we get (since  /by inclusion and independence/ $\tilde\WW(\tau_\mathsf{e}\in M)=1$) 
\begin{equation}\label{wiener-hopf:equal-in-law}
\tilde \WW(\tau_\mathsf{e}<\mathsf{e}-r,\Delta_{\mathsf{e}-r}\vert_{[0,\infty)}\in H)=\tilde\WW(r<\hat\tau_\mathsf{e},(\mathsf{r}\circ \Delta_{r})\vert_{[0,\infty)}\in H);
\end{equation} plugging \eqref{wiener-hopf:equal-in-law} back into \eqref{eq:big-display} we conclude that 
$$\tilde \WW(L,l<\tau_\mathsf{e},\tau_\mathsf{e}\in M',r<\mathsf{e}-\tau_\mathsf{e},\Delta_{\mathsf{e}-r}\vert_{[0,\infty)}\in H)=\tilde \WW(L,l<\tau_\mathsf{e})\tilde \WW(\tau_\mathsf{e}\in M')\tilde \WW(r<\mathsf{e}-\tau_\mathsf{e},\Delta_{\mathsf{e}-r}\vert_{[0,\infty)}\in H).$$

By approximation, linearity and bounded convergence 
it follows that $$\tilde \WW[Z_{\tau_{\mathsf{e}}};\tau_\mathsf{e}\in M',r<\mathsf{e}-\tau_\mathsf{e},\Delta_{\mathsf{e}-r}\vert_{[0,\infty)}\in H]=\tilde \WW[Z_{\tau_\mathsf{e}}]\tilde \WW(\tau_\mathsf{e}\in M')\tilde \WW(r<\mathsf{e}-\tau_\mathsf{e},\Delta_{\mathsf{e}-r}\vert_{[0,\infty)}\in H)$$ for all  bounded real left-continuous $\FF^\rightarrow$-adapted processes $Z$. By functional monotone class we deduce that $$\tilde \WW(\mathsf{L},\tau_\mathsf{e}\in M',r<\mathsf{e}-\tau_\mathsf{e},\Delta_{\mathsf{e}-r}\vert_{[0,\infty)}\in H)=\tilde \WW(\mathsf{L})\tilde \WW(\tau_\mathsf{e}\in M')\tilde \WW(r<\mathsf{e}-\tau_\mathsf{e},\Delta_{\mathsf{e}-r}\vert_{[0,\infty)}\in H),$$ 
which is to say
$$ \tilde \WW(\mathsf{L},\tau_\mathsf{e}\in M',r<\mathsf{e}-\tau_\mathsf{e},\Delta_{\mathsf{e}-\tau_\mathsf{e}-r}(\Delta_{\tau_\mathsf{e}}\vert_{[0,\infty)})\in H)=\tilde \WW(\mathsf{L})\tilde \WW(\tau_\mathsf{e}\in M')\tilde \WW(r<\mathsf{e}-\tau_\mathsf{e},\Delta_{\mathsf{e}-\tau_\mathsf{e}-r}(\Delta_{\tau_\mathsf{e}}\vert_{[0,\infty)})\in H),$$ 
this for all $\mathsf{L}\in \FF_{-\infty,\tau_\mathsf{e}}$. Besides, the class of sets 
$$\left\{\{(t,\omega)\in [0,\infty)\times \Theta_0:r<t,\, \Delta_{t-r}(\omega)\in H\}:(r,H)\in [0,\infty)\times \mathcal{B}_{\Theta_0}\right\}\subset 2^{[0,\infty)\times \Theta_0}$$
is a $\pi$-system that generates $\mathcal{B}_{[0,\infty)}\otimes \mathcal{B}_{\Theta_0}$ on $[0,\infty)\times \Theta_0$. 
By an application of Dynkin's lemma (``it is enough to check independence on generating $\pi$-systems'') we get \ref{splitting:i} and in fact the joint independence  of the following triplet: $$\FF_{-\infty,\tau_{\mathsf{e}}};\quad \{\tau_\mathsf{e}\in M'\}; \quad (\mathsf{e}-\tau_\mathsf{e},\Delta_{\tau_\mathsf{e}}\vert_{[0,\infty)}).$$ Since $\FF_{-\infty,\tau_{\mathsf{e}}}$ and $(\mathsf{e}-\tau_\mathsf{e},\Delta_{\tau_\mathsf{e}}\vert_{[0,\infty)})$ generate the whole of the $\sigma$-field of $\tilde\WW$ we infer at once \ref{splitting:ii}. Finally, as probability laws are uniquely determined by their values on a generating $\pi$-system (which, incidentally, follows again from Dynkin's lemma) we deduce  from a rewriting of \eqref{wiener-hopf:equal-in-law}, namely
$$\tilde \WW(r<\mathsf{e}-\tau_\mathsf{e},\Delta_{\mathsf{e}-\tau_\mathsf{e}-r}(\Delta_{\tau_\mathsf{e}}\vert_{[0,\infty)})\in H)=\tilde\WW(r<\hat\tau_\mathsf{e},\Delta_{\hat\tau_\mathsf{e}-r}((\mathsf{r}\circ \Delta_{\hat\tau_\mathsf{e}})\vert_{[0,\infty)})\in H),$$
 also the validity of \ref{splitting:three}.
\end{proof}
Call an honest indexation $\tau$ of a  random countable set $M$  symmetric if $\hat{\tau}_{0,t}=\tau_{0,t}$ a.s.-$\WW$ for all $t\in (0,\infty)$.  For such an indexation Theorem~\ref{theorem:splitting}\ref{splitting:three} becomes $(\mathsf{e}-\tau_\mathsf{e},\Delta_{\tau_\mathsf{e}}\vert_{[0,\infty)})_\star \tilde\WW=(\tau_\mathsf{e},(\mathsf{r}\circ \Delta_{\tau_\mathsf{e}})\vert_{[0,\infty)})_\star \tilde\WW$. 
\begin{question}
 Examples of symmetric honest indexations include the standard indexation of the local minima (that of Example~\ref{example:index-min} with $\kappa=0$) and the corresponding standard indexation of  the local maxima. Are there any others?
\end{question}
\subsection{Thickness, local times and excursions}
Return to Question~\ref{question:stopping-times-belonging}. In our second substantial result of this section, let us show that for random sets admitting an honest indexation, no stopping time can belong to them with positive probability. In the terminology of \cite[Definition~5.1]{monique} such sets are thick.
\begin{theorem}\label{theorem:stopping-times-no}
Let $M$ be a random set admitting an honest indexation $\tau$ and let $S$ be an $(\FF_{-\infty,t})_{t\in \mathbb{R}}$-stopping time (meaning: $S$ is $[-\infty,\infty]$-valued and  $\{S\leq a\}\in \FF_{-\infty,a}$ for all $a\in \mathbb{R}$). Then $\WW(S\in M)=0$.
\end{theorem}
By reflection of time $\WW(S\in M)=0$ for any $(\FF_{t,\infty})_{t\in \mathbb{R}}$-reverse stopping time $S$ (meaning: $S$ is $[-\infty,\infty]$-valued and  $\{S\geq a\}\in \FF_{a,  \infty}$ for all $a\in \mathbb{R}$), just the same. Theorem~\ref{theorem:stopping-times-no} and the preceding statement also generalize trivially to any $M$ which is merely a countable union of random sets, each admitting its own honest indexation.
\begin{proof}
Suppose per absurdum $\WW(S\in M)>0$. By stationarity we reduce at once to the case when $S$ is $(0,\infty)$-valued and we assume this henceforth.

First we pass to a nice version of $\tau$ that will be ideally suited to the problem at hand. Let $Q:=\{\frac{k}{2^n}:(k,n)\in \mathbb{Z}\times \mathbb{N}_0\}$ be the dyadic numbers. We may and do insist that $\tau_{p,q}$ is $[p,q]$-valued for each pair $p<q$ from $Q$. The family $\tau\vert_{Q\times Q}$ is a.s.-$\WW$ nondecreasing in both its ``coordinates''. We may and do ask further then that it is nondecreasing in both its ``coordinates'' with certainty (by setting e.g. $\tau_{p,q}=\frac{p+q}{2}$ for $p<q$ from $Q$ on the exceptional set on which it fails). Put next $$\hat\tau_{u,v}:=\sup\{\tau_{p,q}:(p,q)\in Q^2, p<q,p< u, q< v\},\quad (u,v)\in \mathbb{R}^2,\, u<v.$$ Then $\hat\tau:=(\hat\tau_{u,v})_{(u,v)\in \mathbb{R}^2,u<v}$ is jointly left-continuous, nondecreasing and $\tau_{u,v}\in [u,v]$ for all real $u<v$ by construction. Furthermore, $\hat\tau_{u,v}=\tau_{u,v}$ a.s.-$\WW$ for all real $u<v$: certainly $\hat\tau_{u,v}\leq \tau_{u,v}$ a.s.-$\WW$; for the reverse inequality note that a.s.-$\WW$, $\tau_{u,v}=\tau_{u,q}$ for some $q\in Q\cap (u,v)$ and then up to a L\'evy shift and a reflection of time $\mathsf{r}$, both of which are measure-preserving for $\WW$, the argument reduces to the observation that $\hat\tau'_{0,\cdot}$ is a version of $\hat\tau_{0,\cdot}$ in the proof of Lemma~\ref{lemma:perfect-good} (the $\cdot$ here and in the continuation of the proof below runs over values in $(0,\infty)$).  We may and do assume then that $\tau$ was jointly left-continuous, nondecreasing and that $\tau_{u,v}\in [u,v]$ for all real $u<v$ to begin with. 

Second, notice that for a deterministic random time $T$, $(\Delta_T)^{-1}(\FF_{0,\infty})$ is $\WW$-independent of $\FF_{-\infty,T}$ and $(\Delta_T)_\star\WW=\WW$ on $\FF_{0,\infty}$, which extends, on $\{T\in \mathbb{R}\}$, to an arbitrary  $(\FF_{-\infty,t})_{t\in \mathbb{R}}$-stopping time $T$ that assumes at most countably many values with $\WW$-probability one (notation: $\FF_{-\infty,T}:=\{A\in \FF_{-\infty,\infty}:A\cap \{T\leq t\}\in \FF_{-\infty,t}\text{ for all $t\in \mathbb{R}$}\}$). This --- let us call it the simple Markov property --- will be used below in conjunction with the fact that $\tau_{0,\cdot}$ is $\FF_{0,\infty}$-measurable.

Since the filtration $\FF^\rightarrow$ is Brownian, therefore predictable, there is a sequence of $\FF^\rightarrow$-stopping times $(T_n)_{n\in \mathbb{N}}$ that is nondecreasing to $S$ and such that  $T_n<S$ for all $n\in \mathbb{N}$. For each $n \in\mathbb{N}$ there is then a $\FF^\rightarrow$-stopping time $S_n$ valued in $Q$ such that $T_n\leq S_n\leq T_n+2^{-n}$ and such that $\WW(S_n> S)\leq 2^{-n}$. By Borel-Cantelli it follows that $S_n> S$ for at most finitely many $n\in \mathbb{N}$ a.s.-$\WW$, hence by left-continuity of $\tau$ we have that $\lim_{n\to\infty}\tau_{S_n,S_n+t}=\tau_{S,S+t}$ for all $t\in (0,\infty)$ a.s.-$\WW$. At the same time, for each $n\in \mathbb{N}$, since $S_n$ assumes only countably many values, we have that $\tau_{S_n,S_n+\cdot}-S_n=\tau_{0,\cdot}(\Delta_{S_n})$ a.s.-$\WW$ and hence by the simple Markov property  $\tau_{S_n,S_n+\cdot}-S_n$ has the same law as $\tau_{0,\cdot}$ under $\WW$. So, the $\WW$-law of $\tau_{S,S+\cdot}-S$ is that of $\tau_{0,\cdot}$ too. In particular, $\tau_{S,S+t}>S$ for all $t\in (0,\infty)$ a.s.-$\WW$.

Now, on the event that $\{S\in M\}$, a.s.-$\WW$, for some pair $p<q$ from $Q$, $p<S=\tau_{p,q}<q$ and hence $S=\tau_{r,q}$ for all $r\in (p,S)\cap Q$, which in turn by left-continuity yields $S=\tau_{S,q}$, contradicting the conclusion of the preceding paragraph.
\end{proof}

The preceding result allows to develop a theory of ``excursions of $B$ relative to a random set admitting an honest indexation $\tau$'' (more precisely, from an associated set, $D$, to be introduced presently), which generalizes the excursions of $B$ from its running minimum in the case of $\tau$ being the standard indexation of the local minima of $B$. Though we shall not find any immediate use for it, this ``It\^o's excursion point of view'' is at our fingertips, and since we also find it interesting in its own right and may prove useful in future explorations of the subject, we provide the details.

For the remainder of this section let then $\tau$ be an honest indexation of a random set $M$ and assume that the process $\tau_{0,\cdot}=(\tau_{0,t})_{t\in (0,\infty)}$ is right-continuous, $(B\vert_{[0,\infty)})^{-1}(\HH)$-measurable, $(0,\infty)$-valued, nondecreasing, majorized by $\mathrm{id}_{(0,\infty)}$ and has $\tau_{0,\tau_{0,t}}=\tau_{0,t}$ for all $t\in (0,\infty)$ (this we may ask for according to Lemma~\ref{lemma:perfect-good} by passing to a suitable version of $\tau$). The right-continuity and  $(B\vert_{[0,\infty)})^{-1}(\HH)$-measurability are important. The other properties we ask for in order to avoid some a.s. qualifiers and ease the sailing below, but are otherwise not consequential (they hold a.s. anyway). Introduce $$D:=\{t\in (0,\infty):t=\tau_{0,t}\}=\{\tau_{0,t}:t\in (0,\infty)\}\in (B\vert_{[0,\infty)})^{-1}(\HH)\otimes \mathcal{B}_{(0,\infty)},$$ which is a closed $\FF^{0,\rightarrow}$-optional subset of $(0,\infty)$ with $0$ as an accumulation point. 

An important property to note at once concerning the relation of $D$ to $M$ is as follows. Set $$R_t:=E_t-t:=\inf\{s\in (t,\infty):s\in D\}-t,\quad t\in (0,\infty),$$ and $$G:=\{t\in(0,\infty):R_{t-}=0, R_t>0\}=\{t\in D:R_t>0\}$$ (the set of the left end-points of the intervals contiguous to $D=\{t\in (0,\infty):R_{t-}=0\}$, i.e. the set of those points of $D$, which are isolated on the right in $D$). Then $$M\cap D=G\text{ a.s.-$\WW$}.$$ For, if $t\in D$ and $t$ is isolated from the right in $D$, then there is a rational $p\in (t,\infty)$, such that $t=\tau_{0,t}=\tau_{0,p}$; conversely, a.s.-$\WW$, if $t\in M\cap D$, then  $p<t=\tau_{p,q}<q$ for some rational $p<q$ and consequently $\tau_{0,q}=\tau_{p,q}=t=\tau_{0,t}$, implying that $t$ is isolated on the right in $D$. 

A second immediate observation that we can make is that $(0,\infty)\backslash D$ is dense in $(0,\infty)$ a.s.-$\WW$: the converse would indeed imply that for some rational $p\in(0,\infty)$ with positive $\WW$-probability $\tau_{0,p}=p$, which cannot be.

The reader may also wonder at this point whether or not the set $D$ has to be unbounded a.s.-$\WW$. The answer is no: just take $\kappa>0$ in Example~\ref{example:index-min}, in which case $D$ is bounded a.s.-$\WW$.

\begin{proposition}\label{proposition:property-of-T}
Let $T$ be an $\FF^{0,\rightarrow}$-stopping time. Then  
\begin{equation}\label{eq:regenration_D}
D\cap (T,\infty)=T+D(\Delta_{T})\text{ a.s.-$\WW$ on $\{T\in D\}$},
\end{equation}
\end{proposition}
We may remark that \eqref{eq:regenration_D} implies, in particular, that $D\cap (E_t,\infty)=E_t+D(\Delta_{E_t})\text{ a.s.-$\WW$}$ and hence that $D\cap (0,E_t]$ is independent  of $[D\cap (E_t,\infty)]-E_t$ on  $\{E_t<\infty\}$, this for each $t\in (0,\infty)$. Thus $D$ is a regenerative set in the sense  of \cite[p.~1]{regenerative-sets} (with a trivial, constant, ``modulating'' process $X$ -- but not with $X=B$!).
\begin{proof}
Theorem~\ref{theorem:stopping-times-no}, together with the observation that $M\cap D$ are a.s.-$\WW$ those points of $D$ that are  isolated on the right, entails that $\tau_{0,T+\cdot}-T$ is strictly positive a.s.-$\WW$ on $\{T\in D\}$. From Lemma~\ref{lemma:perfect-good} we hence get that  $$\tau_{0,T+\cdot}=T+\tau_{0,\cdot}(\Delta_{T})\text{ a.s.-$\WW$ on }\{T\in D\},$$ provided a nice enough version of $\tau_{0,\cdot}$ is chosen in the stipulated equality. The claim follows (because $\WW(\Delta_T\in A)=0$ for $\WW$-negligible $A\in (B\vert_{[0,\infty)})^{-1}(\HH)$, so that once can pass back to the given version).
\end{proof}
We aim next to construct a local time for $D$ with a certain nice regenerative property relative to the L\'evy shifts. We turn to this after preparing
\begin{lemma}\label{projections}
Let $\zeta$ be a $\UU$-stopping time. For a process $X:\Theta_0\times [0,\infty)\to \mathbb{R}$ that is $(\HH\otimes \mathcal{B}_{[0,\infty)})/\mathcal{B}_\mathbb{R}$-measurable put $$\widetilde X:=X^\zeta+\widehat X:=X^\zeta+[(X-X_0)(\Delta_\zeta)]_{\cdot-\zeta}\mathbbm{1}_{\llbracket\zeta,\infty\rrparenthesis}.$$ 
\begin{enumerate}[(i)]
\item\label{projections:i} If $X$ is optional, then so is $\widehat X$ (and, hence, $\widetilde X$).
\item\label{projections:ii} If $X$ is a right-continuous martingale, the running supremum of the absolute value of which is $\PPP$-integrable at finite deterministic times, then $\hat{X}$ (and, hence, $\widetilde X$) has these same properties.
\item\label{projections:iii}  If $X$ is  right-continuous and bounded on bounded time intervals, then $\widehat X$ has the same properties and ${}^o\widehat X=\widehat{{}^oX}$. 
\item\label{projections:iv} If $X$ is nondecreasing, right-continuous and bounded on bounded time intervals then  $\widetilde X$ has these properties also and $(\widetilde X)^p=\widetilde{X^p}$ a.s.-$\PPP$. 
\end{enumerate}
Here ${}^oX$ (resp. $X^p$) denotes the $(\UU,\PPP)$-optional (resp. dual predictable) projection of $X$.
\end{lemma}
\begin{proof}
All notions of the general theory of stochastic process in this proof are relative to the pair $(\UU,\PPP)$. The filtration being Brownian the optional and predictable $\sigma$-field coincide and we shall use this without special mention. We may and do assume $\PPP(\zeta<\infty)>0$ (for otherwise $\widehat X=0$ \& $\widetilde X=X$ a.s.-$\PPP$ and the matter is trivial).

\ref{projections:i}. By a monotone class argument we reduce  to checking it in the case when $X=\mathbbm{1}_{A\times [u,\infty)}$ with $A\in \UU_u$, $u\in (0,\infty)$. In that case $$\widehat X_t=\mathbbm{1}_{\{\zeta+u\leq t,\Delta_\zeta\in A\}},\quad t\in [0,\infty),$$ so that $\widehat X$ is right-continuous and adapted ($\because$ $\zeta+u$ is a stopping time and $(\Delta_\zeta)^{-1}(\UU_u)\subset \UU_{\zeta+u}$), hence optional.

\ref{projections:ii}. We may and do assume $X_0=0$. By \ref{projections:i} $\widehat{X}$ is adapted. Let $s_1\leq s_2$ be from $[0,\infty)$. On the one hand, on $\{s_1\leq \zeta\}$, 
\begin{align*}
\PPP\left[\widehat{X}_{s_2}\vert \UU_{s_1}\right]&=\PPP\left[X(\Delta_\zeta)_{s_2-\zeta}\mathbbm{1}_{\{\zeta\leq s_2\}}\vert \UU_{s_1}\right]=\PPP\left[\PPP\left[X(\Delta_\zeta)_{s_2-\zeta}\mathbbm{1}_{\{\zeta\leq s_2\}}\vert \UU_\zeta\right]\vert\UU_{s_1}\right]\\
&=0=\widehat{X}_{s_1}
\end{align*}
a.s.-$\PPP$, since $\Delta_\zeta$ is independent of $\UU_\zeta$ under $\PP(\cdot\vert \zeta<\infty)$, $\zeta$ is $\UU_\zeta$-measurable, $(\Delta_\zeta)_\star \PPP(\cdot\vert \zeta<\infty)=\PPP$ and $\PPP[X_u]=0$ for all $u\in [0,\infty)$. On the other hand, for all $u\in [0,\infty)$, 
$\PPP[X_{u+s_2-s_1}\vert \UU_{u}]=X_u$ a.s.-$\PPP$, hence, on $\{\zeta<\infty\}$, a.s.-$\PPP$,
\begin{align*}
X(\Delta_\zeta)_u&=\left(({\Delta_\zeta}_\star\PPP(\cdot\vert\zeta<\infty))[X_{u+s_2-s_1}\vert\UU_u]\right)(\Delta_\zeta)=\PPP(\cdot\vert\zeta<\infty)[X(\Delta_\zeta)_{u+s_2-s_1}\vert \Delta_\zeta^{-1}(\UU_u)]\\
&=\PPP(\cdot\vert\zeta<\infty)[X(\Delta_\zeta)_{u+s_2-s_1}\vert({\Delta_\zeta}^u)^{-1}(\mathcal{B}_{\Theta_0})]=\PPP(\cdot\vert\zeta<\infty)[X(\Delta_\zeta)_{u+s_2-s_1}\vert\UU_\zeta\lor ({\Delta_\zeta}^u)^{-1}(\mathcal{B}_{\Theta_0})]\\
&=\PPP[X(\Delta_\zeta)_{u+s_2-s_1}\vert\UU_\zeta\lor ({\Delta_\zeta}^u)^{-1}(\mathcal{B}_{\Theta_0})]
\end{align*}
(using quasi left-continuity and predictability of $\UU$ together with the Blackwell property of $(\Theta_0,\mathcal{B}_{\Theta_0})$ it is possible to see \cite[Corollary 7.20]{jacka} that $\UU_\zeta\lor ({\Delta_\zeta}^u)^{-1}(\mathcal{B}_{\Theta_0})=\UU_{\zeta+u}$ but we do not need it). Therefore
$$\PPP[X(\Delta_\zeta)_u\vert \UU_{s_1}]=\PPP[X(\Delta_\zeta)_{u+s_2-s_1}\vert\UU_{s_1}]$$
a.s.-$\PPP$ on $\{u\geq s_1-\zeta,\zeta\leq s_1\}$. Take now an approximating sequence of stopping times $(\zeta_k)_{k\in \mathbb{N}}$ for $\zeta$ with the following property: for each $k\in \mathbb{N}$, $\zeta_k$ takes on only countably many values, $\zeta_k\to \zeta$ as $k\to\infty$ and $\zeta_k> \zeta$ for at most finitely many $k\in \mathbb{N}$ a.s.-$\PPP$ (we have seen the method for the construction of such a sequence in the proof of Theorem~\ref{theorem:stopping-times-no}). 
Then, for each $k\in \mathbb{N}$, on $\{\zeta\leq s_1,\zeta_k\leq \zeta\}$, a.s.-$\PPP$, 
$$\PPP[X(\Delta_\zeta)_{s_1-\zeta_k}\vert \UU_{s_1}]=\PPP[X(\Delta_\zeta)_{s_2-\zeta_k}\vert\UU_{s_1}];$$ on letting $k\to \infty$ we deduce that $$\widehat X_{s_1}=\PPP[X(\Delta_\zeta)_{s_1-\zeta}\vert \UU_{s_1}]=\PPP[X(\Delta_\zeta)_{s_2-\zeta}\vert\UU_{s_1}]=\PPP[\widehat X_{s_2}\vert \UU_{s_1}] $$ a.s.-$\PP$ on $\{\zeta\leq s_1\}$. Thus $\widehat X$ is indeed a martingale and the other properties are immediate.

\ref{projections:iii}. We may and do assume $X_0=0$. By \ref{projections:i} $\widehat{{}^oX}$ is adapted; it is also right-continuous. Both these properties are true of $\hat{X}$ also. With these two preliminary observations out of the way we compute that, for all $u\in [0,\infty)$, on $\{\zeta<\infty\}$, a.s.-$\PPP$, 
\begin{align*}
{}^oX_u(\Delta_\zeta)&=\PPP[X_u\vert \UU_u](\Delta_\zeta)=\left(({\Delta_\zeta}_\star\PPP(\cdot\vert\zeta<\infty))[X_{u}\vert\UU_u]\right)(\Delta_\zeta)=\PPP(\cdot\vert\zeta<\infty)[X_u(\Delta_\zeta)\vert ({\Delta_\zeta}^u)^{-1}(\mathcal{B}_{\Theta_0})]\\
&=\PPP(\cdot\vert\zeta<\infty)[X_u(\Delta_\zeta)\vert\UU_\zeta\lor ({\Delta_\zeta}^u)^{-1}(\mathcal{B}_{\Theta_0})]=\PPP[X_u(\Delta_\zeta)\vert\UU_\zeta\lor ({\Delta_\zeta}^u)^{-1}(\mathcal{B}_{\Theta_0})].
\end{align*}
Therefore, for all $t\in [0,\infty)$, 
$$\PPP[{}^oX_u(\Delta_\zeta)\vert \UU_{t}]=\PPP[X_u(\Delta_\zeta)\vert \UU_{t}]$$ a.s.-$\PPP$ on $\{u\geq t-\zeta,\zeta\leq t\}$. Taking $(\zeta_k)_{k\in \mathbb{N}}$ as in \ref{projections:ii} we get that for each $k\in \mathbb{N}$, on $\{\zeta\leq t,\zeta_k\leq \zeta\}$, a.s.-$\PPP$, 
$$\PPP[{}^oX_{t-\zeta_k}(\Delta_\zeta)\vert \UU_{t}]=\PPP[X_{t-\zeta_k}(\Delta_\zeta)\vert \UU_{t}];$$
on letting $k\to \infty$ we deduce that 
$$\widehat{{}^oX}_t=\PPP[\widehat{{}^oX}_t\vert \UU_{t}]=\PPP[{}^oX_{t-\zeta}(\Delta_\zeta)\vert \UU_{t}]=\PPP[X_{t-\zeta}(\Delta_\zeta)\vert \UU_{t}]=\PPP[\widehat{X}_t\vert\UU_t]={}^o\widehat{X}_t$$ a.s.-$\PPP$ on $\{ \zeta\leq t\}$. This relation being trivial on $\{\zeta<t\}$ the claim follows. 

\ref{projections:iv}. Since the dual predictable projection commutes with stopping and is linear it suffices to establish that $(\widehat X)^p=\widehat{X^p}$ a.s.-$\PPP$. We may and do assume $X_0=0$. By \cite[Theorem~VI.21.4]{rogers2000diffusions} (and deterministic localization) ${}^oX-X^p$ is a right-continuous martingale having the integrability property of \ref{projections:ii}. Therefore $\reallywidehat{{}^oX-X^p}=\widehat{{}^oX}-\widehat{X^p}={}^o\widehat{X}-\widehat{X^p}$ is also a right-continuous martingale having the integrability property of \ref{projections:ii}, where we have used in addition \ref{projections:iii}. By another application of \cite[Theorem~VI.21.4]{rogers2000diffusions} (and deterministic localization) the desired conclusion follows.
\end{proof}

\begin{proposition}\label{proposition:local-time}
There exists a, $\WW$-a.s. unique up to a multiplicative constant from $(0,\infty)$, continuous nondecreasing $\FF^{0,\rightarrow}$-adapted, $(B\vert_{[0,\infty)})^{-1}(\HH)$-measurable real process $L=(L_t)_{t\in [0,\infty)}$ vanishing at zero, such that $\mathrm{supp}(\dd L)\cap (0,\infty)=D$ a.s.-$\WW$ and such that   for any $\FF^{0,\rightarrow}$-stopping time $T$ the L\'evy shift regenerative property
\begin{equation}\label{eq:local-time-regeneration}
L_{T+\cdot}=L_T+L(\Delta_T)\text{ a.s.-$\WW$ on }\{T\in D\}
\end{equation}
holds true.
\end{proposition}
In particular it means that $D$ has no isolated points a.s.-$\WW$. Also, with $\zeta$ the right-continuous inverse of $L$,
$$\zeta_t:=\inf\{s\in [0,\infty):L_s>t\},\quad t\in [0,\infty),$$ we have that for all $s\in [0,\infty)$, $$\zeta_{s+t}=\zeta_s+\zeta_t(\Delta_{\zeta_s})\text{ for all $t\in [0,\infty)$ a.s.-$\WW$ on $\{\zeta_s<\infty\}$}.$$ In turn it means that the bivariate process $(\zeta,B_\zeta)$, defined on the temporal interval $[0,L_\infty)$, is a possibly killed L\'evy process under $\WW$ in the filtration $\FF^{0,\rightarrow}_\zeta=(\FF^{0,\rightarrow}_{\zeta_t})_{t\in [0,\infty)}$. Especially, $L_\infty$ has an exponential distribution under $\WW$ (possibly degenerate, equal to $\delta_\infty$, of course) and $\zeta$ is a $(\FF^{0,\rightarrow}_\zeta,\WW)$-subordinator, the closure of the range of which is equal to $D$ a.s.-$\WW$.

\begin{proof} 
For the purposes of the proof we transfer everything relevant from $\WW$ under $\PPP$ in the obvious way, but do not bother to introduce new notation for $D$,  the $R_t$, $t\in (0,\infty)$, and $G$. To placate the reader as to this we note that by \cite[Lemma~1.19]{jacod2013limit} any stopping time of  $\FF^{0,\rightarrow}$ is a.s.-$\WW$ equal to  a stopping time of $((B\vert_{[0,\infty)})^{-1}(\UU_t))_{t\in  [0,\infty)}$ (because the latter filtration is right-continuous, which follows from Lemma~\ref{lemma:commute-decreasing-intersection-trace} and the right-continuity of $\UU$).

As in Lemma~\ref{projections} all notions of the general theory of stochastic process in this proof are relative to the pair $(\UU,\PPP)$.

Existence. The construction is standard; we follow \cite[Chapter~XX, Section~1, \# 11]{dellacherie-sets}.  Then set  
$$\mathcal{L}_t:= \mathscr{L}(D\cap (0,t])+\sum_{g\in G\cap (0,t]}(1-e^{-R_g}),\quad t\in [0,\infty).$$ The process $\mathcal{L}=(\mathcal{L}_t)_{t\in [0,\infty)}$ belongs to $(\HH\otimes \mathcal{B}_{[0,\infty)})/\mathcal{B}_{\mathbb{R}}$ and is nondecreasing, right-continuous, vanishing at zero,  bounded on bounded time intervals, furthermore, thanks to Proposition~\ref{proposition:property-of-T}, for any $\UU$-stopping time $T$, 
\begin{equation}\label{eq:local-time-regeneration-raw}
\mathcal{L}_{T+\cdot}=\mathcal{L}_T+\mathcal{L}_\cdot(\Delta_T)\text{ a.s.-$\PPP$ on }\{T\in D\}.
\end{equation}
 Denote by $L$ the dual predictable projection of $\mathcal{L}$. We get evidently a $\UU$-adapted real nondecresing right-continuous process vanishing at zero (or anyway can choose its version in this way). This process is (may be chosen) continuous because of the following argument: according to \cite[Remark on p.~132 in \# XX.11]{dellacherie-sets} the set of its discontinuities is $\PPP$-a.s. contained in $G_o$, which is a thin set (i.e. exhausted by a countable union of $\UU$-stopping times) \cite[Chapter~20, Section~1, (8.1)]{dellacherie-sets} that is contained in $G$; as discussed above $G=M\cap D$ a.s.-$\PPP$ and we conclude by applying Theorem~\ref{theorem:stopping-times-no} that $G_o=\emptyset$ a.s.-$\PPP$. 
 $L$ satisfies moreover \eqref{eq:local-time-regeneration}, which follows from \eqref{eq:local-time-regeneration-raw} and Lemma~\ref{projections}\ref{projections:iv} on taking $X=\mathcal{L}$ and $\zeta=T\mathbbm{1}_{\{T\in D\}}+\infty\mathbbm{1}_{\{T\notin D\}}$ therein. Finally, we check that $\mathbf{L}:=\mathrm{supp}(\dd L)\cap (0,\infty)=D$ a.s.-$\PPP$. On the one hand, the Lebesgue-Stieltjes measure $\dd \mathcal{L}$ is carried on $(0,\infty)$ by the predictable set $D$, therefore the same is true of $\dd L$, i.e. $D\supset \mathbf{L}$. Conversely, $\mathbf{L}$ is   the closure of the set of points of left increase of $L$, which is a predictable set (use \cite[VI.5, (6.29)]{rogers2000diffusions}) that carries $\dd L$ hence $\dd \mathcal{L}$ on $(0,\infty)$, therefore since $D$ is actually equal to the support of $\dd\mathcal{L}$ on $(0,\infty)$ we get also  $D\subset \mathbf{L}$.
 
 Uniqueness. The argument of \cite[Proposition~IV.5]{bertoin} having to do with the  local time for excursions away from a point of a Markov process applies, mutatis mutandis. We omit the details of the straightforward modification, noting only (i) that the role of the Markov property of \cite{bertoin} is played by the regenerative properties  \eqref{eq:regenration_D}-\eqref{eq:local-time-regeneration} of $D$ and $L$ relative to the L\'evy shifts, while (ii) the fact that the point from which the excursions are being looked at in \cite{bertoin} is regular and instantaneous corresponds to $0$ being an accumulation point of $D$ and to $(0,\infty)\backslash D$ being dense in $(0,\infty)$ a.s.-$\PPP$, respectively.
\end{proof}
With the continuous local time $L$, having the regenerative property, of Proposition~\ref{proposition:local-time}  in hand, the excursion structure of $B$ from the set $D$ follows easily. Simply define
\begin{equation*}
\epsilon_t:=
\begin{cases}
(\Delta_{\zeta_{t-}}\vert_{[0,\infty)})^{\zeta_t-\zeta_{t-}},&\text{ if }\zeta_t>\zeta_{t-}\\
\partial,&\text{ otherwise}
\end{cases},\quad t\in (0,\infty),
\end{equation*}
where $\partial$ is a cemetery (coffin) state. Then, under $\WW$, $\epsilon=(\epsilon_t)_{t\in (0,\infty)}$ is a Poisson point process in the filtration $\FF^{0,\rightarrow}_\zeta$ \cite[Theorem~3.1]{ito}.

 We leave it here at that, except for noting that the splitting result of Theorem~\ref{theorem:splitting} with $M'=M$ or $M'=\emptyset$ is not really surprising in view of (and could perhaps be  got from) the excursion-theoretic structure presented just now -- see e.g. \cite{greenwood-pitman} for a pleasant exposition of the Wiener-Hopf factorization based on excursion theory in the case when $\tau$ is the standard indexation of the local maxima.  Nevertheless, we certainly cannot  through this lense alone explain away the extra information of the splitting coming from an $M'$ whose relation to $M$ is not a priori trivial. It is precisely this extra information that we shall find immediate use for in the next section.
\section{Minimality}\label{section:minimality}
Here we wish to examine the question of ``minimality'' of a local stationary random countable set. 

\begin{proposition}
If $M_1$ and $M_2$ are two stationary local random countable sets then either $M_1\subset M_2$ a.s.-$\WW$  or with $\WW$-probability zero; in the former case either $M_1=M_2$ a.s.-$\WW$ or $M_2\backslash M_1$ is a stationary local dense random countable set.
\end{proposition}
\begin{proof}
For all $v\in \mathbb{R}$, $\Delta_v^{-1}(\{M_1\subset M_2\})=\{M_1\subset M_2\}$ a.s.-$\WW$, therefore $\WW(M_1\subset M_2)\in \{0,1\}$ (recall Remark~\ref{rmk:zero-one}). Clearly $M_2\backslash M_1$ is a stationary local random countable set, therefore is empty $\WW$-a.s. or dense $\WW$-a.s. by Proposition~\ref{proposition:dense-or-empty}.
\end{proof}
\begin{definition}
A stationary local dense random countable set that admits no proper  dense stationary local random countable subset with positive $\WW$-probability shall be called minimal. 
\end{definition}
By the preceding it is equivalent to ask that it does not decompose into the disjoint union of two dense local stationary random countable sets. The local extrema of $B$ are not minimal (they decompose into the local maxima and minima). Are the local maxima  (minima) of $B$ minimal? If not, then the answer to Tsirelson's question would have been to the affirmative at once. However, it could not have been  so easy, as we shall see.

\begin{theorem}\label{theorem:minimal}
Let $M$ be a dense stationary local random countable set admitting an honest indexation $\tau=(\tau_{s,t})_{(s,t)\in \mathbb{R}^2,s< t}$. If $(\dagger)$ $\WW(\tau_{0,t}\in M')$ is $\{0,1\}$-valued for arbitrarily small $t\in (0,\infty)$ for all stationary local random countable sets $M'$ contained in $M$ a.s.-$\WW$, then $M$ is minimal (the converse is also true, but trivial).
\end{theorem}
\begin{proof}
Let $M'$ be a stationary local random countable set contained in $M$ a.s.-$\WW$. 
For $\{s,t\}\subset (0,\infty)$, we argue as follows: 
 \begin{itemize}
 \item by nestedness and locality of $\tau$, a.s.-$\WW$, either $\tau_{0,s+t}\in (0,s)$ and then $\tau_{0,s+t}=\tau_{0,s}$ or $\tau_{0,s+t}\in (s,s+t)$ and then $\tau_{0,s+t}=\tau_{s,s+t}$ (recall that $\WW(\tau_{0,s+t}=s)=0$), hence $$\{\tau_{0,s}\in M'\}\cap \{\tau_{s,s+t}\in M'\}\subset \{\tau_{0,s+t}\in M'\};$$
 \item by locality of $M'$ and since $\tau_{0,s}$ is $\FF_{0,s}$-measurable, while $\tau_{s,t}$ is $\FF_{s,t}$-measurable, the two events $\{\tau_{0,s}\in M'\}\mathrel{\stackrel{\makebox[0pt]{\mbox{\normalfont{\fontsize{2.5}{4}\selectfont a.s.-$\WW$}}}}{=}}\{\tau_{0,s}\in M'\cap (0,s)\}$ and $\{\tau_{s,s+t}\in M'\}\mathrel{\stackrel{\makebox[0pt]{\mbox{\normalfont{\fontsize{2.5}{4}\selectfont a.s.-$\WW$}}}}{=}}\{\tau_{s,s+t}\in M'\cap (s,s+t)\}$ are respectively $\FF_{0,s}$- and $\FF_{s,s+t}$-measurable;
 \item by stationarity of $M'$ and $\tau$, $\{\tau_{s,s+t}\in M'\}=\{s+\tau_{0,t}(\Delta_s)\in s+M'(\Delta_s)\}=\Delta_s^{-1}(\{\tau_{0,t}\in M'\})$ [which fact is true for $s\in \mathbb{R}$, not just $s\in(0,\infty)$];
 \item therefore, using the basic properties of $\WW$,
$$ \WW(\tau_{0,s+t}\in M')\geq \WW(\tau_{0,s}\in M',\tau_{s,s+t}\in M')=\WW(\tau_{0,s}\in M')\WW(\tau_{s,s+t}\in M')
=\WW(\tau_{0,s}\in M')\WW(\tau_{0,t}\in M').$$
 \end{itemize}
In  words, the map $f:=((0,\infty)\ni t\mapsto \WW(\tau_{0,t}\in M'))$ is supermultiplicative. Let $D$ be  a  countable subset of $(0,\infty)$ with $0$ as an accumulation point on which $f$ is $\{0,1\}$-valued; it exists by assumption $(\dagger)$.  If (*) $f(t)=0$ for some $t\in D$, then $f(t2^{-n})=0$ for all $n\in \mathbb{N}_0$ by supermultiplicativity of $f$. But 
 \begin{equation*}M=\{\tau_{tm2^{-n},t(m+1)2^{-n}}:(m,n)\in \mathbb{Z}\times \mathbb{N}_0\}\text{ a.s.-$\WW$}.\end{equation*} On using also the third bullet point above, (*) thus leads to $M'$ being empty. On the other hand, if (**) $f(t)=1$ for all $t\in D$, then since also  \begin{equation*}M=\{\tau_{p,p+t}:(p,t)\in \mathbb{Q}\times D\}\text{ a.s.-$\WW$},\end{equation*}
we get by the third bullet point again, that (**) leads to $M'=M$ a.s.-$\WW$. We conclude that $M$ is indeed minimal.
\end{proof}

\begin{corollary}\label{corollary:minimal}
Suppose the dense stationary local random countable set $M$ admits an honest indexation $\tau$. Then $M$ is minimal.
\end{corollary}
\begin{proof}
Let $M'$ be a stationary local random countable set contained in $M$ a.s.-$\WW$. By Theorem~\ref{theorem:splitting} (with an arbitrary $\lambda$, $\lambda=1$ say) and in the notation thereof $ \{\tau_{\mathsf{e}}\in M'\}$ is $\tilde\WW$-trivial; by Tonelli and Theorem~\ref{theorem:minimal} it is enough.
\end{proof}

Immediately it begs

\begin{question}
If a stationary local random countable set $M$ is minimal, must it admit an honest indexation?
\end{question}

\begin{corollary}\label{corollary:minimal-Md}
For $d\in (0,2)$  the set $M^{(d)}$ is minimal. In particular the local minima (maxima) are minimal.
\end{corollary}
\begin{proof}
Apply Corollary~\ref{corollary:minimal} recalling Example~\ref{example:Md-honest}.  For $d=1$ we get the local minima; for the local maxima one slides along the map which sends $B$ to $-B$.
\end{proof}

\begin{corollary}
The local extrema do not admit an honest indexation.
\end{corollary}
\begin{proof}
Indeed they are not minimal so the contrapositive of Corollary~\ref{corollary:minimal} applies.
\end{proof}

Left open, but natural is 

\begin{question}
Does every local stationary random countable set decompose into a countable union of such sets which are minimal?
\end{question}
By a ``straightforward'' Zorn lemma type argument (cf. the proof of Proposition~\ref{propo:every-enumeration-is-stabilising} below) it is equivalent to asking whether or not every local stationary random countable set that is not empty (hence is dense) admits a minimal subset of its kind. 

\section{Noises out of random sets}\label{section:noises-out-of-random-sets}
In this section we connect the stationary local random countable sets that we have studied above  to some nonclassical extensions of the Wiener noise. All Hilbert spaces considered below are complex.

\subsection{Construction}\label{subsection:construction}
Preparation. Let $M$ be a dense perfectly stationary local random countable set, countable with certainty (recall Proposition~\ref{proposition:version-perfect-stationary} concerning the existence of such a version for a given local stationary $M$). Take a shift-invariant  $\WW$-almost certain set $\Omega_1$ on which $M$ is dense (for instance, $\Omega_1=\{M\text{ is dense}\}$ will do /by perfect stationarity/).
Fix also a measurable enumeration $S$ for $M$, which we (may and do) assume is injective, $\mathbb{R}^\mathbb{N}$-valued and satisfies $M=[S]$, all of this on $\Omega_1$. 

We shall work, fore the most part, on $\Omega_1$, letting $\mathbb{W}$ be the restriction of Wiener measure $\WW$ to $\Omega_1$. The spruced  up properties of $M$ and $S$ on $\Omega_1$, as above, will save us from a number of ``a.s., mod-$0$, up to a negligible set'' qualifiers in what follows. Though,  there seems to be a certain  ``partial law of preservation of nuissance'' in mathematics, and here too we are not  able to escape it. For instance, $\Omega_0$ is canonically identified with $ \Omega_0\vert_{(-\infty,0]} \times \Omega_0\vert_{[0,\infty)}$ via the bijection $\omega\mapsto (\omega\vert_{(-\infty,0]},\omega\vert_{[0,\infty)})$ which is measure-preserving between $\WW$ and the product of the Wiener measures on the two factors of  $\Omega_0\vert_{(-\infty,0]} \times \Omega_0\vert_{[0,\infty)}$ (i.e. it pushes forward the first onto the second), and there is no a priori reason why $\Omega_1$ should also have such a canonical product structure. (Although, it is true ``up to  negligible sets'': the same map, restricted to $\Omega_1$, is an injective map that pushes $\mathbb{W}$ forward onto the product of the Wiener measures on the two factors of $\Omega_1\vert_{(-\infty,0]} \times \Omega_1\vert_{[0,\infty)}$. Hence it is actually a mod-$0$ isomorphism between the two probabilities (since they are standard, see \cite[p.~22, Section~2.5, Theorem on isomorphisms]{rohlin}).) For this reason the reader will notice that sometimes we  prefer to use $\Omega_0$ where $\Omega_1$ would (prima facie) be more natural, e.g. in \eqref{eq:decomposition}-\eqref{equation:conditional-spectral-measure}. We have anyway strived to keep the  overhang of ``negligibility'' at a minimum.

The objects we shall introduce below will depend on $\Omega_1$, some of them on $S$ and basically all of them on $M$. To keep it manageable we will make explicit only  $M$ in the notation as it  seems enough to preclude confusion (when later on we shall deal with more than one $M$ at a time). 
\begin{remark}\label{remark:same-omega1}
In fact, the same (a universal) $\Omega_1$ can be chosen for all stationary local random countable sets on passing to suitable versions thereof (these choices/changes do not really matter -- see the forthcoming Remark~\ref{rmk-iso-noise}). Let indeed $M'$ be another perfectly stationary local random countable set, countable with certainty. Attach to it $\Omega_1'$ as $\Omega_1$ was attached to some fixed chosen $M$ (say, the local minima). Then change $M'$ to $M$ on $\Omega_1\backslash \Omega_1'$. We get a version of $M'$ for which we may take $\Omega_1'=\Omega_1$.  (However, we \emph{cannot ever} take $\Omega_1=\Omega_0$ -- for instance it contains the constant function zero, and there can be no countable dense set $m$ attached to this function, which would satisfy $m=u+m$ for all $u\in \mathbb{R}$.)
\end{remark} 

With these preliminaries out of the way, put $\Omega^M:=\{(\omega_1,\eta):\omega_1\in \Omega_1,\eta\in \{-1,1\}^{M(\omega_1)}\}$, which will be the sample space for Brownian paths enhanced by independent equiprobable random signs attached to the points of the random subset $M$. We stress again that $\Omega^M$ depends on both $\Omega_1$ and $M$, but only reference to $M$ has been made in the notation. So as to formalize the corresponding probability measure, define the bijective map $\Theta^M:\Omega^M\to \Omega_1\times \{-1,1\}^\mathbb{N}$ by putting \begin{equation*}\Theta^M(\omega_1,\eta):=(\omega_1,\eta\circ S(\omega_1))=(\omega_1,(\eta(S_k(\omega_1)))_{k\in \mathbb{N}}),\quad (\omega_1,\eta)\in \Omega;\end{equation*} besides $M$ this map depends also on $\Omega_1$ and, more importantly, $S$, but still only $M$ has been noted explicitly. On $\Omega_1\times \{-1,1\}^\mathbb{N}$ we consider the probability $\mathbb{Q}:=\mathbb{W}\times (\frac{1}{2}\delta_{-1}+\frac{1}{2}\delta_{1})^{\times \mathbb{N}}$ (product of Wiener measure on $\Omega_1$ and of the infinite product of equiprobable random signs; no completion \emph{of the product} is needed or made) and pull  it back, including the measurable structure, via $\Theta^M$ to a probability on $\Omega^M$, which we shall denote $\PP^M$ and its domain by $\BB^M$. For two different choices $S^1$ and $S^2$ of $S$ the transition map (in the obvious notation) $\zeta:=\Theta^{M2}\circ(\Theta^{M1})^{-1}:\Omega_1\times \{-1,1\}^\mathbb{N}\to \Omega_1\times \{-1,1\}^\mathbb{N}$ is given by \begin{equation*}\zeta(\omega_1,p)=(\omega_1,p\circ S^1(\omega_1)^{-1}\circ S^2(\omega_1)),\quad (\omega_1,p)\in \Omega_1\times \{-1,1\}^\mathbb{N},\end{equation*} and is bimeasurable, bijective and measure-preserving for $\mathbb{Q}$, the latter because the actions of the permutations of $\mathbb{N}$ are measure-preserving for $(\frac{1}{2}\delta_{-1}+\frac{1}{2}\delta_{1})^{\times \mathbb{N}}$. Therefore the probability $\PP^M$, including its sigma-field $\BB^M$, actually does not depend on the choice of $S$, even though $\Theta^M$ does. At this point we complete $\PP^M$ and its $\sigma$-field $\BB^M$, but keep on using the same symbols for the completions by an abuse of notation.


Let us denote next by $(W,P)$ the canonical projections on $\Omega_1\times \{-1,1\}^\mathbb{N}$, which is just to say that the pair $(W,P)$ is the identity on this space. It is plain that the law of $W\circ \Theta^M$  under $\PP^M$ is $\mathbb{W}$. 

An important structural property (of the $\L2$ space) of $\PP^M$ is revealed by
\begin{proposition}\label{lemma:general-f}
For $f\in\L2(\PP^M)$ there exist a.s.-$\WW$ unique maps $f_K\in \L2(\WW)$, $K\in (2^\mathbb{N})_{\mathrm{fin}}$, satisfying $\sum_{K\in (2^\mathbb{N})_{\mathrm{fin}}}\WW[\vert f_K\vert^2]<\infty$ and
\begin{equation}\label{eq:decomposition}f((\omega_1,\eta))=\sum_{K\in (2^\mathbb{N})_{\mathrm{fin}}} f_K(\omega_1)\prod_{k\in K}\eta(S_k(\omega_1))\text{ for $\PP^M$-a.e. }(\omega_1,\eta),\end{equation} i.e. $f=\sum_{K\in (2^\mathbb{N})_{\mathrm{fin}}}f_K(W\circ \Theta^M)\prod_{k\in K}P_k\circ \Theta^M$ a.s.-$\PP^M$; moreover, \eqref{eq:decomposition} is an $\L2(\PP^M)$-orthogonal sum decomposition.
\end{proposition}
The $(f_K)_{K\in (2^\mathbb{N})_{\mathrm{fin}}}$ of \eqref{eq:decomposition} depend on $M$, $\Omega_1$ and $S$, but we do not reference these, not even $M$, since we will only need it  for the given $(M;\Omega_1,S)$.
\begin{proof}
The map $\Theta^M$ induces the unitary isomorphism $\L2(\mathbb{Q})\ni g\mapsto g\circ \Theta^M\in\L2(\mathbb{P}^M)$.
We have also the natural unitary isomorphism between $\L2(\mathbb{Q})=\L2(\mathbb{W}\times (\frac{1}{2}\delta_{-1}+\frac{1}{2}\delta_{1})^{\times \mathbb{N}})$ and $\L2(\mathbb{W})\otimes \L2((\frac{1}{2}\delta_{-1}+\frac{1}{2}\delta_{1})^{\times \mathbb{N}})$. 
Finally, the orthonormal basis $\left(\prod_{k\in K}P_k\right)_{K\in(2^\mathbb{N})_{\mathrm{fin}}}$ of  $ \L2((\frac{1}{2}\delta_{-1}+\frac{1}{2}\delta_{1})^{\times \mathbb{N}})$ induces in the known manner a unitary isomorphism of $\L2(\mathbb{W})\otimes \L2((\frac{1}{2}\delta_{-1}+\frac{1}{2}\delta_{1})^{\times \mathbb{N}})$ onto $\oplus_{K\in (2^\mathbb{N})_{\mathrm{fin}}}\L2(\mathbb{W})$. Decomposing a general element of  $\oplus_{K\in (2^\mathbb{N})_{\mathrm{fin}}}\L2(\mathbb{W})$ according to the orthogonal sum and sliding back and forth along these isomorphisms gives the posited representation up to the trivial canonical identification of $\L2(\mathbb{W})$ with $\L2(\WW)$.
\end{proof}

The probability $\PP^M$ having been specified and some of its basic properties unearthed, we now define a one-dimensional  factorization of sigma-fields associated to $M$ as follows. For extended-real $s<t$ we set $N_{s,t}^M$ to be the $\PP^M$-complete (referred to as the property of ``completeness'' below) sigma-field generated by the increments of the Brownian motion $W\circ \Theta^M$ on the interval $(s,t)$   \emph{and} by the random variables \begin{equation*}R_k:=\left(\Omega^M\ni(\omega_1,\eta)\mapsto \eta(S_{s,t}(k)(\omega_1))\right),\quad k\in \mathbb{N},\end{equation*} where $S_{s,t}$ is an $\FF_{s,t}$-measurable  enumeration of $M\cap (s,t)$ that is injective, $(s,t)^\mathbb{N}$-valued and satisfies $M\cap (s,t)=[S_{s,t}]$, all of this on $\Omega_1$. Such an enumeration we may ask for by locality of $M$ and by the density of $M$ on $\Omega_1$. Notice that  $N_{s,t}^M$ does not depend on the choice of the enumeration $S_{s,t}$, just because $S_{s,t}$ is $\FF_{s,t}$-measurable and $N_{s,t}^M$ includes the information generated by the increments of $W\circ\Theta^M$ on $(s,t)$.  Informally we think of $N_{s,t}^M$ as being generated by the movement of $B$ on $(s,t)$ and by the random signs attached to the points of $M$ which belong to $(s,t)$ together with the trivial sets. 

To see the so-called ``factorizability'' property of $N^M$ under $\PP^M$, namely that the $\sigma$-fields $N_{t_0,t_1}^M,\ldots,N_{t_{n-1},t_n}^M$ are $\PP^M$-independent for extended-real $-\infty=t_0<\cdots <t_n=\infty$, $n\in \mathbb{N}$, and together generate $N_{-\infty,\infty}^M=$ the whole of the $\sigma$-field of $\PP^M$, take $S$ such that $S(n(m-1)+l)=S_{t_{l-1},t_l}(m)$ for $m\in \mathbb{N}$ and $l\in [n]$ a.s.-$\mathsf{W}$ (we can do it because $\WW(\{t_1,\ldots,t_{n-1}\}\cap M\ne \emptyset)=0$), then use the $\mathbb{Q}$-independence of the  $\sigma$-fields generated by the increments of $W$ on $(t_{l-1},t_l)$ and by the random signs $(P(n(m-1)+l))_{m\in \mathbb{N}}$ as $l$ ranges over $[n]$.

 Thus we have unambiguously introduced a continuous factorization of $\sigma$-fields (a.k.a. a continuous product of probability spaces) \cite[Definition~3c1]{tsirelson-nonclassical} $N^M=(N_{s,t}^M)_{(s,t)\in [-\infty,\infty]^2,s<t}$ associated to the random set $M$ on $\Omega_1$. Note that a slightly different definition of this notion appears in \cite[Definition~3.16]{picard2004lectures}, namely one asks in addition for ``upward continuity'' in the sense that $N^M_{s,t}$ is generated by $\cup_{\epsilon\in (0,\infty) }N^M_{s+\epsilon,t-\epsilon}$ for all extended-real $s<t$, where we interpret $-\infty+\epsilon:=-1/\epsilon$ and $\infty-\epsilon:=1/\epsilon$. This property will actually be automatic once we have established below that $N^M$ corresponds to a one-dimensional noise  \cite[Proposition~3d3]{tsirelson-nonclassical}. 

The preceding construction is basically that of \cite[p.~566]{tsirelson-arveson}, where it is done for $M=\{\text{local minima of $B$}\}$. 

The family $N^M$ is enhanced to a one-dimensional noise (a.k.a. a homogeneous continuous product of probability spaces) \cite[Definition~3d1]{tsirelson-nonclassical} by the introduction of the  group $T^M=(T_h^M)_{h\in \mathbb{R}}$ of $\BB^M$-bimeasurable bijections of $\Omega^M$, measure-preserving for $\PP^M$, defined as follows: 
\begin{equation*}T_h^M((\omega_1,\eta))=(\Delta_h(\omega_1),\eta(\cdot+h)),\quad (\omega_1,\eta)\in \Omega^M,\, h\in \mathbb{R}.\end{equation*} Each $T_h^M$ induces indeed an automorphism of the probability $\PP^M$ that sends  $N_{s,t}^M$ to $(T_h^M)^{-1}(N_{s,t}^M)=N_{s+h,t+h}^M$, this for all extended-real $s<t$ and all $h\in \mathbb{R}$ (the property of ``homogeneity'').

In partial summary, given a triplet $(M;\Omega_1,S)$, we have
$$\text{ the canonical process } (W,P)\text{ on $\Omega_1\times \{-1,1\}^\mathbb{N}$ under }\QQ={\Theta^M}_\star \PP^M\text{ and  the noise }(N^M;\PP^M,T^M).$$
\begin{remark}\label{rmk-iso-noise}
If $\Omega^{M'}$ etc. are constructed from $(M';\Omega_1',S')$ in lieu of $(M;\Omega_1,S)$ and if $M'$ is a version of $M$, then $\Omega^M\cap\Omega^{M'}=\{(\omega_1,\eta):\omega_1\in \{M=M'\}\cap \Omega_1\cap \Omega_1',\eta\in \{-1,1\}^{M(\omega_1)}\}$,  $\PP^{M}(\Omega^M\cap\Omega^{M'})=1=\PP^{M'}(\Omega^M\cap\Omega^{M'})$; furthermore, $N^M=N^{M'}$ and $T^M=T^{M'}$ on $\Omega^M\cap \Omega^{M'}$.
\end{remark}

 \subsection{A continuity condition}
Actually, the system $(N^M;\PP^M,T^M)$ is not as yet evidently fully in line with  \cite[Definition~3d1]{tsirelson-nonclassical} for we are still missing  measurability (equivalently, continuity) of the group of automorphisms, viewed as a map from $\mathbb{R}$ into the Polish group of all automorphisms of $\PP^M$ (for the strong topology of the associated unitary operators on $\L2(\PP^M)$). Before immersing ourselves  into establishing this technical fact it seems appropriate here to quote Tsirelson (yet again): ``Unfortunately, the latter assumption (continuity of the group action) is missing in my former publications, which opens the door for pathologies.'' \cite[p.~42]{picard2004lectures}. Let us preclude these pathologies in the case of the noises at hand!

\begin{proposition}\label{proposition:ct-partial-result}
The following are equivalent. 
\begin{enumerate}[(i)]
\item\label{continuity:i} For each $f\in \L2(\PP^M)$, the map $\mathbb{R}\ni h\mapsto f(T_h^M)\in \L2(\PP^M)$ is continuous.
\item\label{continuity:ii} $\lim_{h\to 0}\WW(S_k=h+S_k(\Delta_h))=1$ for all $k\in \mathbb{N}$.
\end{enumerate}
\end{proposition}
\begin{proof}
Assume \ref{continuity:ii}. Using the fact that, by functional monotone class,  products of the form $\prod_{t\in T}f_t(W_{t}\circ \Theta^M)$, with $f_t:\mathbb{R}\to \mathbb{C}$ continuous bounded  for each $t\in T$,  $T$ a finite subset of $\mathbb{R}$, are total in $\L2(\PP^M\vert_{\overline{\sigma}(W\circ\Theta^M)})$, we deduce easily by approximation the requisite continuity of \ref{continuity:i} in case $f\in \L2(\PP^M\vert_{\overline{\sigma}(W\circ\Theta^M)})$ ($\overline{\sigma}$ signifies completion w.r.t. $\PP^M$). (It is just the continuity of the L\'evy shifts for the Wiener noise and it does not use \ref{continuity:ii}.) From Proposition~\ref{lemma:general-f} we reduce further at once to the case when $f$ has $f_K=0$ for all but one $K\in (2^\mathbb{N})_{\mathrm{fin}}$ and, moreover, by telescoping, to the case when $f_K=0$ except for a single singleton $K=\{k\}$, $k\in \mathbb{N}$, for which $f_{\{k\}}=1$. Then $f=P_k\circ \Theta^M$ and we compute
\begin{align*}
\PP^M\left[\vert f-f\circ T_h^M\vert^2 \right ]&=\QQ\left[( P_k-P_k\circ \Theta^M\circ T_h^M\circ (\Theta^M)^{-1})^2\right]=\QQ\left[\left(P_k-\sum_{l\in \mathbb{N}}P_l\mathbbm{1}_{\{S_l(W)=h+S_k(\Delta_hW)\}}\right)^2\right]\\
&=\WW(S_k\ne h+S_k(\Delta_h))+\sum_{l\in \mathbb{N}\backslash \{k\}}\WW(S_l=h+S_k(\Delta_hW))=2\WW(S_k\ne h+S_k(\Delta_h)).
\end{align*}
This computation being valid whether or not \ref{continuity:ii} holds true we see not only that \ref{continuity:ii} implies \ref{continuity:i} but also the reverse implication.
\end{proof}
%

\begin{definition}
A random variable under $\WW$ with values in $\mathbb{R}^\dagger$ is called shift-stabilising when $\lim_{h\to 0}\WW(S=h+S(\Delta_h))=1$. A measurable enumeration of a random countable set is said to be shift-stabilising if each member thereof is so.
\end{definition}

For the next two results we suspend temporarily the overarching setting that we have built in this section hitherto (namely, in Subsection~\ref{subsection:construction}). 

\begin{lemma}\label{lemma:non-empty-stabilising}
Let $M$ be a stationary random countable set that is not empty. Then there is a stationary random countable set $M'$ that is contained in $M$ a.s.-$\WW$, that is not empty and that admits a shift-stabilising enumeration.
\end{lemma}
\begin{proof}
By Proposition~\ref{proposition:version-perfect-stationary} we may and do assume $M$ is perfectly stationary, countable with certainty and belongs to $\mathcal{B}_{\Omega_0}\otimes \mathcal{B}_\mathbb{R}$. Then set $A:=\{0\in M\}\in \mathcal{B}_{\Omega_0}$ and check  that $M=M^A$ (the proof is verbatim the same as in Theorem~\ref{theorem:construction-through-zero-time-section}). 

Let next $S$ be a measurable perfect enumeration of $M$ and set $\mathbb{A}:=\sum_{k\in \mathbb{N}}2^{-k}(\Delta_{S_k})_\star \WW\vert_{\{S_k\in \mathbb{R}\}}$, which is a subprobability measure on $\mathcal{B}_{\Omega_0}$ carried by $A$. Using  \cite[Theorem~1.1]{billingsley1968convergence} (which is to say, using  the fact that on metric spaces measurable sets can be approximated from within by closed sets relative to any finite measure) we see that there exists a nondecreasing sequence of closed sets $(C_n)_{n\in \mathbb{N}}$ of $\Omega_0$ such that $A\supset C:=\cup_{n\in \mathbb{N}}C_n$ has full $\mathbb{A}$-measure. Then $M^C\in \mathcal{B}_{\Omega_0}\otimes \mathcal{B}_\mathbb{R}$ (again the proof is verbatim the same as in Theorem~\ref{theorem:construction-through-zero-time-section}) and  $M^C\subset M^A=M$. Moreover, $M^C=M$ a.s.-$\WW$. For suppose per absurdum that it were not so. Then, for some $n\in \mathbb{N}$, we would have that $S_n\notin M^C$, i.e. $\Delta_{S_n}\notin C$, with positive $\WW$-probability on $\{S_n\in\mathbb{R}\}$, contradicting the fact that $C$ has full $\mathbb{A}$-measure. Besides, for each $n\in \mathbb{N}$ and $\omega\in \Omega_0$,  $M^{C_n}(\omega)$ is a closed subset of $\mathbb{R}$ since it is the preimage of $C_n$ under the continuous map $\mathbb{R}\ni t\mapsto \Delta_t(\omega)\in \Omega_0$ (the continuity comes from the fact that the continuous path $\omega$ is uniformly continuous on compacta).

For $p\in \mathbb{R}$ and $n\in \mathbb{N}$ define 
\begin{equation*}
S_p^n:=\inf (M^{C_n}\cap [p,\infty)) \qquad (\inf \emptyset:=\dagger).
\end{equation*}
Because  $M^{C_n}$ is closed we have $S_p^n\in M^{C_n}$ on $\{S_p^n\in \mathbb{R}\}$. Since by stationarity of $M$, $\WW(p\in M)=0$, and again because  $M^{C_n}$ is closed we have that a.s.-$\WW$ there is $\epsilon>0$ such that $M^{C_n}\cap (p-\epsilon,p+\epsilon)=\emptyset$. Hence $\lim_{h\downarrow 0}\WW(S_p^n=h+S_p^n(\Delta_h))=1$, indeed $\WW$-a.s. $S_p^n=h+S_p^n(\Delta_h)$ eventually as $h\downarrow 0$. Let us show that the set $M_n:=\{S_p^n:p\in \mathbb{Q}\}\cap \mathbb{R}$ is stationary. Pick an arbitrary $h\in \mathbb{R}$. It will suffice to establish that $h+M_n(\Delta_h)\subset M_n$ a.s.-$\WW$. But for $p\in \mathbb{Q}$, $S_p^n(\Delta_h)>p$ a.s.-$\WW$ on $\{S_p^n(\Delta_h)\in \mathbb{R}\}$, hence $\WW$-a.s.  $h+S_p^n(\Delta_h)=S_q^n$ for some sufficiently small rational $q\in [p+h,\infty)$. 

Finally, note that since $M=\cup_{n\in \mathbb{N}}M^{C^n}$ a.s.-$\WW$ and since $M$ is not empty, then, for a large enough $n\in \mathbb{N}$, $M^{C_n}$ is not a.s.-$\WW$  empty and therefore neither is $M_n$ empty. Take $M':=M_n$.
\end{proof}

\begin{proposition}\label{propo:every-enumeration-is-stabilising}
A measurable  enumeration of a   stationary random countable set $M$ is shift-stabilising. 
\end{proposition}
We may recall from Remark~\ref{rmk:no-stationary} that, by contrast, no measurable enumeration of a random countable set $M$ that is not empty can be stationary.
\begin{proof}
First we show that $M$ admits a shift-stabilising measurable enumeration. 

Let $\mathfrak{M}$ be the collection of  stationary random countable sets that are contained in $M$ a.s.-$\WW$, that are not empty  and that admit an enumeration that is shift-stabilising, these sets being quotiented out by $\WW$-a.s. equality. Let $\mathfrak{F}\subset 2^\mathfrak{M}$ be the collection of those subsets $S$ of $\mathfrak{M}$ whose elements are a.s.-$\WW$ pairwise disjoint (so, informally, the collection of (possibly inexhaustive) partitions of $M$, each member of which admits a shift-stabilising enumeration). Partially order $\mathfrak{F}$ by inclusion. $\emptyset\in \mathfrak{F}$, therefore $\mathfrak{F}$ is not empty. Each chain in $\mathfrak{F}$ has an upper bound, namely its union. By Zorn's lemma $\mathfrak{F}$ admits a maximal element, say $\mathsf{F}$. The set $\mathsf{F}$ is countable, since, for a(ny) given measurable enumeration $S$ of $M$ and then for each $n\in \mathbb{N}$, $\WW(S_n\in D)>0$ for at most countably many $D\in \mathsf{F}$ (just because the events $\{S_n\in D\}$, $D\in \mathsf{F}$, are pairwise disjoint a.s.-$\WW$), while for each $D\in \mathsf{F}$, $\WW(S_n\in D)>0$ for some $n\in \mathbb{N}$. The union of the members of $\mathsf{F}$ (defined up to a.s.-$\WW$ equality), call it $G$, is a stationary random countable set contained in $M$ a.s.-$\WW$ that admits a shift-stabilising measurable enumeration (by a diagonalization of the shift-stabilising enumerations of the individual members of $G$). Finally, $G$ must be a.s.-$\WW$ equal to $M$, since otherwise $M\backslash G$ would be a stationary random countable set, not empty, and Lemma~\ref{lemma:non-empty-stabilising} would yield a contradiction with the maximality of $\mathsf{F}$.

Now let $  S$ be a shift-stabilising measurable enumeration for $M$ and let $\tilde S$ be an arbitrary measurable enumeration of $M$. For each $n\in \mathbb{N}$ we have $\WW$-a.s.
\begin{align*}
\{\tilde S_n\ne h+\tilde S_n(\Delta_h)\}&\subset \{\tilde S_n=\dagger,\tilde S_n(\Delta_h)\ne \dagger\}\cup \Big(\cup_{k\in \mathbb{N}}\{\tilde  S_n\ne \dagger, \tilde S_n\ne S_1,\ldots, \tilde S_n\ne S_{k-1},\tilde S_n= S_k\}\\
&\qquad \qquad \cap [\{\tilde S_n= S_k,\tilde S_n(\Delta_h)\ne S_k(\Delta_h)\}\cup \{S_k\ne h+ S_k(\Delta_h)\} ]\Big).
\end{align*}
In the proof of Proposition~\ref{proposition:ct-partial-result} we established \ref{continuity:i} thereof for $f\in \L2(\PP^M\vert_{\overline{\sigma}(W\circ\Theta^M)})$. This means that  $\lim_{h\downarrow 0}\WW(\tilde S_n=\dagger,\tilde S_n(\Delta_h)\ne \dagger)=0$ and $\lim_{h\downarrow 0}\WW(\tilde S_n= S_k,\tilde S_n(\Delta_h)\ne S_k(\Delta_h))=0$ for all $k\in \mathbb{N}$. Combined with the shift-stabilisation property of $S$ we conclude at once via countable subaddivity and dominated convergence.
\end{proof}

Return to the standing setting of this section as delineated in Subsection~\ref{subsection:construction}.
\begin{corollary}
Item~\ref{continuity:i} of Proposition~\ref{proposition:ct-partial-result} holds true.\qed 
\end{corollary}
Thus the ``continuity of the group action'' has been established. Combined with $\PP^M$ being standard, as well as the completeness, factorizability and homogeneity properties noted already in the previous subsection, the system $(N^M;\PP^M,T^M)$ is indeed a ``fully-fledged'' noise. Another is $(\FF;\WW,\Delta)$  -- the  Wiener noise, this qualification, used hitherto only informally, being now understood in the formal sense of  the properties just listed.

\subsection{The stable part and the first superchaos}
We denote by  $H_{\mathrm{stb}}^M$ and 
$H_{\mathrm{sens},1}^M$    the stable part \cite[p.~67]{picard2004lectures} and the first superchaos \cite[p.~71]{picard2004lectures} of the noise $(N^M;\PP^M,T^M)$, respectively. They are closed linear subspaces of $\L2(\PP^M)$.  For extended-real $s<t$ set also
\begin{equation*}P_{s,t}(k):=R_k\circ (\Theta^M)^{-1}=P(S(W)^{-1}(S_{s,t}(W)_k)),\quad k\in \mathbb{N}.
\end{equation*} $P_{s,t}:=(P_{s,t}(k))_{k\in \mathbb{N}}$ is the sequence of random signs  that is ``associated'' with $(s,t)$  on $\Omega_1\times\{-1,1\}^\mathbb{N}$; it depends on $S$, $S_{s,t}$, $M$ and $\Omega_1$, none of which we reference, since we will only need it for the given $S$, $S_{s,t}$, $M$ and $\Omega_1$.

\begin{proposition}\label{lemma:identifications-stable-super}
We have the identifications $H_{\mathrm{stb}}^M=\{f\in \L2(\PP^M):f_K=0\text{ for all }K\in (2^\mathbb{N})_{\mathrm{fin}}\backslash\{\emptyset\}\}=\L2(\PP^M\vert_{\overline{\sigma}(W\circ\Theta^M)})$ (as usual, $\overline{\sigma}$ indicates completion w.r.t. the relevant measure, $\PP^M$ in this case) and $H_{\mathrm{sens},1}^M=\{f\in \L2(\PP^M):f_K=0\text{ for all }K\in (2^\mathbb{N})_{\mathrm{fin}}\text{ of size }\ne 1\}$. 
\end{proposition}
A noise is called classical when the stable part is the whole of the $\L2$ space. Thus $(N^M;\PP^M,T^M)$ is not classical. At the other side of the spectrum it is called black when the stable part is $\{0\}$ but the sensitive part is not $\{0\}$. The noise $(N^M;\PP^M,T^m)$ is not black either.  Since $\sigma(H^M_{\mathrm{stb}},H^M_{\mathrm{sens},1})=N^M_{-\infty,\infty}$, $(N^M;\PP^M,T^m)$ has only superchaoses of finite order (no non-zero ``super-superchaoses'') \cite[p.~71, penultimate paragraph]{picard2004lectures}. 

\begin{proof}
Suppose we have established that $H_{\mathrm{stb}}^M=\L2(\PP^M\vert_{\overline{\sigma}(W\circ\Theta^M)})$. Knowing this, the determination of the first superchaos is relatively straightforward, we just project the sensitive subspace $H_{\mathrm{sens}}^M=\L2(\PP^M)\ominus H_{\mathrm{stb}}^M=\L2(\PP^M)\ominus \L2(\PP^M\vert_{\overline{\sigma}(W\circ\Theta^M)})$ onto the first superchaos using \cite[Theorem 6.8]{picard2004lectures}. Let then $K\in (2^\mathbb{N})_{\mathrm{fin}}\backslash \{\emptyset\}$ and $f\in \L2(\WW)$ be bounded. We show that if $\vert K\vert=1$ then the projection onto the first superchaos leaves  $F:=f(W\circ \Theta^M)\prod_{k\in K}P_k\circ \Theta^M$ invariant, while if $\vert K\vert>1$ then it sends this vector to zero; by Proposition~\ref{lemma:general-f} and the fact that bounded  maps are dense in $\L2(\WW)$ it will be enough. 

As preparation for the main computation we first observe, for extended-real $s<t$, as follows. On the one hand, on $\cap_{k\in K}\{S(W)_k\in (s,t)\}\in \overline{\sigma}(W)$, 
\begin{equation*}\prod_{k\in K}P_k=\sum_{l\in \mathbb{N}^K}\mathbbm{1}_{\cap_{k\in K} \{S_{s,t}(W)_{l_k}=S(W)_k\}}\prod_{k\in K}P_k=\sum_{l\in \mathbb{N}^K}\mathbbm{1}_{\cap_{k\in K} \{S_{s,t}(W)_{l_k}=S(W)_k\}}\prod_{k\in K}P_{s,t}(l_k)\in \overline{\sigma}(W,P_{s,t});\end{equation*}
therefore 
\begin{equation*}
\mathbb{Q}\left[\prod_{k\in K}P_k\Big\vert W,P_{s,t}\right]=\prod_{k\in K}P_k\text{ a.s.-$\QQ$ on }\cap_{k\in K}\{S(W)_k\in (s,t)\}.
\end{equation*} On the other hand, for disjoint  $L\subset K$ and $R\subset K$, $L\cup R\ne \emptyset$,  $L$ (resp. $R$) empty if $s=-\infty$ (resp. $t=\infty$), we see similarly that, on   $(\cap_{k\in L}\{S(W)_k\in (-\infty,s)\})\cap (\cap_{k\in K\backslash (L\cup R)}\{S(W)_k\in (s,t)\})\cap (\cap_{k\in R}\{S(W)_k\in (t,\infty)\})$, 
\begin{align*}
\prod_{k\in K}P_k=&\sum_{l\in \mathbb{N}^K}\mathbbm{1}_{(\cap_{k\in L} \{S_{-\infty,s}(W)_{l_k}=S(W)_k\})\cap (\cap_{k\in K\backslash (L\cup R)} \{S_{s,t}(W)_{l_k}=S(W)_k\})\cap (\cap_{k\in R} \{S_{t,\infty}(W)_{l_k}=S(W)_k\})}\\
&\left(\prod_{k\in L}P_{-\infty,s}(l_k)\prod_{k\in K}P_{s,t}(l_k)\prod_{k\in R}P_{t,\infty}(l_k)\right).
\end{align*}
Further, \emph{given} $W$ we have as follows: the sequences of random signs $P_{-\infty,s}$, $P_{t,\infty}$ (being empty respectively as $s=-\infty$, $t=\infty$) are $\QQ$-independent, jointly  $\QQ$-independent of  $P_{s,t}$ and have entries of zero mean. Therefore, since also $\QQ(S_k(W)\in \{s,t\}\text{ for some }k\in K)=0$,
\begin{equation*}\mathbb{Q}\left[\cap_{k\in K}P_k\Big\vert W,P_{s,t}\right]=0\text{ a.s.-$\QQ$ off $\cap_{k\in K}\{S(W)_k\in (s,t)\}$}.\end{equation*}
Altogether we conclude that 
$$\mathbb{Q}\left[\cap_{k\in K}P_k\Big\vert W,P_{s,t}\right]=\mathbbm{1}_{\cap_{k\in K}\{S_k(W)\in (s,t)\}}\prod_{k\in K}P_k \text{ a.s.-$\QQ$}.$$

With this in hand we return to projecting $F$ onto the first superchaos. To this end let $-\infty=t_0^n<\cdots<t_{k_n}^n=\infty$, $k_n\in \mathbb{N}$, with $A_n:=\{t_l^n:l\in [k_n-1]\}\uparrow$ to a dense subset of $\mathbb{R}$ as $n\in\mathbb{N}$ increases to $\infty$, and we compute, for each $n\in \mathbb{N}$,
\begin{align*}
\sum_{i\in [k_n]}\PP^M\left[F\Big\vert N_{t^n_{i-1},t^n_i}^M \lor \sigma(W\circ \Theta^M)\right]&=\sum_{i\in [k_n]}\PP^M\left[f(W\circ\Theta^M)\prod_{k\in K}(P_k\circ \Theta^M)\Big\vert N_{t^n_{i-1},t^n_i}^M \lor \sigma(W\circ \Theta^M)\right]\\
&=f(W\circ \Theta^M)\left(\sum_{i\in [k_n]}\QQ\left[\prod_{k\in K}P_k\Big\vert P_{t^n_{i-1},t^n_i},W\right]\right)\circ\Theta^M\\
&=f(W\circ \Theta^M)\sum_{i\in [k_n]}\mathbbm{1}_{\cap_{k\in K}\{S_k(W)\circ \Theta^M\in (t^n_{i-1},t^n_i)\}}\prod_{k\in K}(P_k\circ\Theta^M)\\
&=F\sum_{i\in [k_n]}\mathbbm{1}_{\cap_{k\in K}\{S_k(W)\circ \Theta^M\in (t^n_{i-1},t^n_i)\}}
\end{align*}
a.s.-$\PP^M$. Now if $\vert K\vert=1$, then since $\PP^M(S_k(W\circ \Theta)\in A_n\text{ for some }k\in K)=0$, we get that the expression in the preceding display is equal to $F$ a.s.-$\PP^M$; thus the projection onto the first superchaos does indeed leave $F$ invariant. On the other hand, if $\vert K\vert>1$ the second mean of the expression in the preceding display is 
\begin{equation*}\leq \Vert f\Vert_\infty \sum_{i\in [k_n]}\WW(S_k\in (t_{i-1}^n,t_i^n)\text{ for all $k\in K$})\downarrow \Vert f\Vert_\infty\WW(\text{the $S_k$, $k\in K$, are all equal})=0\text{ as $n\to\infty$};\end{equation*}
whence we conclude that the projection of $F$ onto the first superchaos is zero in this case.

Returning now to the identification of the stable part, it is easy to check that the first chaos of the factorization associated to the Wiener process $W\circ\Theta^M$ --- denote this subfactorization of $N^M$ by $N^{M}\vert_W$ --- is included in the first chaos of $N^M$. It is then further straightforward to establish that the higher classical chaoses of $N^M\vert_W$ are  included in the respective higher classical chaoses of $N^M$. Suppose per absurdum that the first chaos of $N^M$ did not coincide with the first chaos of $N^M\vert_W$. Since the classical chaoses are orthogonal, by the preceding it would mean that there is a non-zero element $G$ of the first chaos of $N^M$, which is orthogonal to $\L2(\PP^M\vert_{\sigma(W\circ\Theta^M)})$. Let  $K\in (2^\mathbb{N})_{\mathrm{fin}}$, $\vert K \vert\geq 2$ and $f\in \L2(\WW)$ be bounded. Let further $-\infty=t_0^n<\cdots<t_{k_n}^n=\infty$, $k_n\in \mathbb{N}$, with $\{t_l^n:l\in [k_n-1]\}\uparrow$ to a dense subset of $\mathbb{R}$ as $n\in\mathbb{N}$ increases to $\infty$. Then the second mean of the projection of $F:=f(W\circ\Theta^M)\prod_{k\in K}(P_k\circ \Theta^M)$ onto the first chaos is \cite[Theorem~6.3]{picard2004lectures}
\begin{equation*}\lim_{n\to\infty}\sum_{i\in [k_n]}\PP^M\left[\left\vert\PP^M\left[F\Big\vert N_{t_{i-1}^n,t_i^n}^M\right]\right\vert^2\right]\leq \lim_{n\to\infty}\sum_{i\in [k_n]}\PP^M\left[\left\vert\PP^M\left[F\Big\vert N_{t_{i-1}^n,t_i^n}^M\lor\sigma(W\circ\Theta^M)\right]\right\vert^2\right]=0,\end{equation*} where the inequality is because conditional expectations are $\L2$-contractions and the equality follows directly from the superchaos computation above. By a density argument the first chaos of $N^M$, therefore $G$, is orthogonal to all $f\in \L2(\PP^M)$ for which $f_K$ is non-zero only for $\vert K\vert\geq 2$, $K\in (2^\mathbb{N})_{\mathrm{fin}}$. Using Proposition~\ref{lemma:general-f}  it follows that $G=\sum_{k\in \mathbb{N}}g_k(W\circ \Theta^M)(P_k\circ\Theta^M)$ for some $g_k\in \L2(\WW)$, $k\in \mathbb{N}$, satisfying $\sum_{k\in \mathbb{N}}\WW[\vert g_k\vert^2]<\infty$. $G$ being an additive integral of $N^M$, for any $t\in \mathbb{R}$ we must have, a.s.-$\PP^M$,
\begin{align*}
&\sum_{k\in \mathbb{N}}g_k(W\circ \Theta^M)(P_k\circ\Theta^M)\mathbbm{1}_{\{S_k(W\circ \Theta^M)\in (-\infty,t)\}}\\
&\quad +\sum_{k\in \mathbb{N}}g_k(W\circ \Theta^M)(P_k\circ\Theta^M)\mathbbm{1}_{\{S_k(W\circ \Theta^M)\in (t,\infty)\}}\\
&=\sum_{k\in \mathbb{N}}g_k(W\circ \Theta^M)(P_k\circ\Theta^M)=G=\PP^M[G\vert N^M_{-\infty,t}]+\PP^M[G\vert N^M_{t,\infty}]\\
&=\PP^M\left[\sum_{k\in \mathbb{N}}g_k(W\circ \Theta^M)(P_k\circ\Theta^M)\mathbbm{1}_{\{S_k(W\circ \Theta^M)\in (-\infty,t)\}}\vert N^M_{-\infty,t}\right]\\
&\quad+\PP^M\left[\sum_{k\in \mathbb{N}}g_k(W\circ \Theta^M)(P_k\circ\Theta^M)\mathbbm{1}_{\{S_k(W\circ \Theta^M)\in (t,\infty)\}}\vert N^M_{t,\infty}\right].
\end{align*} 
Taking the second mean of this equality and using the fact that conditional expectations are $\L2$-contractions shows that $X_t:=\sum_{k\in \mathbb{N}}g_k(W\circ \Theta^M)(P_k\circ\Theta^M)\mathbbm{1}_{\{S_k(W\circ \Theta^M)\in (-\infty,t)\}}$ must in fact be $N^M_{-\infty,t}$-measurable and $\sum_{k\in \mathbb{N}}g_k(W\circ \Theta^M)(P_k\circ\Theta^M)\mathbbm{1}_{\{S_k(W\circ \Theta^M)\in (t,\infty)\}}$ must in fact be $N^M_{t,\infty}$-measurable. We infer that 
\begin{equation*}X_t=\PP^M[X_t\vert N^M_{-\infty,t}\lor \sigma(P\circ \Theta^M)]=\sum_{k\in \mathbb{N}}(P_k\circ\Theta^M)\PP^M\left[g_k(W\circ \Theta^M)\mathbbm{1}_{\{S_k(W\circ \Theta^M)\in (-\infty,t)\}}\vert (N^M\vert_W)_{-\infty,t}\right]\end{equation*}
a.s.-$\PP^M$. Comparing this expression for $X_t$ with the one defining $X_t$ it follows by orthogonality that for all $k\in \mathbb{N}$, $g_k\mathbbm{1}_{\{S_k\in (-\infty,t)\}}$ is $\FF_{-\infty,t}$-measurable. Similarly we deduce that $g_k\mathbbm{1}_{\{S_k\in (t,\infty)\}}$ is $\FF_{t,\infty}$-measurable. Therefore 
\begin{equation*}\WW[g_k\vert\FF_{-\infty,t}]=g_k\mathbbm{1}_{\{S_k\in (-\infty,t)\}}+\WW[g_k;S_k>t]\end{equation*} a.s.-$\WW$, on using $\WW(S_k=t)=0$. We are therefore witnessing in $(g_k\mathbbm{1}_{\{S_k\in (-\infty,t)\}}+\WW[g_k;S_k>t])_{t\in \mathbb{R}}$ a discontinuous $\L2$-bounded right-continuous martingale in the Brownian filtration $(\FF_{-\infty,t})_{t\in \mathbb{R}}$, which is  a contradiction, unless $g_l=0$ a.s.-$\WW$ for all $l\in \mathbb{N}$, but the latter yields $G=0$ a.s.-$\PP^M$, which in itself is in contradiction with what we have assumed. (If the reader feels we have been a little cavalier about calling the two-sided filtration $(\FF_{-\infty,t})_{t\in \mathbb{R}}$ Brownian, he/she would be correct. But anyway, for all $a\in \mathbb{R}$, the filtration $(\FF_{a+t})_{t\in [0,\infty)}$ is generated by the Brownian motion $\Delta_a\vert_{[0,\infty)}$ and the independent $\sigma$-field $\FF_{-\infty,a}$ (initial enlargement), so the argument can be done ``locally'' on $[a,\infty)$ for an $a$ such that $\WW(S_k>a)>0$ (and such $a$ exists) to get $g_k=0$ a.s.-$\WW$, no matter what the choice of $k\in \mathbb{N}$.)
\end{proof}
Recall Proposition~\ref{lemma:general-f}. For $f\in H_{\mathrm{sens},1}^M$ write $f_k:=f_{\{k\}}$, $k\in \mathbb{N}$, and introduce the finite  measure $\mu_f^M$ on $\mathbb{R}\times \Omega_0$,
\begin{equation}\label{equation:conditional-spectral-measure}
\mu_f^M(E):=\sum_{k\in \mathbb{N}}\WW[\vert f_k\vert^2\mathbbm{1}_E(S_k,B)],\quad E\in \mathcal{B}_{\mathbb{R}}\otimes \GG.\end{equation}  Notice that taking the $f_k$, $k\in \mathbb{N}$, such that they are all a.s.-$\WW$ non-zero (a choice that can be made) gives a measure $\mu_f^M$ that has, besides being carried by $\llbracket M\rrbracket $ (which is true always), the following property: $\mu_f^M(\llbracket M\rrbracket\backslash \llbracket N\rrbracket)>0$ for any random countable set for which $M\backslash N$ is not empty --- we will say that $\mu_f^M$ has full support. (Here $\llbracket N\rrbracket:=\{(t,\omega)\in \mathbb{R}\times \Omega_0:t\in  N(\omega)\}$ is the graph of a random countable set $N$ on $\mathbb{R}\times\Omega_0$. The reader will forgive us the notational shenanigan vis-\`a-vis the same notation of Remark~\ref{remark:ranomd-countable-set} albeit with the order of time and of the sample space in the product reversed.) 
The following result will be used up in the next subsection.

\begin{proposition}\label{lemma:identify-spectral}
Let $f\in H_{\mathrm{sens},1}^M$. For $A\in\GG$ and extended-real $s<t$ we identify
\begin{equation}\label{eq:conditional-spectral-measure}
\mu_f^M((s,t)\times A)=\PP^M[\vert \PP^M[f\vert N_{s,t}^M\lor \sigma(W\circ \Theta^M)]\vert^2;W\circ \Theta^M\in A].
\end{equation}
In particular $\mu_f^M$ does not depend on the choice of the enumeration $S$.
\end{proposition}
We may thus think of $\mu_f^M$ as a kind of spectral measure of $f$ \emph{conditioned} on the Brownian motion $W\circ \Theta^M$.
\begin{proof}
 We have seen in the proof of Proposition~\ref{lemma:identifications-stable-super} 
that 
\begin{equation*}
\mathbb{Q}\left[P_k\Big\vert W,P_{s,t}\right]=P_k\mathbbm{1}_{\{S_k(W)\in (s,t)\}}\text{ a.s.-$\QQ$},\quad k\in \mathbb{N}.
\end{equation*}
Then we can compute, using Proposition~\ref{lemma:identifications-stable-super} in the first equality:
\begin{align*}
&\PP^M[\vert\PP^M[f\vert N_{s,t}^M\lor \sigma(W\circ \Theta^M)]\vert^2;W\circ \Theta^M\in A]\\
&=\PP^M\left[\left\vert\PP^M\left[\sum_{k\in \mathbb{N}}f_k(W\circ \Theta^M)(P_k\circ \Theta^M)\Big\vert N_{s,t}^M \lor \sigma(W\circ \Theta^M)\right]\right\vert^2;W\circ\Theta^M\in A\right]\\
&=\mathbb{Q}\left[\left\vert\mathbb{Q}\left[\sum_{k\in \mathbb{N}}f_k(W)P_k\Big\vert W,P_{s,t}\right]\right\vert^2;W\in A\right]\\
&=\mathbb{Q}\left[\left\vert\sum_{k\in \mathbb{N}}f_k(W)\mathbbm{1}_A(W)\mathbb{Q}\left[P_k\Big\vert W,P_{s,t}\right]\right\vert^2\right]\\
&=\mathbb{Q}\left[\left\vert\sum_{k\in \mathbb{N}}f_k(W)\mathbbm{1}_A(W)P_k\mathbbm{1}_{(s,t)}(S_k(W))\right\vert^2\right]\\
&=\sum_{k\in \mathbb{N}}\mathbb{Q}\left[\vert f_k(W)\vert^2\mathbbm{1}_A(W)\mathbbm{1}_{(s,t)}(S_k(W))P_k^2\right]\quad \text{(by orthogonality)}\\
&=\sum_{k\in \mathbb{N}}\WW\left[\vert f_k\vert^2\mathbbm{1}_{(s,t)\times A}(S_k,B)\right]\\
&=\mu_f^M((s,t)\times A).
\end{align*}
This establishes \eqref{eq:conditional-spectral-measure}. The final claim follows by Dynkin's lemma.
\end{proof}

\subsection{Non-isomorphic noises}
Recall that an isomorphism of two noises  $(N^1;\PP^1,T^1)$ and $(N^2;\PP^2,T^2)$   is a mod-$0$ isomorphism $\psi$ between the probabilities $\PP^1$ and $\PP^2$ which sends $N^1_{s,t}$ onto $N^2_{s,t}$ for all extended-real $s<t$ and intertwines the shifts in the sense that $\psi\circ T_h^1=T_h^2\circ \psi$ a.s.-$\PP^1$ for all $h\in \mathbb{R}$ \cite[Definitions~4a1 and~4a3]{tsirelson-nonclassical}. Let $(M^1;\Omega^1_1,S^1_1)$ and $(M^2;\Omega_1^2,S^2_2)$ be two triplets such as $(M;\Omega_1,S)$ --- two stationary local random countable sets, enumerated and spruced according to Subsection~\ref{subsection:construction} --- and put $(N^i; \PP^i,T^i):=(N^{M_i};\PP^{M_i},T^{M_i})$, $i\in \{1,2\}$, and similarly for the other pieces of notation (for instance, $\Omega^i=\Omega^{M_i}$, $i\in \{1,2\}$).  We set also \begin{equation*}
\mathsf{s}^M:=(\Omega^M\ni (\omega_1,\eta)\mapsto (-\omega_1,\eta)\in \Omega^{M(-B)})\end{equation*} whenever we have the same $\Omega_1$ for $M(-B)$ in introducing $\Omega^{M(-B)}$ as we do for $M$ (implicitly below we shall take this to be the case, referring to Remarks~\ref{remark:same-omega1} and~\ref{rmk-iso-noise} as necessary) -- the sign change of the Brownian component.

\begin{theorem}\label{thm:not-iso-noises} 
The following are equivalent.
\begin{enumerate}[(A)]
\item\label{thm:not-iso-noises:i}  $M^1\triangle M^2$ and $M^1\triangle M^2(-B)$ are not empty.
\item\label{thm:not-iso-noises:ii} The noises $(N^1;\PP^1,T^1)$ and $(N^2;\PP^2,T^2)$ are not isomorphic.
\end{enumerate}
\end{theorem}
\begin{proof}
By Remarks~\ref{remark:same-omega1} and~\ref{rmk-iso-noise} we may and do assume $\Omega_1^1=\Omega_1^2$ and write it as just $\Omega_1$ (but $\Omega^1\ne\Omega^2$ unless $M^1=M^2$ on $\Omega_1$).

If $(N^+;\PP^+,T^+)$ is associated to $M$ and $(N^-;\PP^-,T^-)$ to $M(-B)$, then these two noises are isomorphic, the isomorphism carrying the first onto the second being $\mathsf{s}^M$. Thus we see easily that \ref{thm:not-iso-noises:ii} implies \ref{thm:not-iso-noises:i}. 

For the converse implication assume \ref{thm:not-iso-noises:i} and suppose per absurdum $\psi$ is an isomorphisms between $(N^1;\PP^1,T^1)$ and $(N^2;\PP^2,T^2)$. It carries the stable part of $N^1$ onto the stable part of $N^2$ so restricts to an isomorphism of the white noises associated to $W\circ \Theta^1$ and $W\circ \Theta^2$ under $(\PP^1,T^1)$ and $(\PP^2,T^2)$, respectively. It is well-known \cite[p.~184]{tsirelson-nonclassical} that there are only two such isomorphisms; the identity, and the sign change. In other words, either ($+$) $W\circ \Theta^2\circ \psi=W\circ\Theta^1$ a.s.-$\PP^1$  or ($-$) $W\circ \Theta^2\circ \psi=-W\circ\Theta^1$ a.s.-$\PP^1$. 

Suppose ($+$) and $M^2\backslash M^1$ is not empty in the first instance.  Take any $f\in H^2_{\mathrm{sens},1}$ such that $\mu_f^2$ has full support.  We get $f\circ \psi\in H^1_{\mathrm{sens},1}$. Also, from Proposition~\ref{lemma:identify-spectral}, for all $A\in\GG$ and extended-real $s<t$,
\begin{align*}
\mu^2_f((s,t)\times A)&=\PP^2[\vert\PP^2[f\vert N_{s,t}^2\lor \sigma(W\circ \Theta^2)]\vert^2;W\circ \Theta^2\in A]\\
&=\PP^1[\vert\PP^1[f\circ\psi\vert N_{s,t}^1\lor \sigma(W\circ \Theta^2\circ\psi)]\vert^2;W\circ \Theta^2\circ \psi\in A],\\
&=\PP^1[\vert\PP^1[f\circ\psi\vert N_{s,t}^1\lor \sigma(W\circ \Theta^1)]\vert^2;W\circ \Theta^1\in A]\\
&=\mu^1_{f\circ \psi}((s,t)\times A).
\end{align*}
By Dynkin's lemma we infer that $\mu_f:=\mu^2_f=\mu^1_{f\circ \psi}$. Then  $\mu_f(\llbracket M^2\rrbracket\backslash \llbracket M^1\rrbracket)=\mu_f^2(\llbracket M^2\rrbracket\backslash \llbracket M^1\rrbracket)>0$, but also $\mu_f(\llbracket M^2\rrbracket\backslash \llbracket M^1\rrbracket)=\mu_{f\circ\psi}^1(\llbracket M^2\rrbracket\backslash \llbracket M^1\rrbracket)=0$, a contradiction. 

If ($+$) still holds but instead $M^1\backslash M^2$ is not empty, apply the argument with the roles of $M^1$ and $M^2$ interchanged, $\psi^{-1}$ playing the role of $\psi$.

In case of ($-$) apply the preceding to $M^2(-B)$ in lieu of $M^2$, noticing that the noises associated to these two random sets are isomorphic, the isomorphism being the sign change $\mathsf{s}^2$  (thus the role of $\psi$ is then played by $\mathsf{s}^2\circ \psi$).
\end{proof}
\begin{remark}\label{rmk:non-iso-extensions}
The same proof shows that if merely $M^1\triangle M^2$ is not empty, then $(N^1;\PP^1,T^1)$ and $(N^2;\PP^2,T^2)$ are non-isomorphic extensions of the Wiener noise, i.e. not isomorphic as noises where in addition we insist that $W\circ\Theta^1$ is sent to $W\circ \Theta^2$ by the isomorphism.  
\end{remark}
\begin{example}
Let $(N^{\mathrm{loc}\min};\PP^{\mathrm{loc}\min},T^{\mathrm{loc}\min})$ be associated to $M=\{\text{local minima of $B$}\}$ and $(N^{\mathrm{loc}\max};\PP^{\mathrm{loc}\max},T^{\mathrm{loc}\max})$ be associated to  $M=\{\text{local maxima of $B$}\}$. These two noises are isomorphic. 
\end{example}
\begin{example}
Let $(N^{\mathrm{loc}\,\mathrm{ext}};\PP^{\mathrm{loc}\,\mathrm{ext}},T^{\mathrm{loc}\,\mathrm{ext}})$ be associated to $M=\{\text{local extrema of $B$}\}$. Theorem~\ref{thm:not-iso-noises} implies that $(N^{\mathrm{loc}\min};\PP^{\mathrm{loc}\min},T^{\mathrm{loc}\min})$  and $(N^{\mathrm{loc}\, \mathrm{ext}};\PP^{\mathrm{loc}\,\mathrm{ext}},T^{\mathrm{loc}\,\mathrm{ext}})$ are not isomorphic noises. 
\end{example}


%
Returning to the random countable sets got from the squared Bessel processes, at least for dimensions from $[1,2)$ we have
\begin{proposition}\label{proposition:disjoint-negative}
Let $\{d_1,d_2\}\subset [1,2)$. Then  $M^{(d_1)}\cap M^{(d_2)}(-B)$ is empty. 
\end{proposition}
\begin{proof}
Assume $\{d_1,d_2\}\subset (1,2)$ at first. According to \cite[Exercise~XI.1.26]{revuz} if $Z=Z^{(d_i)}$  solves \eqref{sde} for $d=d_i$, $i\in \{1,2\}$, then $\WW$-a.s.
\begin{equation*}\sqrt{Z_t^{(d_i)}}=B_t+\frac{d_i-1}{2}\int_0^t\frac{\dd s}{\sqrt{Z_s^{(d_i)}}},\quad t\in [0,\infty).\end{equation*}
Taking the sum yields
\begin{equation*}\sqrt{Z_t^{(d_1)}}+\sqrt{Z^{(d_2)}(-B)_t}=\frac{d_1-1}{2}\int_0^t\frac{\dd s}{\sqrt{Z_s^{(d_1)}}}+\frac{d_2-1}{2}\int_0^t\frac{\dd s}{\sqrt{Z^{(d_2)}(-B)_s}},\quad t\in [0,\infty),\end{equation*}
a.s.-$\WW$, so that $\left(\sqrt{Z_t^{(d_1)}}+\sqrt{Z^{(d_2)}(-B)_t}\right)_{t\in (0,\infty)}$ is a.s.-$\WW$ strictly positive, hence at most one of $Z^{(d_1)}$ and $Z^{(d_2)}(-B)$ is zero at any given point in time from   $(0,\infty)$. We deduce that $g^{(d_1)}_{s,t}\ne g^{(d_2)}_{s,t}(-B)$ with $\WW$-probability one for all real $s< t$. Due to Lemma~\ref{lemma:gamma}\ref{lemma:gamma:iii} it renders  $M^{(d_1)}\cap M^{(d_2)}(-B)=\emptyset$ a.s.-$\WW$, as required.

If, say, $d_1=1$, then  \cite[Exercise~XI.1.26]{revuz} 
\begin{equation*}\sqrt{Z_t^{(d_1)}}=B_t+\frac{1}{2}L^0_t,\quad t\in [0,\infty),\end{equation*} where $L^0_t$ is the local time of $\sqrt{Z^{(d_1)}}$ at zero; similarly for $d_2=1$. The remainder of the argument to handle the case when possibly $d_1=1$ or $d_2=1$ is essentially verbatim the same.
\end{proof}
For $d\in (0,2)$ associate $(N^{(d)};\PP^{(d)},T^{(d)})$ to $M^{(d)}$. 
\begin{corollary}
The noises  $(N^{(d)};\PP^{(d)},T^{(d)})$, $d\in [1,2)$, are pairwise non-isomorphic.
\end{corollary}
\begin{proof}
Combine Propositions~\ref{proposition:a.s.distinct} and~\ref{proposition:disjoint-negative} with Theorem~\ref{thm:not-iso-noises}.
\end{proof}
We do not attempt here the decomposition into isomorphism classes of the entire family $(N^{(d)};\PP^{(d)},T^{(d)})$, $d\in (0,2)$. Though, the latter are certainly  pairwise non-isomorphic extensions of the Wiener noise in the sense of (and as a consequence of) Remark~\ref{rmk:non-iso-extensions}.

\subsection{Subnoises}
Under a subnoise of a noise $(N;\PP,T)$ we shall understand a family $N'=(N'_{s,t})_{(s,t)\in [-\infty,\infty]^2,s<t}$ of $\PP$-complete sub-$\sigma$-fields of $N_{-\infty,\infty}$ having the factorizability property relative to $\PP\vert_{N'_{-\infty,\infty}}$ such that $N'_{s,t}=N_{s,t}\cap N_{-\infty,\infty}'$ for all extended-real $s<t$ and such that $T_h\in N'_{-\infty,\infty}/N'_{-\infty,\infty}$ for all $h\in \mathbb{R}$ (hence, actually,  $T_h\in N_{s+h,t+h}'/N_{s,t}'$ for all extended-real $s<t$ and all $h\in \mathbb{R}$). Passing, in the obvious way, to the associated standard quotient probability $\PP'/N_{-\infty,\infty}'$ (note: the probability $\PP'\vert_{N_{-\infty,\infty}'}$ is not standard unless $N_{-\infty,\infty}'=N_{-\infty,\infty}$), the associated quotient sub-$\sigma$-fields $N'/N_{-\infty,\infty}'$ and the associated  group of mod-$0$ isomorphisms $T/N_{-\infty,\infty}'$  we get indeed a (one-dimensional) noise, which justifies the nomenclature.  It is important to note that for a subnoise $N'$, as a consequence of the factorizability property, the $\sigma$-field $N'_{-\infty,\infty}$ (which, incidentally, wholly determines $N'$) commutes with all of the $\sigma$-fields of $N$ in the sense that $\PP[\PP[\cdot\vert N'_{-\infty,\infty}]\vert N_{s,t}]=\PP[\PP[\cdot\vert  N_{s,t}]\vert N'_{-\infty,\infty}]$ (and hence $=\PP[\cdot\vert N'_{s,t}]$) on $\L2(\PP)$ for all extended-real $s<t$: by a monotone class argument it suffices indeed to check that $\PP[ff'\vert N'_{-\infty,\infty}\vert N_{s,t}]=\PP[ff'\vert  N_{s,t}\vert N'_{-\infty,\infty}]$ for $f\in \L2(\PP\vert_{N_{s,t}})$ and $f'\in \L2(\PP\vert_{N_{-\infty,s}\lor N_{t,\infty}})$, in which case it follows from the following computation, in which we write $y:=N'_{-\infty,\infty}$, $x:=N_{s,t}$ and $x':=N_{-\infty,s}\lor N_{t,\infty}$ for short,
\begin{align*}
\PP[\PP[ff'\vert x]\vert y]&=\PP[f\vert y]\PP[f']=\PP[f\vert (y\cap x)\lor (y\cap x')]\PP[f']=\PP[f\vert y\cap x]\PP[f']\\
&=\PP[\PP[f\vert y\cap x]\PP[f'\vert y\cap x']\vert x]=\PP[\PP[ff'\vert (y\cap x)\lor (y\cap x')]\vert x]=\PP[\PP[ff'\vert y]\vert x].
\end{align*} We recall also that any noise $(N;\PP,T)$ admits a largest classical subnoise $N^{\mathrm{lin}}$, called the linear or classical part of $(N;\PP,T)$, whose underlying $\sigma$-field $N^{\mathrm{lin}}_{-\infty,\infty}$ is connected to the stable part $H_{\mathrm{stb}}$ of $\L2(\PP)$ by the relation $H_{\mathrm{stb}}=\L2(\PP\vert_{N^{\mathrm{lin}}_{-\infty,\infty}})$. We will use these facts below without special mention.

A stationary local random countable set $M'$ that is contained in $M$ a.s.-$\WW$ leads to a subnoise of $(N^M;\PP^M,T^M)$ as follows. For extended-real $s<t$ let $S_{s,t}'$ be a measurable enumeration of $M'\cap (s,t)$ and define $N^{M'}_{s,t}$ as the $\PP^M$-complete $\sigma$-field generated by the increments of $W\circ \Theta^M$ on $(s,t)$ and by the random signs 
\begin{equation*}
\left((W\circ \Theta^M)^{-1}(\{S_{s,t}'(k)\in M\})\ni (\omega_1,\eta)\mapsto \eta(S_{s,t}'(k)(\omega_1))\right),\quad k\in \mathbb{N}.\end{equation*} Then $N^{M'}:=(N^{M'}_{s,t})_{(s,t)\in \mathbb{R}^2,s<t}$ is a subnoise of $(N^M;\PP^M,T^M)$  (there is no dependence of $N^{M'}$ on the actual enumerations used).  The notation $N^{M'}$ is in conflict with $N^{M'}$ under $\PP^{M'}$, but we will suffer it, since we will only work with $(N^M;\PP^M,T^M)$ in this subsection, no $(N^{M'};\PP^{M'},T^{M'})$ shall appear. Two special cases are worth pointing out: when $M'$ is empty, then $N^{M'}$ is generated just by $W\circ \Theta^M$ and we get the Wiener subnoise, which is also the classical part of $(N^M;\PP^M,T^M)$ by Proposition~\ref{lemma:identifications-stable-super};  when $M\backslash M'$ is empty, then $N^{M'}=N^M$ and we have the ``full'' subnoise.


%

Actually all the non-void subnoises of $(N^M;\PP^M,T^M)$ are got in this way. (A subnoise is called void when all its $\sigma$-fields are $\PP$-trivial.)

\begin{theorem}\label{thm:subnoises}
Let $N'$ be a non-void subnoise of $(N^M;\PP^M,T^M)$. Then $N'$ is equal to $N^{M'}$ for some stationary local random countable set $M'$  that is contained in $M$.
\end{theorem}
\begin{proof}
The spectral measure type  \cite[bottom of p.~274]{tsirelson-nonclassical} of $N'$ is absolutely continuous w.r.t. that of $N^M$, indeed for each $f\in \L2(\PP\vert_{N'_{-\infty,\infty}})$ the  $N'$-spectral measure of $f$ \cite[p.~272]{tsirelson-nonclassical} is equal to the $N$-spectral measure of $f$. 
 The stable part of $N'$ is contained in the stable part of $N^M$. Since  $N^M$ has only superchaoses of finite order, the same is true of $N'$. If the stable part of $N'$ is void, then the first superchaos of $N'$ is void too, whence $N'$ is rendered void, which is a contradiction with the assumption. Therefore the stable part of $N'$ is non-void and since the only non-void subnoise of the one-dimensional Wiener noise is the Wiener noise  itself (Lemma~\ref{lemma:subnoise-of-wiener-noise} below) 
it follows that $N'$ and $N^M$ have the same stable part, the same maximal classical subnoise, which is the subnoise generated by $W\circ \Theta^M$. 
In symbols, $(N^M)^{\mathrm{lin}}_{-\infty,\infty}=(N')^{\mathrm{lin}}_{-\infty,\infty}=\overline{\sigma}(W\circ \Theta^M)$.

Now consider the commutative von Neumann algebra $\AA^M_{\mathrm{sens},1}$ of operators on the Hilbert space $H^M_{\mathrm{sens},1}$ generated by the  (commuting) family of the multiplications $\mathsf{M}_A$ with the indicators of $\{W\circ \Theta^M\in A\}$ for $A\in \GG$ (notice they leave $H^M_{\mathrm{sens},1}$ invariant) and of the conditional expectations $\mathsf{P}_{s,t}$ with respect to the sub-$\sigma$-fields $\sigma(W\circ\Theta^M)\lor N^M_{s,t}$ for extended-real $s< t$ (they also leave $H^M_{\mathrm{sens},1}$ invariant). Pick an $f\in H^M_{\mathrm{sens},1}$ such that $f_k\ne 0$ a.s.-$\WW$ for all $k\in \mathbb{N}$ (so that $\mu_f^M$ has full support) and set $\mu:=\mu_f^M$ (recall \eqref{equation:conditional-spectral-measure} for the notation). For $g\in H^M_{\mathrm{sens},1}$ define $\Psi(g)\in (\mathcal{B}_{\mathbb{R}}\otimes\GG)/\mathcal{B}_\mathbb{C}$ by putting
\begin{equation*}
\Psi(g)((t,\omega)):=\sum_{k\in \mathbbm{N}}\mathbbm{1}_{\{S_k(\omega)\}}(t)\frac{g_k(\omega)}{f_k(\omega)}\mathbbm{1}_{\mathbb{C}\backslash \{0\}}(f_k(\omega)),\quad (t,\omega)\in \mathbb{R}\times \Omega_0.
\end{equation*}
We compute
\begin{equation*}\mu[ \overline{\Psi(g)}\Psi(h)]=\sum_{k\in \mathbb{N}}\WW\left[\vert f_k\vert^2\frac{\overline{g_k}h_k}{\vert f_k\vert^2}; f_k\ne 0\right]=\sum_{k\in \mathbb{N}}\WW\left[\overline{g_k}h_k\right]=\PP^M[\overline{g}h],\quad \{g,h\}\subset H^M_{\mathrm{sens},1};
\end{equation*} therefore $\Psi:H^M_{\mathrm{sens},1}\to \L2(\mu)$ is a linear isometry. Since it is also surjective it is indeed a unitary isomorphism. Furthermore, (*) for $A\in \GG$ and extended-real $s<t$ the product $\mathsf{M}_A\mathsf{P}_{s,t}$ corresponds, via said unitary isomorphism, to multiplication with $(s,t)\times A$ (as we have seen in the proof of Proposition~\ref{lemma:identifications-stable-super}). We deduce that $\Psi$ is a spectral resolution of  $\AA^M_{\mathrm{sens},1}$ satisfying (*), which will be a property (variants of) which we shall refer to later on.

Restricting to $N'$ we have the commutative von Neumann algebra $\AA'_{\mathrm{sens},1}$ of operators on the first superchaos $H'_{\mathrm{sens},1}$ of $N'$ introduced in the same way as  $\AA^M_{\mathrm{sens},1}$ above except that $N'$ replaces $N^M$; one checks that $\AA'_{\mathrm{sens},1}=\{A\vert_{H'_{\mathrm{sens},1}}:A\in \AA^M_{\mathrm{sens},1}\}$. Because $H'_{\mathrm{sens},1}\subset H_{\mathrm{sens},1}^M$ is invariant under the action of the operators from $\AA^M_{\mathrm{sens},1}$ 
we may infer\footnote{In general, if, for a $\sigma$-finite measure $\nu$, $H$ is a closed linear subspace of $\L2(\nu)$  that is closed under the action of the multiplication maps with  elements of $\mathrm{L}^\infty(\nu)$ (or, equivalently, just with indicators of measurable sets belonging to a generating $\pi$-system), then $H=\L2(\mathbbm{1}_E\cdot \nu)$ for some measurable set $E$, up to the canonical inclusion of $\L2(\mathbbm{1}_E\cdot \nu)$ into $\L2(\nu)$. To see it, consider the orthogonal projection $\pr_H$ onto $H$ and note that it commutes with the multiplication operators corresponding to elements of $\mathrm{L}^\infty(\nu)$. It is well-known that this renders $\mathrm{pr}_H$ itself a multiplication operator by an element of $\mathrm{L}^\infty(\nu)$; being a projection, it must be the operator of multiplication with some measurable $E$.} that $\Psi(H'_{\mathrm{sens},1})=\L2(\mathbbm{1}_Q\cdot \mu)$ for some $Q\in  \mathcal{B}_{\mathbb{R}}\otimes\GG$, $Q\subset \llbracket M\rrbracket$. Then  $\Psi\vert_{H'_{\mathrm{sens},1}}:H'_{\mathrm{sens},1}\to \L2(\mathbbm{1}_Q\cdot\mu)$ is a spectral resolution of $\AA'_{\mathrm{sens},1}$ satisfying (*'), i.e. (*) with --- in the obvious notation ---  $\mathsf{M}_A'\mathsf{P}_{s,t}'=\mathsf{M}_A\vert_{H'_{\mathrm{sens},1}}\mathsf{P}_{s,t}\vert_{H'_{\mathrm{sens},1}}$ in lieu of  $\mathsf{M}_A\mathsf{P}_{s,t}$.

Now we argue that $Q$ is actually the graph of the stationary local random countable set $M'$ given by 
\begin{equation*}
M'(\omega):=\{t\in \mathbb{R}:(t,\omega)\in Q\},\quad \omega\in\Omega_0.
\end{equation*}

First, for $k\in \mathbb{N}$ and $\omega\in \Omega_0$, define $S'_k(\omega)$ as $S_k(\omega)$ or $\dagger$ according as $(S_k(\omega),\omega)\in Q$ or not. We get a measurable enumeration $(S'_k)_{k\in \mathbb{N}}$ of $M'$.  Thus $M'$ is a random countable set. 

Second, we show that $M'$ is local. For ease of notation focus on establishing that $M'\cap (-\infty,0)$ admits an $\FF_{-\infty,0}$-measurable enumeration (and, in passing,  $M'\cap (0,\infty)$ an $\FF_{0,\infty}$-measurable enumeration), the general case is not fundamentally more difficult. Denote $\Omega_{0,-}:=\Omega_0\vert_{(-\infty,0]}$ with Wiener measure $\PPP_-$ defined on the $\sigma$-field  $\GG_-$; analogously introduce $\Omega_{0,+}$ and $\PPP_+$ defined on $\GG_+$ (which are then just $\Theta$ and $\PPP$ on $\HH$, respectively). 
Repeating the story above with the commutative von Neumann algebras --- which, note, does not require the intervention of the measure-preserving group, only the factorization! --- but corresponding to the  ``restrictions'' of the factorizations $N^M$ and $N'$ to the temporal interval $(-\infty,0)$, we get  $Q_{-}\in  \mathcal{B}_{(-\infty,0)}\otimes \GG_-$, an $\FF_{-\infty,0}$-measurable enumeration $(S_{-}'(l))_{l\in \mathbb{N}}$ of the random countable set $M'_{-}$,
 \begin{equation*}M_-'(\omega):=\{t\in \mathbb{R}:(t,\omega\vert_{(-\infty,0]})\in Q_-\},\quad \omega\in\Omega_0,\end{equation*}
and a measure $\mu_{-}$ on $(-\infty,0)\times \Omega_{0-}$ together with spectral resolutions $\Psi_{-}:H_{\mathrm{sens},1-}^M\to \L2(\mu_{-})$ of $\AA_{\mathrm{sens},1-}^M$ and  $\Psi_{-}\vert_{H'_{\mathrm{sens},1-}}:H'_{\mathrm{sens},1-}\to \L2(\mathbbm{1}_{Q_{-}}\cdot\mu_{-})$ of $\AA'_{\mathrm{sens},1-}$; likewise with $+$ in lieu of $-$. Moreover, up to a canonical unitary isomorphism $H_{\mathrm{sens},1}^M=\left(H_{\mathrm{sens},1-}^M\otimes \L2(\PPP_+)\right)\oplus \left(\L2(\PPP_-)\otimes H_{\mathrm{sens},1+}^M\right)$. On the other hand we have the obvious identification   $\iota$ of  $(((-\infty,0)\times \Omega_{0-})\times \Omega_{0+})\sqcup (\Omega_{0-}\times ((0,\infty)\times \Omega_{0+}))$ with $\mathbbm{R}\times \Omega_0$  and thus we have also  the identification of $(\mu_-\times \PPP_+)\oplus (\PPP_-\times \mu_+)$ with the $\iota$-push-forward of this measure to the measure $\tilde \mu:=\iota_\star\mu$ on $\mathbb{R}\times \Omega_0$; the same for $((\mathbbm{1}_{Q_-}\cdot \mu_-)\times \PPP_+)\oplus (\PPP_-\times (\mathbbm{1}_{Q_+}\cdot \mu_+))$, which sliding along $\iota$ gives the measure $\mathbbm{1}_{\tilde Q}\cdot \tilde \mu$ with $\tilde Q:=\{(t,\omega)\in \mathbb{R}\times\Omega_0:(t,\omega\vert_{(-\infty,0]})\in Q_-\text{ or }(t,\omega\vert_{[0,\infty)})\in Q_+\}$. Thus, up to these natural identifications, $\tilde\Psi:=\left(\Psi_-\otimes \mathrm{id}_{\L2(\PPP_+)}\right)\oplus \left( \mathrm{id}_{\L2(\PPP_-)}\otimes \Psi_+\right):H_{\mathrm{sens},1}^M\to \L2(\tilde\mu)$ becomes a spectral resolution of  $\AA^M_{\mathrm{sens},1}$ satisfying (*.), which restricts to the spectral resolution $\tilde\Psi\vert_{H_{\mathrm{sens},1}'}:H_{\mathrm{sens},1}'\to \L2(\mathbbm{1}_{\tilde Q}\cdot \tilde \mu)$ of $\AA'_{\mathrm{sens},1}$ satisfying (*'.), the dots in  (*.), (*'.) signifying  that perhaps (*), (*') hold only for $A$ of the form $A_-\cap A_+$ with $(A_-,A_+)\in \FF_{-\infty,0}\times \FF_{0,\infty}$. 
But spectral measures of spectral resolutions of commutative von Neumann algebras with a fixed spectral space and fixed spectral sets on a generating multiplicative system of projections (viz. properties (*.), (*'.)) are unique up to equivalence. This means that $\mu\sim\tilde\mu$ and  $\mathbbm{1}_{\tilde Q}\cdot \tilde \mu\sim \mathbbm{1}_{Q}\cdot \mu$. Therefore $M_-'\cup M_+'=M'$ a.s.-$\WW$, especially $M_-'=M'\cap (-\infty,0)$ a.s.-$\WW$ (and likewise for $M_+'$). The proof of locality is complete.

Third, we proceed to stationarity. Let $h\in \mathbb{R}$. The automorphism $T_h^M$ of the measure space $(\Omega^M,\BB^M,\PP^M)$ induces the unitary automorphism $\psi_h:=(f\mapsto f\circ T_h^M)$ of the Hilbert space $\L2(\PP^M)$, which leaves $H^M_{\mathrm{sens},1}$ invariant. In turn,  $\delta_h:=\psi_h\vert_{H^M_{\mathrm{sens},1}}$  induces the automorphism $\mathfrak{A}_h:=(X\mapsto \delta_h^{-1}X\delta_h)$ of the commutative von Neumann algebra $\AA^M_{\mathrm{sens},1}$. We have also the unitary isomorphism $\Psi\circ \delta_h:H^M_{\mathrm{sens},1}\to \L2(\mu)$, which amounts to a spectral resolution of $\AA^M_{\mathrm{sens},1}$ for which the spectral set associated to $\mathfrak{A}_h(\mathsf{M}_A\mathsf{P}_{s,t})=M_{\Delta_h(A)}\mathsf{P}_{s-h,t-h}$ is $(s,t)\times A$, this for all extended-real $s<t$ and $A\in \GG$. Define 
\begin{equation*}
\rho_h((t,\omega)):=(t-h,\Delta_h(\omega)),\quad (t,\omega)\in \mathbb{R}\times \Omega_0;
\end{equation*}
 then $\rho_h:\mathbb{R}\times \Omega_0\to \mathbb{R}\times \Omega_0$ is a $(\mathcal{B}_\mathbb{R}\otimes\GG)$-bi-measurable bijection,  put 
\begin{equation*}
\tilde \mu:=(\rho_h)_\star\mu
\end{equation*}
 and denote by $\theta_h:\L2(\mu)\to \L2(\tilde\mu)$ the unitary isomorphism $(g\mapsto g\circ \rho_h^{-1})$. We see that $\theta_h\circ \Psi\circ\delta_h:H^M_{\mathrm{sens},1}\to \L2(\tilde \mu)$ is a spectral resolution of  $\AA^M_{\mathrm{sens},1}$  satisfying (*) and deduce that $\tilde\mu\sim \mu$. Repeating the same exercise with the subnoise we get $\widetilde{\mathbbm{1}_Q\cdot \mu}\sim \mathbbm{1}_Q\cdot \mu$. It follows that $\rho_h(Q)=Q$ a.e.-$\mu$, i.e. $M'=-h+M'(\Delta_{-h})$ a.s.-$\WW$. This being true for all $h\in \mathbb{R}$ the  random countable set $M'$ is indeed stationary.

If $M'$ is empty or equal to $M$ a.s.-$\WW$, then clearly $N'=N^{\emptyset}$ or $N'=N^M$ accordingly, simply because in this case the first superchaos of $N'$ is null or coincides with that of $N^M$, respectively (there are no super-superchaoses, so the stable part and the first superchaos generate everything \cite[p.~71, penultimate paragraph]{picard2004lectures}). Therefore we may and do assume neither $M'$ nor $M'':=M\backslash M'$ is empty. We choose an a.s.-$\WW$ injective measurable enumeration $S'=(S'_k)_{k\in \mathbb{N}}$ of $M'$, likewise $S''$ for $M''$. Finally, we insist that the enumeration $S$ of $M$ has been got in such a way that $S_{2k}=S'_k$ and $S_{2(k-1)+1}=S''_{k}$ a.s.-$\WW$ for all $k\in \mathbb{N}$ to begin with (which too we may do). Looking at the first superchaos of $N'$  we now see that it is equal to $\{f\in \L2(\PP^M):f_K=0\text{ for all }K\in (2^\mathbb{N})_{\mathrm{fin}}\text{ of size }\ne 1,\, f_{2(k-1)+1}=0\text{ for all }k\in \mathbb{N}\}$, which is also the first superchaos of $N^{M'}$. Thus the stable part and the first superchaos of $N'$ and $N^{M'}$ are seen to coincide. But then again these two  subnoises are the same, which completes the proof.
\end{proof}
We owe the reader the next result, which must be folklore, but, it appears, is not so easy to pin  down in literature.
\begin{lemma}\label{lemma:subnoise-of-wiener-noise}
The Wiener noise $(\FF;\WW,\Delta)$ has only trivial subnoises: the void subnoise and itself.
\end{lemma}
\begin{proof}
Let $\FF'$ be a subnoise of the Wiener noise $(\FF;\WW,\Delta)$. Like the Wiener noise it is classical (a subnoise of a classical noise is classical). The spectral measure type  of $\FF'$ restricted to the part of the spectral space corresponding to the first chaos --- that we identify canonically with $\mathbb{R}$ --- call it $\mu$, is translation invariant, up to equivalence, by the temporal homogeneity of $\FF'$. We deduce \cite[Proposition~VII.1.11]{bourbaki2004integration} that $\mu=0$ or $\mu\sim \mathscr{L}$. (By the same token it follows that all non-void classical noises have the same spectral measure class \cite[Example~9b9]{tsirelson-nonclassical}, but we do not need it.) In the former case $\FF'$ is void, for a classical noise is generated by its first chaos. In the latter case there is an element $f$ of the first chaos of $\FF'$, whose $\FF'$-spectral measure is equivalent to Lebesgue measure on $\mathbb{R}\equiv {\mathbb{R}\choose 1}$. Such $f$, being also a member of the first chaos of $\FF$,  is of the form $g\cdot B$ (the stochastic integral in the Wiener sense is meant) for a $g\in \L2(\mathscr{L})$. This $g$ must necessarily  be non-zero a.e.-$\mathscr{L}$, for its spectral measure is equal to $\vert g\vert^2\cdot\mathscr{L}$ on $\mathbb{R}\equiv {\mathbb{R}\choose 1}$. Hence the random variables $\WW[f\vert \FF'_{s,t}]=\WW[f\vert \FF_{s,t}]=(g\mathbbm{1}_{(s,t)})\cdot B$, for $s<t$ real, which all belong to the first chaos of $\FF'$, generate the whole of the $\sigma$-field of $\WW$ up to trivial sets. But that can only be if $\FF'_{-\infty,\infty}=\GG=\FF_{-\infty,\infty}$, i.e. $\FF'=\FF$.
\end{proof}

\begin{corollary}
The following statements are equivalent.
\begin{enumerate}[(A)]
\item\label{minimal-and-subnoise:A} $M$ is minimal.
\item\label{minimal-and-subnoise:B} The only non-void proper subnoise of $(N^M;\PP^M,T^M)$ is the classical noise associated with $W\circ \Theta^M$. 
\end{enumerate}
\end{corollary}
\begin{proof}
By Theorem~\ref{thm:subnoises} condition \ref{minimal-and-subnoise:A} implies \ref{minimal-and-subnoise:B}. On the other hand,  if $M$ is not minimal, then we can easily construct a proper non-void non-classical subnoise of $(N^M;\PP^M,T^M)$ out of any proper dense local stationary subset of $M$ following the construction of the second paragraph of this subsection.
\end{proof}

\begin{example}
Thanks to Corollary~\ref{corollary:minimal-Md} $(N^{\mathrm{loc}\min};\PP^{\mathrm{loc}\min},T^{\mathrm{loc}\min})$ and $(N^{\mathrm{loc}\max};\PP^{\mathrm{loc}\max},T^{\mathrm{loc}\max})$ have as their only proper non-void subnoises the Wiener noise (and the same is true for each $(N^{(d)};\PP^{(d)},T^{(d)})$, $d\in (0,2)$). On the other hand these two noises are also both isomorphic to proper non-void subnoises of $(N^{\mathrm{loc}\,\mathrm{ext}};\PP^{\mathrm{loc}\,\mathrm{ext}},T^{\mathrm{loc}\,\mathrm{ext}})$.
\end{example}




\bibliographystyle{plain}
\bibliography{Biblio_noise}

\end{document}